%% file: main.tex
\documentclass[11pt]{article}

\usepackage{natbib}
\usepackage{hyperref}
\usepackage{amssymb, amsmath, url, amsthm, amssymb,color}
\usepackage[T1]{fontenc}
\usepackage[utf8]{inputenc}
\usepackage{graphicx}
\usepackage{caption}
\usepackage{subcaption}
\usepackage[ngerman, main=english]{babel}
\usepackage{tikz}
\usepackage[shortlabels]{enumitem}
\usepackage{soul}
\usepackage{mycommands}
\usepackage{orcidlink}

\newtheorem{Theorem}{Theorem}
\newtheorem{Lemma}{Lemma}
\newtheorem{Corollary}{Corollary}
\newtheorem{Remark}{Remark}
\newtheorem{Proposition}{Proposition}

\newtheorem{Assumption}{Assumption}

\def\appendixtype{internal}
\def\prooftype{internal}
\def\simulations{internal}
\def\blinded{false}

\setcitestyle{round}

\begin{document}

\title{Azadkia–Chatterjee's dependence coefficient\\for infinite dimensional data}
 \author{Siegfried H\"ormann\thanks{\href{mailto:shoermann@tugraz.at}{shoermann@tugraz.at, \orcidlink{0000-0001-5878-2840}}}\hspace{.2cm} and Daniel Strenger\thanks{\href{mailto:daniel.strenger@tugraz.at}{daniel.strenger@tugraz.at, \orcidlink{0009-0002-2465-604X}}}\\
  Institute of Statistics, Graz University of Technology}

\maketitle

\input{abstract}

\input{introduction}

\input{estimator}

\input{main-results}

\input{real-data}

\input{proofs}

\section*{Acknowledgements}
\input{acknowledgments}

\section*{Funding}
\input{funding}

\bibliography{references}

\begin{appendix}
\input{simulations}
\end{appendix}

\end{document}

%% file: abstract.tex
\begin{abstract}
\ We extend the scope of Azadkia–Chatterjee's dependence coefficient between a scalar response $Y$ and a multivariate covariate $X$ to the case where $X$ takes values in a general metric space. Particular attention is paid to the case where $X$ is a curve. Although extending this framework at the population level is relatively straightforward, analyzing the asymptotic behavior of the estimator proves to be complex. This complexity is largely related to the nearest neighbor structure of the infinite-dimensional covariate sample, leading us to explore a topic that has not been previously addressed in the literature. The primary contribution of this paper is to provide insights into this issue and propose strategies to address it. Our findings also have significant implications for other graph-based methods facing similar challenges.\end{abstract}

%% file: introduction.tex
\section{Introduction}

Assume that~$Y$ is a real random variable with distribution function~$F(t)$. 
For some covariate~$X$, let~$G_X(t)=P(Y\geq~t|X)$.
\cite{azadkia_simple_2021} have studied a dependence coefficient between~$Y$ and~$X$, which, for continuous~$F$, can be expressed as
 \begin{equation}
  T(X,Y):=6\times\int \mathrm{Var}(G_X(t))dF(t).
 \end{equation}
Note that if~$X$ and~$Y$ are independent, then~$G_X(t)=G(t)=P(Y\geq~t)$, which is non-random, and hence we have~$T(X,Y)=~0$. On the other hand, if~$Y=f(X)$ for some measurable function~$f$, then~${G_X(t)=1\{Y\geq~t\}}$ and thus it follows that $T(X,Y)=6\int F(t)(1-F(t))dF(t)=1$. \cite{azadkia_simple_2021} show that~$0\leq~T(X,Y)\leq~1$. Moreover,~$T(X,Y)=0$ if and only if~$X$ and~$Y$ are independent and~$T(X,Y)=~1$ if and only if~$Y=f(X)$. 

\par When $X$ is a continuous scalar variable 
this measure and variants thereof have been already introduced earlier in \cite{trutschnig_strong_2011} and \cite{dette_copula-based_2013}. These authors relate the considered dependence measures  to the copula that links the distributions of~$X$ and~$Y$. 
\cite{chatterjee_new_2020} has extended their work to arbitrary marginals and proposed a  tuning parameter-free estimator. Since then, the dependence coefficient~$T=T(X,Y)$ has attracted a lot of attention. \cite{azadkia_simple_2021} have treated a further generalization allowing for multivariate~$X$ and  conditional dependence.
An alternative approach for multivariate and continuously distributed $X$ was given by \cite{griessenberger_multivariate_2022}.  \cite{cao_correlations_2020} relate it to the maximal correlation coefficient and study similar measures that are able to detect functional relationships of prespecified shape.

\par Much research has been conducted on the asymptotic behavior of related estimators and independence tests: \cite{shi_power_2022} prove the asymptotic normality of the estimator proposed by \cite{azadkia_simple_2021} when~$X$ and~$Y$ are independent and reveal a weakness in the corresponding test's capability to identify local alternatives, a matter that is also treated by \cite{bickel_measures_2022} and \cite{lin_boosting_2023}. The latter papers also advise on how to overcome these issues. \cite{auddy_exact_2024} investigate for which kind of contiguous alternatives the test possesses non-trivial power. As~$T$ is non-symmetric, \cite{zhang_asymptotic_2023} studies the asymptotic behavior of a symmetrized version proposed by \cite{chatterjee_new_2020} under the null hypothesis of independence. For a detailed review of recent developments, we refer to \cite{chatterjee_survey_2023}.

\par Motivated by our research in functional data analysis (FDA), a main target of this article is to explore asymptotic inferential procedures for~$T$ (consistent estimation and testing $T=0$) when the covariates~$X$ take values in an infinite-dimensional space. We are particularly interested in the case where~${X=(X(u)\colon u\in\mathcal{U})}$ is a curve. A very common setting in FDA is that~${X\in L^2(\mathcal{U})}$, the space of square integrable functions on some continuum~$\mathcal{U}$. 

\par Attempts to generalize the setting of \cite{azadkia_simple_2021}  to a broad class of topological spaces have been made previously by \cite{deb_measuring_2020} (unconditional dependence) and \cite{huang_kernel_2022} (conditional dependence). Those authors consider measures of association based on so called characteristic kernels. The empirical versions of these measures are relying on geometric graphs related to the sample, like the $k$-nearest neighbor graph.  \cite{deb_measuring_2020} show that the measure~$T$ and the corresponding estimator can be viewed as a special case (see their Proposition 8.2).

A bottleneck arising in current literature arises from required bounds on the maximum degree of the underlying graphs. While the assumptions made can be shown to hold for multivariate data, they are either difficult to verify or, as we are going to demonstrate, may even fail in infinite-dimensional spaces. Our Theorem~\ref{thm:growing-degree} below gives some insight in this matter and shows that nearest neighbor graphs for functional data can have maximal degrees which do not just diverge with sample size, but their rate of divergence can be rather fast. This result easily extends to $k$-nearest neighbor graphs and minimal spanning trees (see Remark~\ref{remark:knn-mst}) and illustrates that assumptions that have been used previously in the literature may fail in infinite-dimensional settings. Since nearest-neighbor methods are an important ingredient to many statistical and machine learning procedures, our  findings may have far-reaching implications for the applicability of graph based methods in the context of functional data.

The subsequent sections are organized as follows: In Section~\ref{s:estimator}, we will briefly outline the estimation approach in \cite{azadkia_simple_2021}. We show that under fairly mild technical conditions the estimator remains consistent in general metric spaces. We also use this example to explain the difficulties that one incurs when using nearest-neighbor-based statistics in infinite-dimensional spaces.  In Section~\ref{s:main-results} we construct some specific functional random samples for which we are able to  give upper and lower bounds on the maximal degree of the resulting nearest neighbor graphs. Additionally, we obtain the limiting distribution of the estimator by \cite{azadkia_simple_2021} under independence. The resulting independence test is shown to be universally consistent. We apply this test to real world data in Section~\ref{s:real-data}. In Section~\ref{s:proofs} we give detailed proofs. \appendixtext{}{The supplement to this article~\citep{hormann_supplement_2024}} contains a simulation study as an illustration of the empirical performance of our theory.

%% file: estimator.tex
\section{An estimator for \texorpdfstring{$T(X,Y)$}{T(X,Y)}}\label{s:estimator}

For the remainder of this paper, we impose some mild technical requirements that will streamline our presentation. To this end, we define~${H(t):=P(d(X,X')\leq~t)}$, where~$X'$ an independent copy of~$X$.

\begin{Assumption}\label{ass:cont} We have
\begin{enumerate}[(a)]
\item~$H(t)$ is continuous; 
\item~$F(t)$ is continuous;
\item for~$P_Y:=P\circ Y^{-1}$ almost all~$t$, the mapping~$x\mapsto G_x(t)$ is continuous~$P_X:=P\circ X^{-1}$ almost everywhere.
\end{enumerate}
\end{Assumption}
Assumption~\ref{ass:cont} constitutes a set of continuity conditions on the distribution of~$(X,Y)$. Continuity assumptions have also been used in \cite{dette_copula-based_2013}, while \cite{azadkia_simple_2021} work under general distributional assumptions. We note that Assumption~\ref{ass:cont} could be relaxed, but in the context of functional data, our assumptions are reasonably general and common. For example, a violation of (a) would arise if~$P_X$ is a discrete measure, which is quite uncommon for a functional data model. %

Assumption~\ref{ass:cont} (b) implies that~$\mathrm{Var}(G_X(t))=\ev G_X^2(t)-G^2(t)$ and~$\int G^2(t)dF(t)=~1/3$. Hence, estimation of~$T(X,Y)$ reduces to estimation of
\begin{equation}\label{e:Q}
Q(X,Y):=\int \ev G_X^2(t)dF(t).
\end{equation}
Consider a random sample~$\{(X_i,Y_i)$,~$1\leq~i\leq~n\}$ with~$(X_i,Y_i)\sim (X,Y)$. 
For every~${i\in \{1,\ldots, n\}}$, let~$N(i)=N_n(i)$ be the index of the nearest neighbor of~$X_i$ in the sample of covariates~$X_1,\ldots, X_n$, that is,~$d(X_i,X_{N(i)})\leq~d(X_i,X_j)$ for all~$j\neq i$. Assumption~\ref{ass:cont} (a) implies that~$N(i)$ is unique with probability one. If~$n$ is large, we expect that~$X_{N(i)}$ is close to~$X_i$ and therefore, considering Assumption~\ref{ass:cont}~(c), we get by some heuristics that for almost all~$t$
\begin{equation}\label{e:crucialapprox}
\ev G_X^2(t)\approx \ev G_{X_i}(t)G_{X_{N(i)}}(t).
\end{equation} Moreover, noting that
\begin{align}\begin{split}
\ev G_{X_i}(t)G_{X_{N(i)}}(t)
&=\ev\Big(\ev\big[1\{Y_i\geq~t\}1\{Y_{N(i)}\geq~t\}|X_1,\ldots, X_n\big]\Big)\\
&=\ev 1\{Y_i\geq~t\}1\{Y_{N(i)}\geq~t\},
\end{split}\end{align}
we obtain  the approximation
\begin{equation}
Q(X,Y)\approx E\int 1\{Y_i\geq~t\}1\{Y_{N(i)}\geq~t\}dF(t)=E \min\{F(Y_i),F(Y_{N(i)})\}.
\end{equation}
This motivates the estimators
\begin{equation}\label{e:estimatorQ}
\widehat Q_n=\frac{1}{n}\sum_{i=1}^n \min\{F_n(Y_i),F_n(Y_{N(i)})\}\quad \text{and}\quad \widehat T_n=6\,\widehat Q_n-2,
\end{equation}
with~$F_n$ being the empirical distribution function of~$Y_1,\ldots, Y_n$.
\begin{Remark}
We note that~$\widehat T_n$ is, in essence, the estimator given in \cite{azadkia_simple_2021}, where we have been taking into account that~$Y$ is supposed to have a continuous distribution. 
\end{Remark}
\begin{Remark}
Let~$\mathcal{G}_n=\mathcal{G}_n(X_1,\ldots,X_n)$ be the nearest neighbor graph related to the sample~$X_1,\ldots, X_n$. 
Then~$\widehat T_n$ is a functional of the responses~$Y_1,\ldots, Y_n$ and of~$\mathcal{G}_n$.
\end{Remark}

A non-trivial question is whether the heuristics leading to~$\widehat T_n$ can be rigorously justified, i.e.\ whether~$\widehat T_n$ is a consistent estimator of~$T(X,Y)$. Our first result establishes weak consistency of~$\widehat T_n$ for i.i.d.\ data under the continuity properties stated in Assumption~\ref{ass:cont}.

\begin{Theorem}\label{thm:consistency}
 Assume that~$X$ takes values in a separable metric space~$(H,d)$. Let Assumption~\ref{ass:cont} hold. Then we have that 
 \begin{equation}\label{e:consistency}
 \widehat T_n\convP T(X,Y)\quad \text{as}\quad n\to\infty. 
 \end{equation}
\end{Theorem}

\cite{azadkia_simple_2021} have shown that in the finite-dimensional setup, almost sure convergence can be obtained in~\eqref{e:consistency} and Assumption~\ref{ass:cont} can be dropped. 
 A closer inspection of their paper reveals that the proof is crucially based on the fact that within a set of arbitrary points~${\mathcal{S}_n=\{x_1,\ldots, x_n\}\subset\mathbb{R}^p}$, an element~$x_i\in \mathcal{S}_n$ can be the nearest neighbor to at most~$k(p)$ points in~$\mathcal{S}_n$, where~$k(p)$ is some finite constant that is independent of~$n$. Formally, define the maximum degree~$L_n=L_n(\mathcal{S}_n):=\max_{1\leq~i\leq~n} L_{i,n}$, where~$L_{i,n}$ is the number of elements in~$\mathcal{S}_n\setminus x_i$ that have~$x_i$ as their nearest neighbor. Then
\begin{equation}\label{eq:nb-bound}
 L_n\leq~k(p)\quad \text{for all~$n\geq~1$}.
\end{equation}
For example, if we generate a sequence in~$\mathbb{R}$ then trivially~$L_n\leq~2$. In~$\mathbb{R}^2$ it is not hard to see that~$L_n\leq~6$. \cite{kabatjanski_bounds_1978} have shown that~$k(p)\leq~\mathrm{const}\times \gamma^p$ for some absolute constant~$\gamma>1$ and for all~$p\geq~1$.

\par In this paper, however, we are interested in an infinite-dimensional covariate space. Here we generally cannot bound~$L_n$ by a constant. Consider, for example, an orthonormal sequence of elements~$x_k$,~$k\geq~2$, in~$L^2(\mathcal{U})$, i.e.
\begin{equation}
    \int_\mathcal{U} x_k(u)x_\ell(u)du=\delta_{k,\ell},
\end{equation}
where~$\delta_{k,\ell}$ denotes the Kronecker delta. Then it holds that~${d(x_k,x_\ell)=\sqrt{2}(1-\delta_{k,\ell})}$. If we set~$x_1=0$, then~$x_1$ is the nearest neighbor to all other elements, and thus~$L_n=n-1$. 
In this example, however big is~$n$,~$d(x_i,x_{N(i)})$ is not becoming small, and the heuristics that lead to \eqref{e:crucialapprox} are no longer applicable. This demonstrates that we need to explore and manage the order of magnitude of $L_n$ in a random sample in  infinite-dimensional spaces—an aspect that appears to be entirely unexamined in the literature. For insightful discussions on the problematic usage of nearest neighbor methods in high dimension, we refer to \cite{beyer_when_1997} and \cite{durrant_when_2009}.

\par Let us illustrate this problem on the data example that will be presented in Section~\ref{s:real-data}. These data consist of the age distributions of~$n=2117$ Austrian municipalities. In Figure~\ref{fig:agecluster} we see one curve which turns out to be the nearest neighbor of 66 other curves, which corresponds to~$\approx 3\%$ of the sample size. 
The example confirms that our theoretical issue is also relevant in practice and that we may get nearest neighbor graphs with some rather large degrees.
\begin{figure}[h]
     \centering
     \includegraphics[width=10cm]{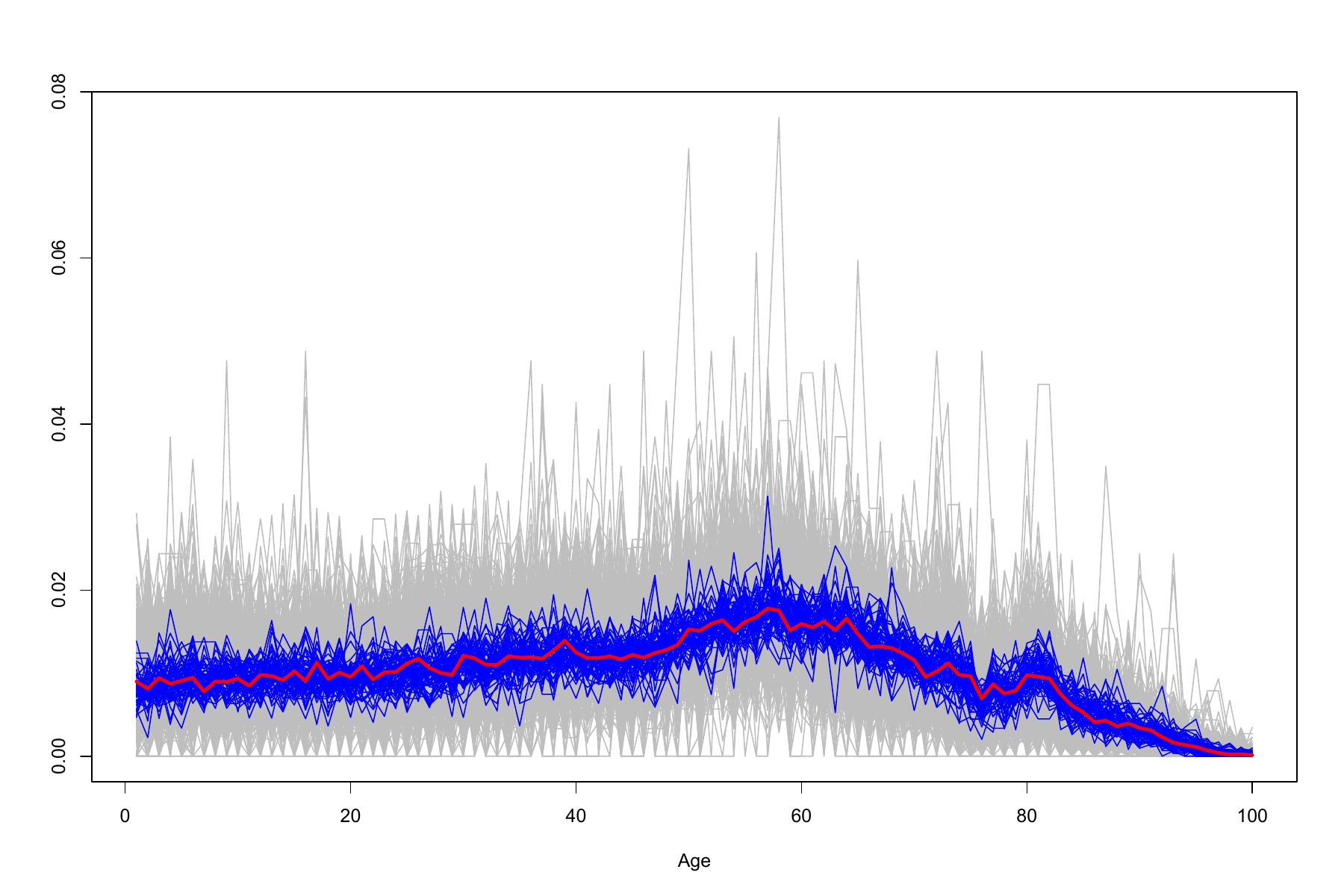}
     \caption{Age distribution curves for 2117 municipalities in Austria. On the~$y$-axis we see the proportions. The curve related to Wolfsberg (solid red) is the nearest neighbor of the age distribution curves of 66 municipalities (blue). }
     \label{fig:agecluster}
\end{figure}

%% file: main-results.tex
\section{Main results}\label{s:main-results}

The data points from the counterexample in Section~\ref{s:estimator} constitute an infinite series of orthogonal functions, leading to a set of curves that differs significantly from of a random sample. In our next result we demonstrate that $L_n$ may diverge at a polynomial rate even in random samples. To this end, we consider random elements~$X_i$ taking values in~$L^2([0,1])$. The space~$L^2([0,1])$ is equipped with the inner product
\begin{equation}
    \langle x,y\rangle=\int_0^1 x(u)y(u)du
\end{equation}
and the corresponding norm~$\|x\|^2=\langle x, x\rangle$, which gives rise to the distance~${d(x,y)=\|x-y\|}$.
Let~$\Sigma =\mathrm{Var}(X_1)$ be the covariance operator of~$X_1$, that is
\begin{equation}
\Sigma(v)=\ev (X_1-EX_1)\langle X_1-EX_1,v\rangle
\end{equation}
and denote by~$\lambda_1\geq~\lambda_2\geq~\cdots$ its eigenvalues and by~$e_1, e_2,\ldots$ corresponding eigenfunctions.
Moreover, recall that a function~$\ell(x)$ is called slowly varying at~$\infty$  if~${\ell(y x)/\ell(x)\to 1}$ for any~$y>0$ and~$x\to\infty$. We write~$\ell \in R_0$. It is a well known fact that~$\ell\in R_0$ implies that~$\ell(x)=o(x^\delta)$ for any~$\delta>0$. 
We use the notion~$a_n=\omega(b_n)$ if~$|a_n/b_n|\to\infty$.

\begin{Theorem}\label{thm:growing-degree}
 Let~$\Sigma$ be a symmetric, positive semidefinite operator on~$L^2([0,1])$ having eigenvalues~$\lambda_k$. Suppose that~$\Sigma$ is trace-class, i.e.,~$\sum_{k\geq~1}\lambda_k<\infty$. Then there is a random variable~$X\in L^2([0,1])$ with covariance operator~$\mathrm{Var}(X)=\Sigma$, such that for a random sample~$X_1,X_2,\ldots~$ with~$X_i\stackrel{\text{iid}}{\sim} X$ the following holds:
 \begin{enumerate}
  \item[(i)] If~$\lambda_k=o\left(\sum_{j\geq~k}\lambda_j\right)$, then  ~$L_n=L_n(X_1,\ldots, X_n)\to\infty$ in probability.
    \item[(ii)] If~$\lambda_k=\ell(k) k^{-a}$ with~$\ell\in R_0$, then~$L_n = \omega_P\Big(h(n)n^{1/({2a-1})}\Big)$ for some~$h\in R_0$.
  \item[(iii)] Under the same condition as in (ii), it holds that such that~${L_n = o_P(h(n)n^{2/(2a-1)})}$ for some~${h\in R_0}$.
 \end{enumerate}
\end{Theorem}

For example, if we choose $\Sigma$ to be the covariance operator of a standard Brownian motion, then $\Sigma$ is a kernel operator with kernel $c(u,v)=\min\{u,v\}$.
It is well known that the eigenvalues of $\Sigma$ decay at the rate $k^{-2}$. Our Theorem~\ref{thm:growing-degree} now implicates, that we can find a functional random sample with covariance function $c(u,v)$, such that $L_n$ diverges at some polynomial rate lying between $n^{1/3}$ and $n^{2/3}$.

\begin{Remark}\label{remark:knn-mst}
   It is easy to see from the proof of Theorem~\ref{thm:growing-degree}, that in our construction of the random sample the maximum degree $L_n$ diverges even faster if, instead of the $1$-nearest neighbor graph, one uses $k$-nearest neighbor graphs with $k>1$ (possibly growing) or minimum spanning trees.
\end{Remark}

Assuming that~$X$ has a density and is independent of~$Y$, \cite{shi_power_2022} have shown that~$\sqrt{n}\widehat T_n$ will be asymptotically normally distributed. The necessity for~$X$ to possess a density function serves as an initial indication that their result does not directly generalize to infinite-dimensional~$X$. Also, the limiting variance constitutes a rather involved expression and crucially depends on the geometry of the Euclidean space. It is deduced from the results of \cite{henze_fraction_1987} and involves the limiting law of~$L_{1,n}$. Inspired by \cite{deb_measuring_2020} and \cite{lin_limit_2022}, we overcome the latter problem, by using a data-dependent self-normalization. 

An important message learned from Theorem~\ref{thm:growing-degree} is that we should be allowing $L_n$ to diverge at some polynomial rate. We are imposing the following high-level assumption.  
\begin{Assumption}\label{ass2}
We have~$L_n=o_P(n^{1/4})$.
\end{Assumption}
Before we discuss Assumption~\ref{ass2} in more detail, we will state our next result.

\begin{Theorem}\label{thm:normality}
Let Assumptions~\ref{ass:cont} and~\ref{ass2} hold. Assume that~$X_i$ and~$Y_i$ are independent. Then there is a random variable~$W_n=W_n(X_1,\ldots,X_n)$ such that
\begin{equation}
\sqrt{\frac{n}{36\,W_n}}\widehat T_n\convd N(0,1),\quad n\to\infty.
\end{equation}
The variable~$W_n$ is explicitly defined in~\proofreference{\eqref{e:Wn}}{(7) in the supplement}.
\end{Theorem}

\begin{Remark}\label{remark:deb-lin}
 Both, \cite{deb_measuring_2020} and \cite{lin_limit_2022}, have obtained a CLT for~$\widehat T_n$ (or generalizations). Although both results are impressive in their generality, they are in general not applicable in our setting: Assumption~(A3) in \cite{deb_measuring_2020} implies that~$L_n$ is bounded in case of nearest neighbor graphs. But even if other graph functionals are used, which might allow for growing~$L_n$, this growth is limited only to a polylogarithmic rate, which may be limiting as indicated by our Theorem~\ref{thm:growing-degree}. \cite{lin_limit_2022} obtain normality also under dependence of~$Y$ on~$X$, but their result is based on the nearest neighbor CLT by \cite{chatterjee_new_2008}, which requires~$L_n$ to be bounded by a constant.
\end{Remark}

If the target is to test
\begin{equation}
\text{
$\mathcal{H}_0\colon Y_i$ and~$X_i$ are independent\quad v.s.\quad~$\mathcal{H}_A\colon \mathcal{H}_0$ doesn't hold,
}
\end{equation}
we can use the test statistics~$\mathcal{I}_n:=\sqrt{n/({36\,W_n})}\widehat T_n$. If~$n$ is sufficiently large, we reject at significance level~$\alpha$ if~$\mathcal{I}_n>z_{1-\alpha}$, where~$z_\alpha$ is the~$\alpha$-quantile of a standard normal variable. The following corollary shows that under our assumptions this test is universally consistent.

\begin{Corollary}\label{cor:consistenttest}
Let Assumptions~\ref{ass:cont} and~\ref{ass2} hold. Assume that~$X_i$ and~$Y_i$ are not independent. Then
$\mathcal{I}_n\to\infty$ in probability as~$n\to\infty$.
\end{Corollary}

As presented in \simulationreference{Appendix~\ref{s:empirical}}{the supplement to this paper}, the suggested test for independence shows good power against fixed alternatives in practice. However, the lack of power against local alternatives as mentioned above remains a potential issue.  A notable advantage is the pivotal limit of our statistic under the null. If the data set is large, the test not only gains accuracy, but it also remains very fast to execute. Presumably more powerful tests, like those based on distance correlation and ball covariance have a non-pivotal asymptotic and require numerically intensive resampling procedures. For a more detailed discussion on this matter we refer to \simulationreference{Appendix~\ref{s:empirical}}{the supplemental material  \citep{hormann_supplement_2024}}.

While Assumption~\ref{ass2} already constitutes a strong relaxation of conditions imposed on the maximal degree of the nearest neighbour graph in other papers, we are still lacking a general class of functional models for which this condition can be verified. Our last result shows that Assumption~\ref{ass2} holds under some general conditions.  To this end, we need another assumption, which holds e.g.\ for Gaussian processes.
\begin{Assumption}\label{ass:density}
The scores~$Z_{1k}:=\langle X_1,e_k\rangle$,~$k\geq~1$, are independent and have density functions~$f_{k}(s)$. Moreover,~$f_1(s)$ and~$f_2(s)$ are uniformly bounded over~$s\in\mathbb{R}$.
\end{Assumption}

\begin{Theorem}\label{thm:ass-2-fulfilled}
Let Assumption~\ref{ass:density} hold and assume that the eigenvalues~$\lambda_k$ of the covariance operator~$\mathrm{Var}(X)$ satisfy~${\sum_{j\geq~q}\lambda_j=O(\lambda_q)}$ and~${\lambda_{q_n}=o(n^{-6})}$ when~${q_n=\lfloor 0.899\log n\rfloor}$. Then Assumption~\ref{ass2} is fulfilled.
\end{Theorem}

The proof of Theorem~\ref{thm:ass-2-fulfilled} will follow from our general but slightly more technical Proposition~\proofreference{\ref{prop:general-condition-clt} in Section~\ref{s:ass1}}{5 in the supplement}.

%% file: real-data.tex
\section{Real data illustration}\label{s:real-data}
This data set was previously studied in \cite{ofner_quantile_2021} in the context of functional quantile regression. We apply the procedures discussed to COVID-19 vaccination data from 2117 Austrian municipalities. Specifically, we consider for each municipality:
\begin{itemize}
 \item a curve~$X$ representing the population's age distribution on~01.01.2021. This data is provided by \cite{statistik_austria_bevolkerung_2021}.
 \item the proportion~$Y$ of the population who had received at least two COVID-19 vaccinations by 13.10.2021. This data is provided by the \cite{bundesministerium_fur_soziales_gesundheit_pflege_und_konsumentenschutz_bmsgpk_covid-19_2021}.
\end{itemize}
We will compare our statistic~$\widehat T_n$ and the corresponding independence test~$\mathcal{I}_n$ to permutation tests based on distance correlation~($\widehat R_n,~\mathcal{I}^\mathrm{DC}_n$) and ball correlation ($\widehat B_n,~\mathcal{I}^\mathrm{BC}_n$), as well as to an independence test for functional data developed by \cite{garcia-portugues_goodness--fit_2014}~($\mathcal{I}^\mathrm{CvM}_n$). The latter is a bootstrap test utilizing a Cram\'er-von-Mises type statistic. We refer to \cite{szekely_measuring_2007}, \cite{pan_ball_2020} and \cite{garcia-portugues_goodness--fit_2014} for details on these competing approaches.
\par Figure~\ref{fig:vaccination-data} shows the entire sample of age curves. As it is common in FDA, the curves are only observable on a discrete set of observation points (in our case,~100 points). Since there is no measurement error associated with the age curves, we have omitted a preprocessing and smoothing step. From visual inspection, there is no obvious connection between age structure and vaccination rate.
\begin{figure}[h!]
\centering
 \includegraphics[width=0.8\textwidth]{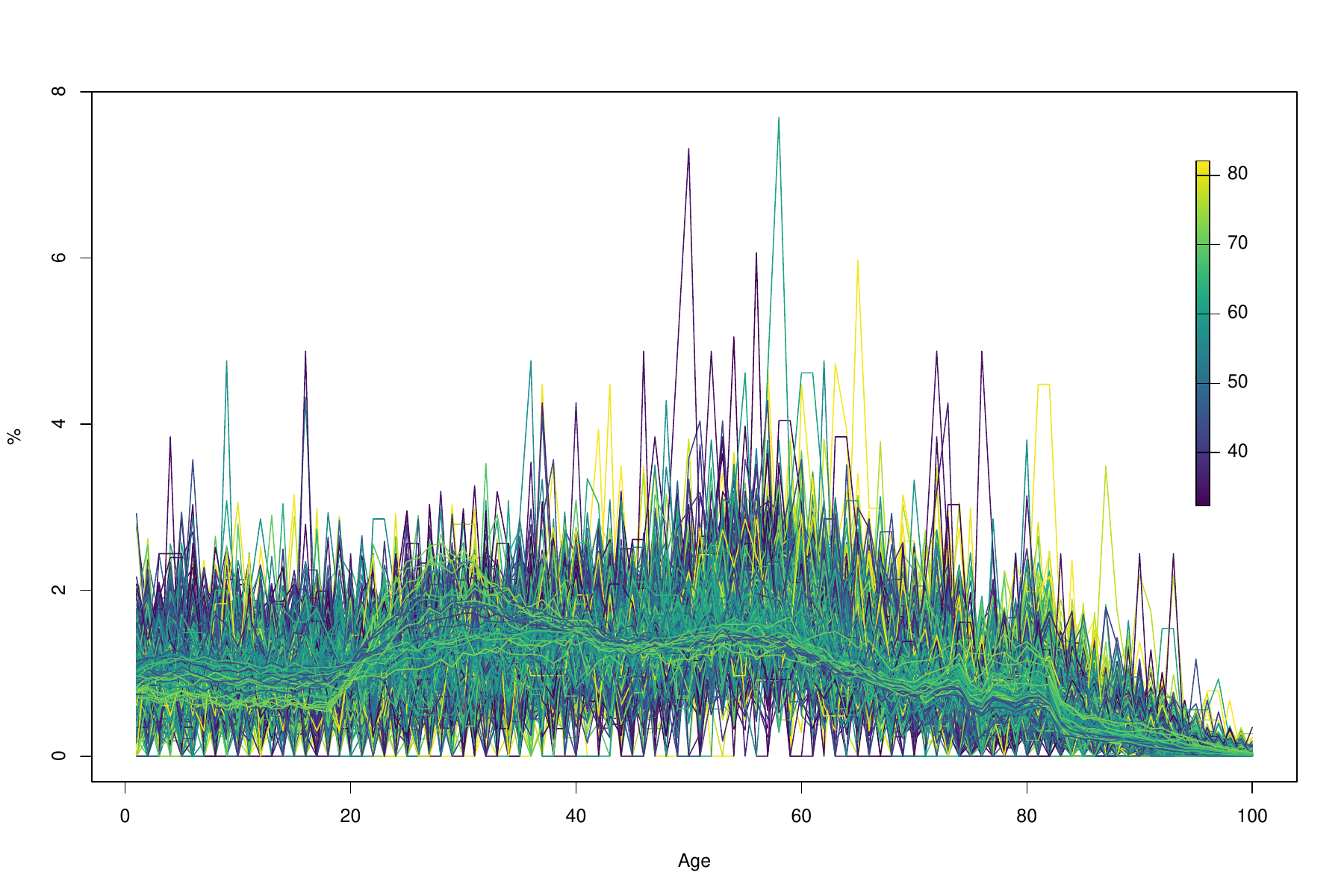}
\caption{Age curves are colored according to the vaccination rates of the corresponding municipalities.}\label{fig:vaccination-data}
\end{figure}
However, the values of~$\widehat T_n=~0.31$ and~$\widehat R_n=0.43$ indicate a clear relationship. Interestingly, the value~$\widehat B_n=0.01$ is very small. Generally, we observed in diverse simulation experiments, that ball correlation yields rather small values even under strong dependence, making it less attractive as a dependence coefficient. (Notice that ball correlation is scaled to remain between 0 and 1.) As the absolute values of these coefficients cannot be directly compared, we perform the corresponding independence tests, including also the~$\mathcal{I}^\mathrm{CvM}_n$ test. The~$p$-values returned by~$\mathcal{I}_n$,~$\mathcal{I}^\mathrm{DC}_n$,~$\mathcal{I}^\mathrm{BC}_n$ and~$\mathcal{I}^\mathrm{CvM}_n$ are~$3\times 10^{-4}\%, 0.02\%, 1\%$ and~$2\times 10^{-14}\%$, respectively. For the permutation and the bootstrap test, a number of~5000 repetitions was used. Hence, all four tests indicate a highly significant connection between the age structure and the vaccination rate. An interesting observation is that~$\mathcal I_n^\mathrm{CvM}$ is sensitive to changes in the dataset: for example, when only the 598 municipalities with at least 3000 inhabitants are considered,~$\mathcal{I}_n$,~$\mathcal I_n^\mathrm{DC}$ and~$\mathcal I_n^\mathrm{BC}$ give similar~$p$-values as before, but the~$p$-value of~$\mathcal I_n^\mathrm{CvM}$ rises to values between~$5\%$ and~$11\%$ (the variation arises from the test being randomized). Also note that for the full dataset, the computation of~$\mathcal I_n$ takes~0.36 seconds, the computation of~$\mathcal I_n^\mathrm{DC}$ takes~33.14 seconds, the one of~$\mathcal I_n^\mathrm{BC}$ takes~1:25 minutes and the computation of~$\mathcal I_n^\mathrm{CvM}$ takes~2:29 minutes.

%% file: proofs.tex
\section{Proofs}\label{s:proofs}

\proofreference{\subsection{Proof of Theorem~\ref{thm:consistency}}}{\subsection{Proof of Theorem~1}}
We need some preparatory lemmas. Variants of Lemmas~\ref{lem:neighbor-convergence} and~\ref{lem:expectaton-convergence} are given by \cite{azadkia_simple_2021}, but we present them here for the sake of completeness.

\begin{Lemma}\label{lem:neighbor-convergence}
 Let~$X_i$ be random elements with values in a separable metric space. Let~$X_{N(1)}$ be the nearest neighbor of~$X_1$ among~$X_2,\ldots,X_n$. Then~$X_{N(1)}\convas X_1$ as~$n\to\infty$.
\end{Lemma}
\begin{proof}
Let~$\varepsilon>0$. Due to separability, the space~$H$ can be covered by countably many balls of diameter~$\varepsilon$. By~$\sigma$-subadditivity of~$P$ some of those balls must have positive probability. Therefore, with probability one,~$X_1$ lies in a ball~$B$ which has positive probability. By the triangle inequality,
\begin{equation}
 P(d(X_1,X_{N(1)})>\varepsilon)\leq~P(X_2,\ldots,X_n\notin B)=(1-P(X_1\in B))^{n-1},
\end{equation}
which tends to zero, since~$P(X_1\in B)>0$.
\end{proof}

\begin{Lemma}\label{lem:expectaton-convergence}
 Let~$\widehat Q_n$ be defined as in~\proofreference{\eqref{e:estimatorQ}}{\lbrack 3\rbrack} and set
 \begin{equation}
     \widetilde{Q}_n:=\frac{1}{n}\sum_{i=1}^n \min\left\lbrace F(Y_i),F\left(Y_{N(i)}\right)\right\rbrace.
 \end{equation}
Then~${|\widehat Q_n-\widetilde Q_n|\convas 0}$ and~${|\ev\widehat Q_n-\ev\widetilde Q_n|\to 0}$.
\end{Lemma}
\begin{proof}

It is easily seen that
$ |\widetilde{Q}_n-\widehat Q_n| \leq~\sup_{t\in\R} |F(t)-F_n(t)|
$ and hence by the Glivenko-Cantelli Theorem, we have~$ |\widetilde{Q}_n-\widehat Q_n|\convas 0$. Since~$\sup_{t\in\R} |F(t)-F_n(t)|\leq~1$, dominated convergence yields~${\lim_{n\to\infty} \ev |\widetilde{Q}_n-\widehat Q_n|=0}$, which implies the second statement.
\end{proof}

\begin{Lemma}\label{lem:square-degree}
It holds that~$P(N(1)=N(2))=o(1)$.
\end{Lemma}
\begin{proof}

Since~$|X_1-X_{N(1)}|\convas 0$ by Lemma~\ref{lem:neighbor-convergence}, we have
\begin{align}\begin{split}
     P(N(1)=N(2))&\leq\!P\left(|X_1-X_{N(1)}|\geq~\frac{|X_1-X_2|}{2}\lor |X_2-X_{N(2)}|\geq~\frac{|X_1-X_2|}{2}\right)\\
     &\leq\!P\left(|X_1-X_{N(1)}|\geq~\frac{|X_1-X_2|}{2}\right)\!+\!P\left(|X_2-X_{N(2)}|\geq~\frac{|X_1-X_2|}{2}\right)\\
     &=o(1).
\end{split}\end{align}

\end{proof}

\begin{Lemma}\label{le:var}
It holds that~$\mathrm{Var}(\widetilde Q_n)\to 0$.
\end{Lemma}
\begin{proof}

We have
\begin{align}\begin{split}
 \mathrm{Var}(\widetilde Q_n)&=\frac{1}{n^2}\sum_{i\neq j} \mathrm{Cov}(\min\lbrace F(Y_i), F(Y_{N(i)})\rbrace,\min\lbrace F(Y_j), F(Y_{N(j)})\rbrace) \\
 &\quad+\frac{1}{n^2}\sum_{i=1}^n \mathrm{Var}(\min\lbrace F(Y_i), F(Y_{N(i)})\rbrace) =:\psi_{1,n}+\psi_{2,n}.
 \end{split}\end{align}
Clearly,~$\psi_{2,n}\to 0$ for~$n\to\infty$. For reasons of symmetry, the summands of~$\psi_{1,n}$ are all equal and hence
\begin{align}\begin{split}
\psi_{1,n}\leq\mathrm{Cov}( F(Y_1)\wedge F(Y_{N(1)}), F(Y_2)\wedge F(Y_{N(2)})).
\end{split}\end{align} 
Let~$F_X(t)=P(Y\leq~t|X)$ and note that we may represent~$Y_i=F_{X_i}^{-1}(Z_i)$, with i.i.d.\ random variables~$Z_i$, which are uniformly distributed on~$(0,1)$ and such that~$\mathbf{Z}=(Z_1,\ldots, Z_n)$ and~$\mathbf{X}=(X_1,\ldots, X_n)$ are independent. Then set~$F(Y_i)=:f(X_i,Z_i)$. We want to prove that~$F(Y_1)\wedge F(Y_{N(1)})$ and~$F(Y_2)\wedge F(Y_{N(2)})$ are asymptotically independent, which then implicates that~$\psi_{1,n}\to 0$. Hence, consider the joint distribution function
\begin{align}\begin{split}
  &P\left(F(Y_1)\wedge F(Y_{N(1)})\leq~s, F(Y_2)\wedge F(Y_{N(2)})\leq~t\right)\\
  &=\ev P\left(f(X_1,Z_1)\wedge f(X_{N(1)},Z_{N(1)})\leq~s, f(X_2,Z_2)\wedge f(X_{N(2)},Z_{N(2)})\leq~t\middle| \mathbf{X}\right).
\end{split}\end{align}
Since~$\mathbf{X}$ and~$\mathbf{Z}$ are independent and the nearest neighbor graph is a function of~$\mathbf{X}$, the above expectation can be written as (see \cite{durrett_probability_2019}, Example 4.1.7.)
\begin{align}
 \int P\big(f(x_1,Z_1)\wedge f(x_{n(1)},Z_{n(1)})\leq~s, f(x_2,Z_2)\wedge f(x_{n(2)},Z_{n(2)}) \leq~t\big) dP_\mathbf{X}(\mathbf{x}),\label{eq:int}
\end{align}
where we integrate over the product space of the~$X_i$'s and where~$x_{n(i)}$ is the nearest neighbor of~$x_i$ within~$\mathbf{x}=(x_1,\ldots, x_n)$.
We now split the integral over regions
\begin{equation}
R_1:=\{\mathbf{x}\colon |\{1,2,n(1),n(2)\}|=4\}\quad \text{and}\quad R_2:=R_1^c.
\end{equation}
We have
\begin{equation}
P(\mathbf{X}\in R_1^c)\leq~P(N(1)=2)+P(N(2)=1)+P(N(1)=N(2)).
\end{equation}
Notice that~$P(N(1)=2)=P(N(2)=1)=1/n$ and~$P(N(1)=N(2))=o(1)$ by Lemma~\ref{lem:square-degree}.
We conclude, that~$P(\mathbf{X}\in R_1^c)\to 0$. 
Hence,
\begin{align}\begin{split}
  &P\left(F(Y_1)\wedge F(Y_{N(1)})\leq~s, F(Y_2)\wedge F(Y_{N(2)})\leq~t\right)\\
  &= \int_{R_1} P\big(f(x_1,Z_1)\wedge f(x_{n(1)},Z_{n(1)})\leq~s, f(x_2,Z_2)\wedge f(x_{n(2)},Z_{n(2)}) \leq~t\big) dP_\mathbf{X}(\mathbf{x})\\
  &\qquad\qquad\qquad\qquad\qquad\qquad\qquad\qquad\qquad\qquad\qquad\qquad\qquad\qquad\qquad\qquad+o(1)\\
  &= \int_{R_1} P\big(f(x_1,Z_1)\wedge f(x_{n(1)},Z_{n(1)})\leq~s\big) P(f(x_2,Z_2)\wedge f(x_{n(2)},Z_{n(2)}\big) \leq~t\big) dP_\mathbf{X}(\mathbf{x})\\
  &\qquad\qquad\qquad\qquad\qquad\qquad\qquad\qquad\qquad\qquad\qquad\qquad\qquad\qquad\qquad\qquad+o(1)\\
  &= \int P\big(f(x_1,Z_1)\wedge f(x_{n(1)},Z_{n(1)})\leq~s\big) P(f(x_2,Z_2)\wedge f(x_{n(2)},Z_{n(2)}\big) \leq~t\big) dP_\mathbf{X}(\mathbf{x})\\
  &\qquad\qquad\qquad\qquad\qquad\qquad\qquad\qquad\qquad\qquad\qquad\qquad\qquad\qquad\qquad\qquad+o(1)\\
  &=P\left(F(Y_1)\wedge F(Y_{N(1)})\leq~s\right)P\left(F(Y_2)\wedge F(Y_{N(2)})\leq~t\right)+o(1).
\end{split}\end{align}
\end{proof}

\begin{proof}[Proof of Theorem~\proofreference{\ref{thm:consistency}}{1}]
Due to Lemmas~\ref{lem:expectaton-convergence} and~\ref{le:var}, the proof follows if we show that~${\ev\widetilde Q_n\to Q}$.
Noting that the variables~$\min\lbrace F(Y_i),F(Y_{N(i)})\rbrace$,~$1\leq~i\leq~n$, are identically distributed and using our derivation in Section~\proofreference{\ref{s:estimator}}{2} we get
\begin{equation}
\ev\widetilde{Q}_n= \ev \min\lbrace F(Y_1),F(Y_{N(1)})\rbrace=\int \ev G_{X_1}(t)G_{X_{N(1)}}(t)dF(t).
\end{equation}
By Lemma~\ref{lem:neighbor-convergence}, Assumption~\proofreference{\ref{ass:cont}}{1}~(c), and dominated convergence, we get the desired result.
\end{proof}

\proofreference{\subsection{Proof of Theorem~\ref{thm:growing-degree}}}{\subsection{Proof of Theorem~3}}

In this section we construct an example of a functional random variable~$X$ as in Theorem~\proofreference{\ref{thm:growing-degree}}{3}. Set~${\sigma^2=\sum_{k=1}^\infty \lambda_k}$ and~${p_k=\lambda_k/\sigma^2}$. Let~$A$ be a continuous random variable with mean~$0$ and variance~$\sigma^2$ and~$K$ a discrete random variable independent of~$A$ with~$P(K=~k)=p_k$.

\begin{Lemma} 
Set~$X=A\times e_K$. Then~$X$ has mean 0 and covariance~$\Sigma$.
\end{Lemma}
\begin{proof}
Obviously we have~$EX=0$. Then
\begin{align}\begin{split}
\ev X\langle v,X\rangle= \sigma^2\sum_{k\geq~1} P(K=k)e_k\langle v,e_k\rangle,
\end{split}\end{align}
which, by the spectral theorem for compact operators, equals to~$\Sigma(v)$.
\end{proof}
Now consider a random sample~$X_1,\ldots,X_n$ with~$X_i=A_i\times e_{K_i}\sim X$. It holds that
\begin{equation}\label{eq:example-distance}
 \|X_i-X_j\|^2=A_i^2+A_j^2-2A_iA_j1\{K_i=K_j\}.
\end{equation}

\begin{Lemma}\label{le:lowerboundLn}
 For any integer~$k\geq~1$, define~$\mathcal{M}_{k,n}=\{i\leq~n\colon K_i=~k\}$,~$f_{k,n}=~|\mathcal{M}_{k,n}|$ and finally~$G_n=~|\{k:~f_{k,n}=~1\}|$. Then~$L_n\geq~G_n-1$.
\end{Lemma}
\begin{proof} Suppose~$f_{k,n}=1$. Then~$\mathcal{M}_{k,n}$ contains a single element, say~$i_k$.  This means that~$X_{i_k}$ is the only curve in the sample with a shape that is proportional to the function~$e_k$. By~\eqref{eq:example-distance}, for any~$j\in\{1,\ldots,n\}$ with~$j\neq i_{k}$ we have
\begin{equation}
    \|X_{i_k}-X_j\|^2=A_{i_{k}}^2+A_j^2.
\end{equation}
This implies that the nearest neighbor of each~$X_{i_{k}}$ is~$X_{i^*}$ with ~$i^*=\mathrm{argmin}_{j\neq i_k}\, A_j^2$. If~$f_{i^*,n}=1$, then~$L_{i^*,n}\geq~G_n-1$, if~$f_{i^*,n}>1$, then even~$L_{i^*,n}\geq~G_n$ holds.
\end{proof}

\begin{proof}[Proof of Theorem~\proofreference{\ref{thm:growing-degree}}{3}]
 Without loss of generality, assume that~$\sigma^2=1$.
 \par (i) For any sequence~$x_n\to\infty$, it holds that
 \begin{equation}
   \theta_n:=P(K\geq~x_n)=\sum_{k\geq~x_n}\lambda_k\quad \text{and}\quad P(K=k)\leq~\lambda_{x_n}\quad \text{for all~$k\geq~x_n$}.
 \end{equation}
Having drawn a random sample~$(A_1,K_1),\ldots,(A_n,K_n)$, let~$\mathcal{R}_n=\{1\leq~i\leq~n\colon K_i\geq~x_n\}|$. It holds that~$ |\mathcal{R}_n|$ has a binomial distribution~$B_{n,\theta_n}$ and that
\begin{equation}
    |\lbrace i:f_{K_i,n}=1\rbrace\cap\mathcal R_n|\leq~G_n.
\end{equation}
The target is to show now (a) that~$|\mathcal R_n|\to\infty$ and (b) that~$|\lbrace i:~f_{K_i,n}=~1\rbrace\cap~\mathcal R_n|=~|\mathcal R_n|$ with high probability. For (a), we choose a sequence~$(x_n)$ which grows slowly enough for~$n\theta_n\to~\infty$.
In this case, we can always find a diverging sequence~$(\psi_n)$ such that also~$\frac{n\theta_n}{\psi_n}\to~\infty$. By elementary properties of the binomial distribution and the Cauchy-Schwarz inequality, we have
\begin{align}\begin{split}
P\left(B_{n,\theta_n}\leq~n\theta_n-\sqrt{\psi_n n\theta_n}\right)&=P\left(B_{n,1-\theta_n}-n(1-\theta_n)\geq~\sqrt{\psi_n n\theta_n}\right)\leq~\psi_n^{-1},
\end{split}\end{align}
and likewise
$
P\left(B_{n,\theta_n}\geq~n\theta_n+\sqrt{\psi_n n\theta_n}\right)\leq~\psi_n^{-1}.
$
By the choice of~$\psi_n$ we have
\begin{equation}
    n\theta_n/2\leq~n\theta_n-\sqrt{\psi_n n\theta_n}< n\theta_n+\sqrt{\psi_n n\theta_n}\leq~2n\theta_n
\end{equation}
if~$n$ is large enough.
It follows that 
\begin{equation}\label{e:Rn}
P(|\mathcal{R}_n|\geq~n\theta_n/2)\to 1\quad\text{and}\quad P(|\mathcal{R}_n|\leq~2n\theta_n)\to 1.
\end{equation}
From the left relation in \eqref{e:Rn} we immediately deduce (a). Moreover, for big enough~$n$,
\begin{align}\begin{split}
&P(|\lbrace i:f_{K_i,n}=1\rbrace\cap\mathcal R_n|= |\mathcal R_n|)/P(|\mathcal{R}_n|\leq~2n\theta_n)\\
&\quad=P\left(\text{all~$K_i$ with~$i\in\mathcal{R}_n$ are distinct}\, \big|\, |\mathcal{R}_n|\leq~2n\theta_n\right)\\
&\quad\geq\prod_{i=1}^{2n\theta_n}(1-i\lambda_{x_n}/\theta_n)\\
&\quad=\exp\left(\sum_{i=1}^{2n\theta_n}\log(1-i\lambda_{x_n}/\theta_n)\right)\\
&\quad\geq~\exp\left(-2\frac{\lambda_{x_n}}{\theta_n}\sum_{i=1}^{2n\theta_n} i\right)\\
&\quad\geq~\exp\left(-4\frac{\lambda_{x_n}}{\theta_n}\big(n\theta_n+1\big)^2\right).
\end{split}\end{align}
Hence, using the left relation in \eqref{e:Rn}, (b) follows if
\begin{equation}\label{eq:n-theta-lambda-0}
 n^2\lambda_{x_n}\theta_n\to 0.
\end{equation}
Since by assumption~$\lambda_{x_n}=o\left(\theta_n\right)$, it is possible to choose a sequence~$x_n$ such that~$n\theta_n\to\infty$ and~\eqref{eq:n-theta-lambda-0} holds.
Combining these results with Lemma~\ref{le:lowerboundLn} we infer
$P(L_n\geq~n\theta_n/2)\to 1,$ and hence that~$L_n$ diverges with a rate at least as fast as~$n\theta_n$.

(ii)  Set~$x_n=n^\frac{2}{2a-1}g(n)$, where~$g$ is another slowly varying function. Using basic properties of slowly varying functions, we obtain~$\theta_n\sim \frac{\ell(x_n)x_n^{-a+1}}{a-1}$ and thus
\begin{equation}
    n\theta_n\sim \frac{1}{a-1}\ell(n^\frac{2}{2a-1}g(n))g^{1-a}(n)n^\frac{1}{2a-1}, \quad n\to\infty,
\end{equation}
and
\begin{equation}
    n^2\lambda_{x_n}\theta_n\sim \frac{1}{a-1}\ell^2(n^\frac{2}{2a-1}g(n))g^{1-2a}(n), \quad n\to\infty,
\end{equation}
where
\begin{equation}
    h(x):=\frac{1}{a-1}\ell(x^\frac{2}{2a-1}g(x))g^{1-a}(x)
\end{equation}
is again slowly varying. By suitable choice of~$g$ we can ensure that~\eqref{eq:n-theta-lambda-0} holds. The rest of the proof follows from the arguments in (i).

 \par (iii) Note that, for~$K_i=K_j$, it holds that
 \begin{equation}
  ||X_i-X_j||^2=(A_i-A_j)^2,
 \end{equation}
while for~$K_i\neq K_j$, it holds that
\begin{equation}
 ||X_i-X_j||^2\geq~||X_i||^2=A_i^2.
\end{equation}
This implies that the only~$i\in\mathcal M_{k,n}$, for which~$N(i)\notin \mathcal M_{k,n}$ is possible, are
\begin{equation}
 i= \arg\min\lbrace A_j\colon j\in \mathcal M_{k,n}, A_j\geq~0\rbrace\quad\text{or}\quad i=\arg\max\lbrace A_j\in \mathcal M_{k,n},A_j\leq~0\rbrace.
\end{equation}
Moreover, there can be at most two indices~$j\in M_{K_i,n}$ such that~$N(j)=i$, since this implies that~$A_i$ is the nearest neighbor of~$A_j$ among~$\{ A_l:l\in\mathcal M_{K_i,n}\}$ and the scalar~$A_i$ can be the nearest neighbor to at most two other scalars in any set. Combining these observations, we obtain that
\begin{equation}\label{eq:ln-upper-bound}
 L_n\leq~2+2x_n+|\mathcal R_n|.
\end{equation}
Now choosing~$x_n$ as in (ii) and observing that~$P(|\mathcal{R}_n|\leq~2n\theta_n)\to 1$, we use \eqref{eq:ln-upper-bound} and thus obtain~$L_n=~o_P(n^{1/4})$.

\end{proof}

\proofreference{\subsection{Proof of Theorem~\ref{thm:normality}}}{\subsection{Proof of Theorem~2}}

For the remainder of this section, assume~$X_i$ and~$Y_i$ are independent.

To show the asymptotic normality of the test statistic~$\mathcal I_n$, we will first consider the quantity~$\widetilde Q_n$. We will use Lemma~\ref{thm:rinott}, a central limit theorem for graph-based statistics under local dependence, to show that it is asymptotically normally distributed. In a second step, we will replace the theoretically simpler but practically infeasible quantity~$\widetilde Q_n$ by the statistic~$\widehat Q_n$. \cite{deb_measuring_2020} and \cite{shi_power_2022} in essence use a similar strategy, but they consider a stochastic centering term and then look for a H\`ajek representation. The intermediate steps in our proof are very different from theirs, and it is unclear whether their results can be adapted to yield a central limit theorem for polynomially growing degree of~$L_n$. In comparison to \cite{deb_measuring_2020}, fixing their kernel function~$K$ to be the~$\min$-function allows us to calculate the variance of~$\widehat Q_n$ given $X$'s exactly (and efficiently) instead of estimating it.

 Let us proceed to show asymptotic normality of~$\widetilde Q_n$. To this end, we will use the following result by \cite{rinott_normal_1994}.

\begin{Lemma}\label{thm:rinott}[\cite{rinott_normal_1994}, Thm. 2.2]
 Let~$V_1,\ldots,V_n$ be random variables having a dependence graph whose maximal degree is strictly less than~$k$, satisfying
 \begin{itemize}
 \item~$|V_i-\ev(V_i)|\leq~B$ almost surely for~$i=1,\ldots,n$,
 \item~$\ev(\sum_{i=1}^n V_i)=0$ and
 \item~$\Var(\sum_{i=1}^n V_i)=1$.
 \end{itemize}
 Then
 \begin{equation}\label{eq:rinott}
 \left|P\left(\sum_{i=1} V_i\leq~z\right)-\Phi(z)\right|\leq~\sqrt{\frac{1}{2\pi}}kB+16\sqrt{n}k^{3/2}B^2+10nk^2B^3.
 \end{equation}
\end{Lemma}

\begin{Lemma}\label{lem:q-tilde-variance}
Let~$\mathcal{G}_n$ be the nearest neighbor graph generated by vertices~$X_1,\ldots,X_n$. Then
 \begin{equation}
 \widetilde W_n:=\Var(\sqrt{n}\,\widetilde Q_n|\mathcal{G}_n)=\frac{1}{n}\left(\frac{W_{n,1}}{45}+\frac{W_{n,2}}{18}\right),
 \end{equation}
where
\begin{equation} 
W_{n,1}=n+\sum_{i=1}^n L_{i,n}^2-2f_n\quad \text{and} \quad W_{n,2}=n+f_n,
\end{equation}
and where~$f_n$ is the number of indices~$i\in\{1,\ldots,n\}$, with~$N(N(i))=i$. We have~$ \widetilde W_n>\frac{4}{45}$.
\end{Lemma}
\begin{proof}
Define~$U_i=F(Y_i)$ and notice that the~$U_i$ are i.i.d.\ uniform on~$[0,1]$. We have
\begin{align}
 \widetilde W_n&=\Var\left(\frac{1}{\sqrt{n}}\sum_{i=1}^n\min\lbrace U_i,U_{N(i)}\rbrace\middle|\mathcal{G}_n\right)\nonumber\\
 &=\frac{1}{n}\sum_{i,j=1}^n \Cov\left(\min\lbrace U_i,U_{N(i)}\rbrace,\min\lbrace U_j,U_{N(j)}\rbrace|\mathcal{G}_n\right).\label{e:Uis}
\end{align}
We partition~$I_n=\{1,\ldots, n\}^2$ into the three~$\mathcal{G}_n$-measurable sets
\begin{equation}
I_n^{(k)}:=\{(i,j)\in I_n\colon |\lbrace j,N(j)\rbrace\cap\lbrace i, N(i)\rbrace|=k\},\quad k=0,1,2.
\end{equation}
Independence of~$X$ and~$Y$ implies that distinct variables among~$U_i, U_{N(i)}, U_j, U_{N(j)}$ in \eqref{e:Uis} are i.i.d. Hence, we get
\begin{align}\begin{split}
 \widetilde W_n&=\frac{|I_n^{(1)}|}{n}\Cov\left(\min\lbrace U_1,U_2\rbrace),\min\lbrace U_1,U_3\rbrace\right)+
 \frac{|I_n^{(2)}|}{n}\Var\left(\min\lbrace U_1,U_2\rbrace\right)\\
 &=\frac{1}{n}\left(\frac{|I_n^{(1)}|}{45}+\frac{|I_n^{(2)}|}{18}\right).
\end{split}\end{align}
We need to show that~$I_n^{(k)}=W_{n,k}$,~$k=1,2$. For~$(i,j)$ to be in~$I_n^{(2)}$ we have two options: either~${j=i}$ ($n$ cases) or~${j\neq i}$ and~${N(i)=j, N(j)=i}$. The latter is equivalent to~${N(N(i))=i}$ ($f_n$ cases) and hence~${I_n^{(2)}=W_{n,2}}$ follows. 

If~${(i,j)\in I_n^{(1)}}$ then~${j\neq i}$ and we have two different possibilities. Option (a) is to have~${N(i)=j}$ and~${N(j)\neq i}$ or the other way around. 
Option (b) is that~${N(i)=N(j)}$. The number of cases in (a) is equal to 
\begin{align}\begin{split}
W_{n,11}=2\times\sum_{i=1}^n\sum_{j=1}^n 1\{N(i)=j\}1\{N(j)\neq i\} 1\{i\neq j\}.
\end{split}\end{align}
Observe that for given~$i$, 
\begin{equation}
    \sum_{j=1}^n 1\{N(i)=j\}1\{N(j)\neq i\} 1\{i\neq j\}=1-~\!1\{N(N(i))=i\},
\end{equation}
and hence~$W_{n,11}=2(n-f_n)$. The number of cases in (b) is equal to
\begin{align}\begin{split}
W_{n,12}&=\sum_{i=1}^n\sum_{j=1}^n 1\{N(i)=N(j)\}1\{i\neq j\}\\
&=\sum_{i=1}^n\sum_{j=1}^n (1\{N(i)=N(j)\}-1)=\sum_{i=1}^n|\{j\colon N(j)=N(i)\}|-n=\sum_{i=1}^nL_{i,n}^2-n.
\end{split}\end{align}

As for the lower bound for~$\widetilde W_n$ note first that~$-\frac{2}{45}f_n+\frac{1}{18}f_n$ is increasing in~$f_n$. Hence, the lowest thinkable contribution coming from~$f_n$ is when~$f_n=0$. In addition, note that we have~$\sum_{i=1}^n L_{i,n}^2\geq~n$. Hence,~$\widetilde W_n\geq~\frac{2}{45}+\frac{1}{18}=\frac{1}{10}$.
\end{proof}

In the sequel, we have the following data generating mechanism in mind: 
\begin{enumerate}
\item We sample the~$X_i$'s and generate~$\mathcal{G}_n$.
\item Independent of~$\mathcal{G}_n$, we generate a random sample~$\widetilde Y_i$ with distribution function~$F$. 
\item We choose a random permutation~$\pi_n$ of~$(1,\ldots, n)$ which is independent of everything and set~$(Y_1,\ldots, Y_n)=\pi_n(\widetilde Y_{(1)},\ldots, \widetilde Y_{(n)})$, where~$\widetilde Y_{(1)}\leq~\cdots\leq~\widetilde Y_{(n)}$ denotes the ordered sample. 
\end{enumerate}
Below, we shall condition on~$\mathcal{G}_n=g$ and sometimes express dependence on~$g$ for different variables explicitly, e.g.\ by writing~$\widetilde W_n(g)$ or~$\widetilde Q_n(g)$. We let~$\mathrm{deg}(g)$ be the maximal degree of~$g$.

\begin{Proposition}\label{prop:q-tilde-normal}
Let~$\mathcal{K}_n$,~$n\geq~1$, be a collection of nearest neighbor graphs~$g$ with~$n$ vertices, such that~$\sup_{g\in \mathcal{K}_n}\mathrm{deg}(g)=o(n^{1/4})$. Then there is a null-sequence~$\psi_n$ such that

\begin{equation}
\sup_{g\in \mathcal{K}_n}\sup_{z\in\R}\left|P\left(\sqrt{n}\,\frac{\widetilde Q_n-\frac{1}{3}}{\sqrt{\widetilde W_n}}\leq~z\middle| \mathcal{G}_n=g\right)-\Phi(z)\right|\leq~\psi_n.
\end{equation}

\end{Proposition}
\begin{proof} We may write
\begin{equation}
\sqrt{\frac{n}{\widetilde W_n}}\left(\widetilde Q_n-\frac{1}{3}\right)=\sum_{i=1}^n V_i,
\end{equation}
where
\begin{equation}
    V_i=\frac{1}{\sqrt{n \widetilde W_n}}\left(\min\{F(Y_i),F(Y_{N(i)})\}-\frac{1}{3}\right).
\end{equation}
With this definition, we can infer from Lemma~\ref{lem:q-tilde-variance} that
\begin{equation}
    {|V_i|\leq~\frac{2}{3}\frac{1}{\sqrt{n \widetilde W_n}}<~\sqrt{5/n}}.
\end{equation}
Moreover, it holds that~${\Var(\sum_{i=1}^n V_i|\mathcal{G}_n=g)=1}$, while~${E(V_i|\mathcal{G}_n=g)=0}$. Hence, the variables satisfy the conditions of Lemma~\ref{thm:rinott} with~${B=\sqrt{5/n}}$.
The maximum degree of the dependence graph of the~$Y_i$ is bounded by~$2L_n$. To see this, fix~$i$. There is an edge between~$Y_i$ and~$Y_j$ with~${j\neq i}$ if either~${j=N(i)}$ (one case) or if~${N(j)=i}$ ($\leq~L_n$ cases) or if~${N(j)=N(i)}$ ($\leq L_n-1$ cases). Conditional on~${\mathcal{G}_n=g}$ with~${g\in\mathcal{K}_n}$ we have~${k=~2L_n(g)=o(n^{1/4})}$. 
\end{proof}

\begin{Proposition}\label{prop:difference-normal}
Under the setting of Proposition~\ref{prop:q-tilde-normal}, there exists a sequence~$(A_n)_{n\in\N}$ of~$N(0,4/45)$ distributed random variables such that for every~$\varepsilon>0$ there exists a null-sequence~$(\eta_n(\varepsilon))$ such that,
 \begin{equation}
 \sup_{g\in \mathcal{K}_n}P\left(|\sqrt{n}(\widehat Q_n-\widetilde Q_n)-A_n|>\varepsilon|\mathcal{G}_n=g \right)\leq~\eta_n(\varepsilon).
 \end{equation}
\end{Proposition}
\begin{proof}
As presented, for example, by \cite{van_der_vaart_asymptotic_1998}, we can define the variables~$Y_i$ on a common probability space~$(\Omega,\mathcal{A},P)$ on which there exists a sequence~$(B_n)_{n\in\N}$ of Brownian bridges such that
\begin{equation}\label{eq:brownian-approximation}
 \limsup_{n\to\infty} \frac{\sqrt{n}}{\log^2 n}\sup_{t\in\mathbb{R}}|\sqrt{n}(F_n(t)-F(t))-B_n(F(t))| <\infty \quad\text{a.s.}
\end{equation}
Now define
\begin{equation}
    A_n:=\int_0^1 (B_n\cdot h)(t)dt,
\end{equation}
 with~$h(t)=(2-2t)1\{t\in [0,1]\}$ and for some positive integer sequence~$m_n\to\infty$,~$m_n=O(n^{3/4})$, and any index~$j\in\lbrace 1,\ldots,m_n\rbrace$, we define the interval~$I_j^{m_n}=~((j-1)/m_n,j/m_n]$.
Further, we set~${Z_i=\min\lbrace F(Y_i),F(Y_{N(i)})\rbrace}$ and note that the~$Z_i$ are identically distributed with density function~$h$. We then decompose
\begin{equation}
 \sqrt{n}(\widehat Q_n-\widetilde Q_n)-A_n=D_n^1+D_n^2+D_n^3+D_n^4+D_n^5,
\end{equation}
where
\begin{align}\begin{split}
 &D_n^1=\sqrt{n}(\widehat Q_n-\widetilde Q_n)-\frac{1}{n}\sum_{i=1}^n B_n(Z_i),\\
 &D_n^2=\frac{1}{n}\sum_{i=1}^n B_n(Z_i)-\sum_{j=1}^{m_n}B_n(j/m_n)\cdot\frac{|\lbrace Z_1,\ldots,Z_n\rbrace\cap I_j^{m_n}|}{n},\\
 &D_n^3=\sum_{j=1}^{m_n}B_n(j/m_n)\cdot\frac{|\lbrace Z_1,\ldots,Z_n\rbrace\cap I_j^{m_n}|}{n}-\sum_{j=1}^{m_n}B_n(j/m_n)\int_{I_j^{m_n}}h(t)dt,\\
 &D_n^4=\sum_{j=1}^{m_n}B_n(j/m_n)\int_{I_j^{m_n}}h(t)dt-\sum_{j=1}^{m_n}\frac{1}{m_n}B_n(j/m_n)h(j/m_n),\\
 &D_n^5=\sum_{j=1}^{m_n}\frac{1}{m_n}B_n(j/m_n)h(j/m_n)-\int_0^1 (B_n\cdot h)(t)dt.
\end{split}\end{align}

We will find bounds for the quantities~$|D_n^i|$ and use the triangular inequality. 
Denote by~${\|\cdot\|_\infty}$ the sup-norm on~$L^2(\lbrack0,1\rbrack)$. Then \eqref{eq:brownian-approximation} implies that 
\begin{equation}
P\left(|D_n^1|>\frac{\log^3 n}{\sqrt{n}}\right)\\
 \leq~P\left(\|\sqrt{n}(F_n-F)-B_n\circ F\|_\infty>\frac{\log^3 n}{\sqrt{n}}\right)\to 0.
\end{equation}

Let~$C>\sqrt{2}$ and~$\alpha<1/2$. Theorem~1.1.1 in \cite{csorgo_strong_1981}, implies that there is a null-sequence~$(\delta_k)$ such that
\begin{equation}\label{e:holderB}
P\left(\sup_{0\leq~t\leq~1-1/k}\sup_{0\leq~s\leq~1/k}|B_n(t+s)-B_n(t)|> Ck^{-\alpha}\right)\leq\delta_k, \quad k\geq~1.
\end{equation}
Since the~$B_n$ are identically distributed, this inequality holds uniformly in~$n$. 
Therefore, with probability~$1-\delta_{m_n}$ we get that 
\begin{align}\begin{split}
 |D_n^2|\leq\frac{1}{n}\sum_{j=1}^{m_n}\sum_{i=1}^n |B_n(Z_i)-B_n(j/m_n)| 1\lbrace Z_i\in I_j^{m_n}\rbrace\leq~C m_n^{-\alpha}.
\end{split}\end{align}

We note that the bounds derived for~$|D_n^1|$ and~$|D_n^2|$ do not depend on the nearest neighbor graph~$\mathcal{G}_n$. Quantities~$|D_n^4|$ and~$|D_n^5|$ do not involve the nearest neighbor graph at all. This will be only relevant for bounding~$|D_n^3|$. Here we have
\begin{align}\begin{split}
 |D_n^3|\leq~\sup_{t\in[0,1]} |B_n(t)| \times\sum_{j=1}^{m_n}\left|\frac{|\lbrace Z_1,\ldots,Z_n\rbrace\cap I_j^{m_n}|}{n}-\int_{I_j^{m_n}}h(t)dt\right|.
\end{split}\end{align}
We have that~$\sup_{t\in[0,1]} |B_n(t)|=O_P(1)$ and thus it remains to show that the sum on the right above converges to zero in probability, uniformly for~$g\in \mathcal{K}_n$.
We remark that~
\begin{equation}
    \ev\left(\frac{1}{n}|\lbrace Z_1,\ldots,Z_n\rbrace\cap I_j^{m_n}| \right)=\int_{I_j^{m_n}}h(t)dt.
\end{equation}
Therefore, by Chebyshev's inequality 
\begin{align}
&P\left(\sum_{j=1}^{m_n}\left|\frac{|\lbrace Z_1,\ldots,Z_n\rbrace\cap I_j^{m_n}|}{n}-\int_{I_j^{m_n}}h(t)dt\right|>\delta\right)\nonumber\\
 &\leq~\frac{1}{\delta} \sum_{j=1}^{m_n}\ev \left|\frac{|\lbrace Z_1,\ldots,Z_n\rbrace\cap I_j^{m_n}|}{n}-\int_{I_j^{m_n}}h(t)dt\right|\nonumber
 \\&\leq~\frac{1}{\delta} \sum_{j=1}^{m_n}\mathrm{Var}^{1/2}\left(\frac{1}{n}\sum_{i=1}^n 1\{Z_i\in I_j^{m_n}\}\right)\nonumber
 \\&\leq~\frac{1}{n\delta}\sum_{j=1}^{m_n}\left(\sum_{i,k=1}^n\mathrm{Cov}(1\{Z_i\in I_j^{m_n}\},1\{Z_k\in I_j^{m_n}\})\right)^{1/2}.\label{e:lastbound}
\end{align}
In the proof of Proposition~\ref{prop:q-tilde-normal} we showed that conditional on~$\mathcal{G}_n=g$, the dependence graph of the variables~$V_i$ (which is the same as the dependence graph of the~$Z_i$) has maximal degree~$2L_n(g)$. Hence, conditional on~$\mathcal{G}_n=g$, we obtain by the Cauchy-Schwarz inequality and the fact the~$Z_i$ are identically distributed, that uniformly on~$\mathcal{K}_n$
\begin{align}\begin{split}
& \sum_{i,k=1}^n\left|\mathrm{Cov}(1\{Z_i\in I_j^{m_n}\},1\{Z_k\in I_j^{m_n}\}\right|\leq~n\times \mathrm{deg}(g) P(Z_1\in I_j^{m_n})=o\left(n^{5/4}/m_n\right).
\end{split}\end{align}
Inserting into \eqref{e:lastbound} shows that~$|D_n^3|$ converges to zero in probability, uniformly on~$\mathcal{K}_n$.
\par By similar arguments we obtain
$
|D_n^4| \leq~2\sup_{t\in[0,1]} |B_n(t)|/m_n =O_P(m_n^{-1})
$
and 
\begin{align}\begin{split}
|D_n^5|&\leq~2\sup_{t\in[0,1]} |B_n(t)|/m_n +2\sum_{j=1}^{m_n}\int_{I_j^{m_n}}|B_n(j/m_n)-B_n(t)|dt=O_P(m_n^{-1/2}).
 \end{split}\end{align}
In the last step, we used the basic fact that~$\int_{I_j^{m_n}}E|B_n(j/m_n)-B_n(t)|dt=\sqrt{\frac{2}{\pi}}\frac{2}{3}m_n^{-3/2}$.
\par It remains to show that~$A_n$ follows an~$N(0,4/45)$ distribution. Since~$B_n$ is a zero mean Gaussian process, it is clear that~$A_n$ is normally distributed with mean zero. To calculate the variance, recall that~$h(t)=2-2t$ on~$\lbrack 0,1\rbrack$ and consider
\begin{align}\begin{split}
 \Var\left(\int_0^1 B_n(t)h(t)dt\right)&=\ev\left(\left(\int_0^1B_n(t)h(t)dt\right)^2\right)\\
 &=\ev\left(\int_0^1 B_n(t)h(t)dt\int_0^1 B_n(s)h(s)ds\right)\\
 &=\int_0^1\int_0^1\ev(B_n(t)B_n(s))h(t)h(s) ds dt\\
 &=\int_0^1\int_0^1(\min\lbrace s,t\rbrace -st)h(s)h(t)ds dt=\frac{4}{45}.
\end{split}\end{align}
\end{proof}

\begin{Lemma}
 Define~$W_{n,1}$ and~$W_{n,2}$ as in Lemma~\ref{lem:q-tilde-variance}. The conditional variance of~$\widehat Q_n$ given~$\mathcal{G}_n$ is
 \begin{equation}
 W_n:= \Var(\sqrt{n}\widehat Q_n|\mathcal G_n)=\frac{W_{n,1} v_1+W_{n,2} v_2+(n^2-W_{n,1}-W_{n,2})v_0}{n},\label{e:Wn}
 \end{equation}
where
\begin{align}\begin{split}
v_0=-\frac{4(n+1)}{45n^2},\quad
 v_1=\frac{4n^4-25n^3+30n^2+25n^1-34}{180n^2(n-1)(n-2)},\quad
v_2=\frac{n^2-n-2}{18n^2}.
\end{split}\end{align}
\end{Lemma}
\begin{proof}
As in the proof of Lemma~\ref{lem:q-tilde-variance} we partition~$I_n$ into~$I_n^{(k)}$,~$k=0,1,2$ . The covariance of pairs~$U_i^\prime,U_j^\prime$, where~$U_i^\prime=\min\{F_n(Y_i),F_n(Y_{N(i)})\}$, for~$(i,j)\in I_n^{(0)}$ is given by
 \begin{equation}
 \frac{1}{n(n-1)(n-2)(n-3)}\sum\limits_{\substack{i,j,k,l=1 \\ \text{all distinct}}}^n \left(\min\left\lbrace \frac{i}{n}, \frac{k}{n}\right\rbrace-\frac{n+1}{3n}\right)\left(\min\left\lbrace \frac{j}{n}, \frac{l}{n}\right\rbrace-\frac{n+1}{3n}\right)=v_0,
 \end{equation}
 for~$(i,j)\in I_n^{(1)}$ by
 \begin{equation}
 \frac{1}{n(n-1)(n-2)}\sum\limits_{\substack{i,j,k=1 \\ \text{all distinct}}}^n \left(\min\left\lbrace \frac{i}{n}, \frac{k}{n}\right\rbrace-\frac{n+1}{3n}\right)\left(\min\left\lbrace \frac{j}{n}, \frac{k}{n}\right\rbrace-\frac{n+1}{3n}\right)=v_1,
 \end{equation}
 and for~$(i,j)\in I_n^{(2)}$ by
 \begin{equation}
 \frac{1}{n(n-1)}\sum\limits_{\substack{i,j=1 \\ \text{distinct}}}^n \left(\min\left\lbrace \frac{i}{n}, \frac{j}{n}\right\rbrace-\frac{n+1}{3n}\right)^2=v_2,
 \end{equation}
 where we used that~$\ev(\min\lbrace F_n(Y_i),F_n(Y_{N(i)})\rbrace)=(n+1)/3n$.
\end{proof}

Note that under Assumption~\proofreference{\ref{ass2}}{2},
\begin{equation}
    \frac{W_n}{\widetilde W_n-\frac{4}{45}}\to 1\quad \text{for}\quad n\to\infty,
\end{equation}
since
\begin{equation}
    \frac{v_1}{n}\sim~\frac{1}{45n},\quad \frac{v_2}{n}\sim~\frac{1}{18n}\quad\text{and}\quad {n^2-W_{n,1}-W_{n,2}\sim n^2}.
\end{equation}

\begin{Proposition}
 Suppose that Assumption~\proofreference{\ref{ass2}}{2} holds. Then, for any~$z\in\R$,
 \begin{equation}
 \left|P\left(\sqrt{n}\frac{\widehat Q_n-\frac{1}{3}}{W_n}\leq~z\right)-\Phi(z)\right|\to 0,
 \end{equation}
as~$n\to\infty$.
\end{Proposition}
\begin{proof}
The construction of the Brownian bridges in \eqref{eq:brownian-approximation} only involves~$F_n$ which is a function of the order statistics~$Y_{(i)}$ and thus~$A_n$ can be assumed to be independent of~$\mathcal{G}_n$ and~$\pi_n$. In contrast, we may represent
\begin{equation}
    \widehat Q_n=\frac{1}{n^2}\sum_{i=1}^n \pi_n(i)\wedge \pi_n(N(i))
\end{equation}
and hence it is measurable with respect to~$\mathcal{G}_n$ and~$\pi_n$. It follows that~$\widehat Q_n$ and~$A_n$ are independent. 

Due to Proposition~\ref{prop:difference-normal} and the fact that~$\widetilde{W}_n\geq~\frac{1}{10}$
\begin{equation}
\sup_{g\in\mathcal{K}_n}P\left(\frac{1}{\sqrt{\widetilde{W}_n}}\left|\sqrt{n}(\widetilde Q_n-\frac{1}{3})-\sqrt{n}\left(\widehat Q_n-\frac{1}{3}\right)-A_n\right|>\varepsilon\middle| \mathcal{G}_n=g\right)\leq~\eta_n(\varepsilon/\sqrt{10}).
\end{equation}
Thus, we have for any~$z\in\mathbb{R}$
\begin{align}\begin{split}
&\sup_{g\in\mathcal{K}_n}P\left(\sqrt{\frac{1}{\widetilde{W}_n}}\left(\sqrt{n}\left(\widehat Q_n-\frac{1}{3}\right)+A_n\right)\leq~z\middle| \mathcal{G}_n=g\right)\\
&\leq~\sup_{g\in\mathcal{K}_n}P\left(\sqrt{\frac{n}{\widetilde{W}_n}}\left(\widetilde Q_n-\frac{1}{3}\right)\leq~z+\varepsilon\middle| \mathcal{G}_n=g\right)+\eta_n(\varepsilon/\sqrt{10})\\
&\leq~\Phi(z+\varepsilon)+\psi_n+\eta_n(\varepsilon/\sqrt{10})\leq~\Phi(z)+\frac{\varepsilon}{\sqrt{2\pi}}+\psi_n+\eta_n(\varepsilon/\sqrt{10}).
\end{split}\end{align}
An analogous lower bound can be established, yielding that
\begin{equation}
\sup_{g\in\mathcal{K}_n}\sup_{z\in\mathbb{R}}\left|P\left(\sqrt{\frac{1}{\widetilde{W}_n}}\left(\sqrt{n}\left(\widehat Q_n-\frac{1}{3}\right)+A_n\right)\leq~z\middle| \mathcal{G}_n=g\right)-\Phi(z)\right|=o(1).
\end{equation}
Note that conditionally on~$\mathcal{G}_n=g$,~$\widetilde W_n$ is deterministic and
\begin{equation}
    S_{2,n}:=\sqrt{\frac{1}{\widetilde{W}_n}}A_n\sim~N\left(0,\frac{1}{\widetilde{W}_n}\frac{4}{45}\right)
\end{equation}
is independent of
\begin{equation}
    S_{1,n}:=\sqrt{\frac{n}{\widetilde{W}_n}}(\widehat Q_n-\frac{1}{3}).
\end{equation}
For the rest of the proof, we use conditional characteristic functions. For some~$A\in\mathcal{A}$ we denote~$\varphi_{X|A}(t)=E[e^{\ii tX}|A]$,~$t\in\mathbb{R}$ and~$\ii =\sqrt{-1}$. 

It follows from our previous derivations that for any~$t\in\mathbb{R}$
\begin{equation}\label{e:char}
\sup_{g\in\mathcal{K}_n}\left|\varphi_{S_{1,n}|\mathcal{G}_n=g}(t)e^{-\frac{1}{\widetilde W_n(g)}\frac{2}{45}t^2}-e^{-\frac{1}{2}t^2}\right|\to 0.
\end{equation}
In fact, it can be shown that the convergence in \eqref{e:char} holds uniformly on any compact interval (see e.g.\ Theorem~2.66 in \cite{jeffreys_theory_2003}). 
The lower bound~$\widetilde W_n(g)\geq~1/10$ implicates that uniformly for~$t$ in any compact interval
\begin{align}\begin{split}
&\sup_{g\in\mathcal{K}_n}\left|\varphi_{S_{1,n}|\mathcal{G}_n=g}(t)-e^{-\frac{1}{2}\left( 1-\frac{1}{\widetilde W_n(g)}\frac{4}{45}\right)t^2}\right|\to 0.
\end{split}\end{align}
Define
\begin{equation}
    R_n:=\left( 1-\frac{1}{\widetilde W_n}\frac{4}{45}\right)
\end{equation}
and observe that~$1/\sqrt{R_n}\in [1,3]$. Hence, if~$t$ is fixed, then~$t/\sqrt{R_n}$ remains in a compact interval. Thus,
\begin{align}\begin{split}
&\sup_{g\in\mathcal{K}_n}\left|\varphi_{S_{1,n}|\mathcal{G}_n=g}(t/\sqrt{R_n(g)})-e^{-\frac{1}{2}t^2}\right|=\sup_{g\in\mathcal{K}_n}\left|\varphi_{S_{1,n}/\sqrt{R_n(g)}|\mathcal{G}_n=g}(t)-e^{-\frac{1}{2}t^2}\right|\to 0.
\end{split}\end{align}
Now note that~$S_{1,n}/\sqrt{R_n}=\sqrt{\frac{n}{\widetilde W_n-4/45}}\left(\widehat Q_n-\frac{1}{3}\right)$. We just proved that for any~$z\in\mathbb{R}$
\begin{equation}
\left|P\left(\sqrt{\frac{n}{\widetilde W_n-4/45}}\left(\widehat Q_n-\frac{1}{3}\right)\leq~z\middle| \mathcal{G}_n\in\mathcal{K}_n\right)-\Phi(z)\right|\to 0.
\end{equation}
The proof follows by observing that under Assumption~\proofreference{\ref{ass2}}{2} we have that~$P(\mathcal{G}_n\in\mathcal{K}_n)\to 1$, and that~${W_n\sim \widetilde W_n-\frac{4}{45}}$,~${n\to\infty}$.
\end{proof}

\proofreference{\subsection{Proof of Theorem~\ref{thm:ass-2-fulfilled}}}{\subsection{Proof of Theorem~4}}\label{s:ass1}

We define~$X_i^q$ as the projection of~$X_i$ onto the space spanned the first~$q$ eigenfunctions, i.e.
\begin{equation}
    X_1^q=\sum_{k=1}^q\langle X_1,e_k\rangle e_k.
\end{equation}
In analogy to~$L_{1,n}$ define~$L_{i,n}^q$  as the corresponding numbers we get with~$X_1,\ldots, X_n$ replaced by~$X_1^q,\ldots, X_n^q$. It is elementary that 
\begin{align}
P(L_{n}>k)& \leq~P(L_{n}^q>k)+P(L_{n}> L_{n}^q)\leq~P(L_{n}^q>k)+nP(L_{1,n}> L_{1,n}^q).\label{eq:nb-prob}
\end{align}

For bounding the first term in \eqref{eq:nb-prob} we will use the multivariate bound of \cite{kabatjanski_bounds_1978} and for the second the following result:
\begin{Proposition}\label{p:Ln} Under Assumption~\proofreference{\ref{ass:density}}{3} we have
\begin{equation}
    P(L_{1,n}> L_{1,n}^q)\leq\mathrm{const}\times n^2 \,\left(\sum_{j>q}\lambda_j\right)^{1/2},
\end{equation}
where the constant is independent of~$n$ and~$q$.
\end{Proposition}

For the proof of Proposition~\ref{p:Ln} we need a few preparatory lemmas. In the sequel, we consider independent random variables~$Z_1$,~$Z_2$ and~$Z$ having bounded densities~$f_1(s)$,~$f_2(s)$ and~$f(s)$.
\begin{Lemma}\label{l:1}
Let~$f_{1+2}(s)$ be the density function of~$Z_1+Z_2$. Then
\begin{equation}
\sup_s f_{1+2}(s)\leq~\min\{\sup_s f_1(s),\sup_s f_2(s)\}.
\end{equation}
\end{Lemma}
\begin{proof}
We have~$f_{1+2}(s)=\int_{-\infty}^\infty f_1(s-t)f_2(t)dt$. Now factor out~$\sup_{t} f_1(s-t)$ or~$\sup_{t} f_2(t)$.
\end{proof}

The next lemma is also elementary.
\begin{Lemma}\label{l:2}
Let~$c\in \mathbb{R}$ be a constant. Then~$(Z-c)^2$ has density
\begin{equation}
\frac{1}{2\sqrt{s}}\big(f(c+\sqrt{s})+f(c-\sqrt{s})\big)I\{s> 0\}.
\end{equation}
\end{Lemma}

\begin{Lemma}\label{l:3}
Let~$c_1, c_2\in \mathbb{R}$ be constants. The density of~$(Z_1-c_1)^2+(Z_2-c_2)^2$ is bounded from above by\,~$\pi\times \sup_s f_1(s)\times\sup_s f_2(s)$.
\end{Lemma}

\begin{proof}
Let~$g_1$ and~$g_2$ be the densities of~$(Z_1-c_1)^2$ and~$(Z_2-c_2)^2$, respectively. Then the density of~$(Z_1-c_1)^2+(Z_2-c_2)^2$ is given as
\begin{align}\begin{split}
g(s)&=\int_{-\infty}^\infty g_1(s-t)g_2(t)dt\\
&=\int_{0}^s\!\frac{1}{4\sqrt{s-t}\sqrt{t}}\big(f_1(c_1\!+\!\sqrt{s-t})+f_1(c_1-\sqrt{s-t})\big)\big(f_2(c_2+\sqrt{t})+f_2(c_2-\sqrt{t})\big)dt\\
&\leq~\int_{0}^s\frac{1}{\sqrt{s-t}\sqrt{t}}dt\times \sup_s f_1(s)\sup_s f_2(s)=\pi\times\sup_s f_1(s)\times\sup_s f_2(s).
\end{split}\end{align}
\end{proof}

Let~$X_{N(1|q)}^q$ and~$X_{M(1|q)}^q$ be the nearest and second-nearest neighbor of~$X_1^q$, respectively, among the observations~$X_2^q,\ldots, X_n^q$.
\begin{Lemma}\label{l:4} Under Assumption~\proofreference{\ref{ass:density}}{3} we have
\begin{equation}
P\left(\|X_1^q-X_{M(1|q)}^q\|^2-\|X_1^q-X_{N(1|q)}^q\|^2<\epsilon\right)\leq~\mathrm{const}\times n^2\epsilon,
\end{equation}
where the constant is independent of~$n$,~$q$ and~$\epsilon$.
\end{Lemma}
\begin{proof}
Let us first obtain a bound conditional on~$X_1^q=x$. Conditional on this event, we have that~${\|X_1^q-X_k^q\|^2\stackrel{iid}{\sim} \|x-X_k^q\|^2}$,~${k=2,\ldots, n}$. Letting~$V_k=V_k(x,q)=\|x-X_k^q\|^2$ and~$V_{(k)}$ being the corresponding order statistics with~$V_{(1)}\leq~\cdots\leq~V_{(n-1)}$, we have
\begin{equation}
P\left(\|X_1^q-X_{M(1|q)}^q\|^2-\|X_1^q-X_{N(1|q)}^q\|^2<\epsilon\big| X_1^q=x\right)=P(V_{(2)}-V_{(1)}<\epsilon).
\end{equation}

We will now derive the density of~$V_{(2)}-V_{(1)}$ and show that it is bounded by some constant depending only on~$n$,~$f_{1}$ and~$f_{2}$. 
To this end, note that Assumption~\proofreference{\ref{ass:density}}{3} assures that~$V_1(x,q)$ has a density function. This is because~$V_1(x,q)=\sum_{k=1}^q (c_k - Z_{1k})^2$, where~${c_k=\langle x,e_k\rangle}$. Let us denote this density by~$g_{x,q}(s)$ and the corresponding distribution function by~$G_{x,q}(s)$. Then by \cite{pyke_spacings_1965} the density of the spacing~$D:=V_{(2)}-V_{(1)}$ between the two smallest variables 
is given as
\begin{align}\begin{split}
f_D(s)&=(n-1)(n-2)\int_0^\infty (1-G_{x,q}(t+s))^{n-3}g_{x,q}(t)g_{x,q}(t+s)dt\\
&\quad \leq~n^2\sup_{t\geq~0} g_{x,q}(t)\leq~n^2\sup_{t\geq~0} g_{x,2}(t)\leq~n^2\pi\sup_t f_{1}(t)\sup_t f_{2}(t).
\end{split}\end{align}
The latter two inequalities follow from Lemma~\ref{l:1} and Lemma~\ref{l:3}. It follows that 
\begin{equation}
P\left(\|X_1^q-X_{M(1|q)}^q\|^2-\|X_1^q-X_{N(1|q)}^q\|^2<\epsilon\big| X_1^q=x\right)\leq~\mathrm{const}\times n^2\epsilon, 
\end{equation}
where the constant is independent of~$n$,~$q$,~$\epsilon$ and~$x$.
\end{proof}

\begin{proof}[Proof of Proposition~\ref{p:Ln}]
We observe that
\begin{equation}
\{L_{1,n}> L_{1,n}^q\}\subset \{\exists i\in\{2,\ldots, n\}\colon \|X_i-X_{N(i|q)}\|^2>\|X_1-X_i\|^2\}.
\end{equation}
Thus, by reasons of symmetry, we have 
\begin{align}\begin{split}
P(L_{1,n}> L_{1,n}^q)&\leq~n P(\|X_1-X_{N(1|q)}\|^2>\|X_2-X_1\|^2).
\end{split}\end{align}
Denote by~$\Pi^q$ the projection onto the space spanned by~$\{e_{q+1},e_{q+2},\ldots\}$. Then by Pythagoras' theorem we have 
$
\|X_1-X_{N(1|q)}\|^2=\|X_1^q-X_{N(1|q)}^q\|^2+\|\Pi^q(X_1-X_{N(1|q)})\|^2.
$
Thus for any~$\epsilon>0$ it holds that
\begin{align}\begin{split}
&P(\|X_1-X_{N(1|q)}\|^2>\|X_1-X_2\|^2)\\
&=P(\|\Pi^q(X_1-X_{N(1|q)})\|^2>\|X_1-X_2\|^2-\|X_1^q-X_{N(1|q)}^q\|^2)\\
&\leq~P(\|\Pi^q(X_1-X_{N(1|q)})\|^2>\epsilon)+P(\|X_1-X_2\|^2-\|X_1^q-X_{N(1|q)}^q\|^2\leq~\epsilon)\\
&\leq~\epsilon^{-1}\times E\|\Pi^q(X_1-X_{N(1|q)})\|^2+P(\|X_1^q-X_2^q\|^2-\|X_1^q-X_{N(1|q)}^q\|^2\leq~\epsilon)\\
&\leq~\epsilon^{-1}\times\sum_{k>q}E\langle X_1-X_{N(1|q)},e_k\rangle^2+\mathrm{const}\times n^2\epsilon.
\end{split}\end{align}
We have used the Markov inequality and Lemma~\ref{l:4}. We remark that the independence of the scores~$\langle X_1,e_k\rangle$ for~$k\geq~1$, and the fact that~$N(1|q)$ is measurable with respect to the~$\sigma$-algebra generated by~$\{\langle X_i,e_k\rangle\colon 1\leq~i\leq~n,\, 1\leq~k\leq~q\}$ imply that~$\langle X_1,e_k\rangle~$ and~$\langle X_{N(1|q)},e_k\rangle$ are independent and have the same distribution if~$k>q$. Thus, 
\begin{equation}
\sum_{k>q}E\langle X_1-X_{N(1|q)},e_k\rangle^2=2\sum_{k>q}E\langle X_1,e_k\rangle^2=2\sum_{k>q}\lambda_k.
\end{equation}
We have shown that 
$
P(L_{1,n}> L_{1,n}^q)\leq~\mathrm{const}\times n \left(\epsilon^{-1}\times\sum_{k>q}\lambda_k+ n^2\epsilon\right).
$
The proof is concluded by setting~$\epsilon=~n^{-1}\times~\left(\sum_{k>q}\lambda_k\right)^{1/2}$.
\end{proof}

\begin{Proposition}\label{prop:general-condition-clt}
 Let Assumption~\proofreference{\ref{ass:density}}{3} hold and assume that we can choose~$q=q_n$ such that
 \begin{align}
  \sum_{j>q} \lambda_j &=o(n^{-6}),\label{eq:lambda-bound}\\
  P(L_n^{q}>n^{1/4})&=o(1).\label{eq:degree-growth} 
  \end{align}
Then Assumption~\proofreference{\ref{ass2}}{2} is fulfilled.
\end{Proposition}

\begin{proof}
The result is immediate from \eqref{eq:nb-prob} and Proposition~\ref{p:Ln}.
\end{proof}

\begin{proof}[Proof of Theorem~\proofreference{\ref{thm:ass-2-fulfilled}}{4}]
 As shown by \cite{kabatjanski_bounds_1978}, if~$n$ is large enough, it holds that~${P(L_n^{q}>n^{1/4})=0}$ for ${\gamma>2^{0.401}}$, and~$q_n\ <\frac{1}{4}\log_\gamma n$, a condition that is met for our choice of~$q_n$. Finally, condition \eqref{eq:lambda-bound} follows directly from our assumptions.
\end{proof}

\begin{Remark}
    We note that a much slower rate of decay for~$\lambda_j$ can be obtained if we could allow a bigger value for~$q$ in \eqref{eq:degree-growth}. While to the best of our knowledge very little is known about distributional properties of~$L_{n}^q$ (even for fixed~$q$), it seems certain that the exponential growth obtained from \cite{kabatjanski_bounds_1978} is way too conservative when applied to a random sample. This is confirmed in extensive simulations,  where we found~$L_n$ to grow at a relatively slow rate, even in cases of very slowly decaying eigenvalues, as, for example, for the Brownian motion. Based on the fact that nearest neighbors constitute a ``local'' property, we conjecture that for fixed~$q$ and~$k$, the probability~$P(L_n^q>k)$ does not depend on the distribution of the score vectors~$\Pi^q X$ as long as they have continuous densities.
\end{Remark}

%% file: acknowledgments.tex
The authors would like to thank Maximilian Ofner for helpful discussions and the compilation of the dataset for the real-world application example. We would also like to thank Norbert Henze and Mathew Penrose for providing insightful feedback.

%% file: funding.tex
This research was funded by the Austrian Science Fund (FWF) [P 35520]. For the purpose of open access, the authors have applied a CC BY public copyright license to any Author Accepted Manuscript version arising from this submission.

%% file: simulations.tex
\section{Simulations}\label{s:empirical}
In this section, we apply our dependence measure and the resulting test for independence~$\mathcal I_n$ to simulated data. In Subsection~\ref{ss:dep} we compute~$\widehat T_n$ in different scenarios and compare to the popular \emph{distance correlation}, whose empirical version is here denoted by~$\widehat R_n$ (see Definition~5 in \cite{szekely_measuring_2007}). We refer to \cite{dehling_distance_2020} for a detailed discussion of the applicability of distance correlation to discretized functional data. We follow this approach with~$p=200$ equidistant sampling points per curve. We will compare our independence test~$\mathcal{I}_n$ to permutation tests based on distance correlation~($\mathcal{I}^\mathrm{DC}_n$) and ball correlation ($\mathcal{I}^\mathrm{BC}_n$), as well as to an independence test for functional data developed by \cite{garcia-portugues_goodness--fit_2014}~($\mathcal{I}^\mathrm{CvM}_n$). The latter is a bootstrap test utilizing a Cram\'er-von-Mises type statistic. See Sections~\ref{ss:finsamp} and~\ref{ss:power}. We refer to \cite{szekely_measuring_2007}, \cite{pan_ball_2020} and \cite{garcia-portugues_goodness--fit_2014} for details on these competing approaches.

\par For our simulation study, we let~${X=\sum_{k=1}^{20} Z_ke_k(u)}$ where~$e_k(u)=\sqrt{2}\sin((k-1/2)\pi u)$ and~$Z_k\stackrel{\text{ind.}}{\sim} N(0, 0.3^k)$, modeling fast decaying eigenvalues (in fact~$\lambda_k=0$ for~$k>20$) and hence being in line with the assumptions in Theorem~\simulationreference{\ref{thm:ass-2-fulfilled}}{4}.

We have done our simulation experiments for sample sizes~$n=20$ (small)~$n=100$ (medium) and~$n=1000$ (large), but for the sake of space, not all results are presented here.

\par Our implementation of the test for independence is available via the R package~\blindedtext{\texttt{FDEP}}{} from the second author's GitHub profile\blindedtext{, \cite{strenger_fdep_2024}}{}. The competing statistics are implemented in the R-packages \texttt{energy} \citep{rizzo_energy_2022}, \texttt{Ball} \citep{zhu_ball_2023} and \texttt{fda.usc} \citep{febrero-bande_fdausc_2022}.

\subsection{Dependence measure}\label{ss:dep}
In our first simulation exercise, we mimic several types of relationships between~$X$ and~$Y$. For a given sample size we generate~$B=500$ samples to compute~$\widehat T_n$ and~$\widehat R_n$. The relation between~$Y$ and~$X$ is determined as~$Y=f(X)+ \varepsilon$, where~${\varepsilon\sim N(0,\sigma^2)}$ independent of~$X$. We chose~$\sigma^2$ such that~$r^2:={{\mathrm{Var}(f(X))}/\mathrm{Var}(Y)}$ is 100$\%$, 90$\%$, 50$\%$ and 10$\%$, respectively. The functions~$f:H\to \mathbb{R}$ are chosen as follows:
\begin{itemize}
    \item[({\tt ind})]~$f(X)=0$,
    \item[({\tt int})]~$f(X)=\int_0^1 X(t)dt$,
    \item[({\tt sqnorm})]~$f(X)=\int_0^1 X(t)^2dt$,
    \item[({\tt weight})] ~$f(X)=\int_0^1 t^2X(t)dt$,
    \item[({\tt sin})]~$f(X)=\sin\left(2\pi\int_0^1 X(t)dt\right)$,
    \item[({\tt max})]~$f(X)=\max_{t\in\lbrack 0,1\rbrack} X(t)$,
    \item[({\tt range})~]$f(X)=\max_{t\in\lbrack 0,1\rbrack} X(t)-\min_{t\in\lbrack 0,1\rbrack} X(t)$, and
    \item[({\tt eval})]~$f(X)=X(0.5)$.
\end{itemize}
 Since we do not have a closed-form expression for the scaling coefficients~$\sigma$, we simply estimate them from the~$n\times B$ signals~$f(X_i^b)$ ($i$-th observation in the~$b$-th sample). Table~\ref{tab:coeff-comparision-100-exp_decay} shows the means and standard deviations of the calculated values of~$\widehat T_n$ and~$\widehat R_n$ for sample size~${n=100}$. In Figure~\ref{fig:visu} the functional relationships between~$Y$ and~$X$ under Setups~({\tt sqnorm}) and~({\tt sin}) are visualized.
 
  \begin{figure}[t!]
\captionsetup[subfigure]{labelformat=empty}
\begin{subfigure}[b]{0.3\textwidth}
 \includegraphics[width=\textwidth]{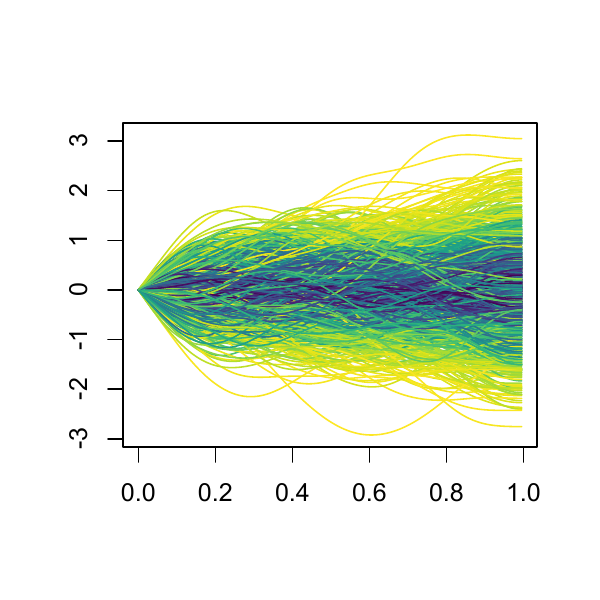}
 \includegraphics[width=\textwidth]{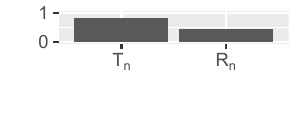}
 \caption{(a)~$r^2=100\%$}
\end{subfigure}
\begin{subfigure}[b]{0.3\textwidth}
 \includegraphics[width=\textwidth]{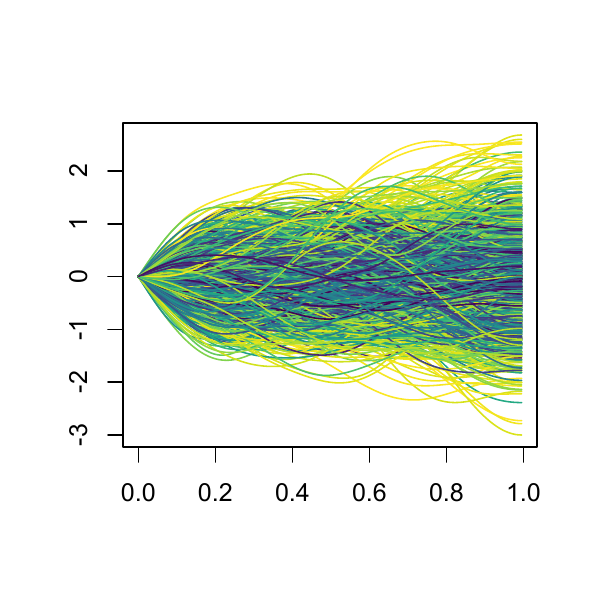}
 \includegraphics[width=\textwidth]{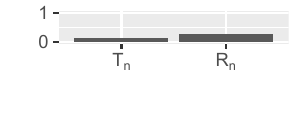}
 \caption{(b)~$r^2=50\%$}
\end{subfigure}
 \begin{subfigure}[b]{0.3\textwidth}
 \includegraphics[width=\textwidth]{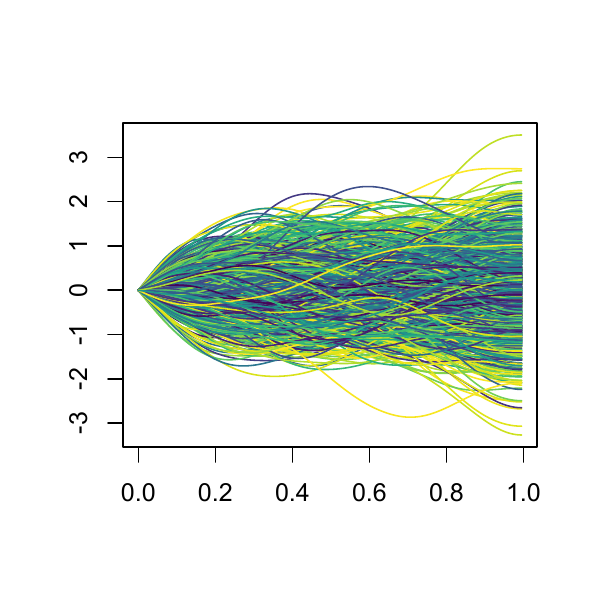}
 \includegraphics[width=\textwidth]{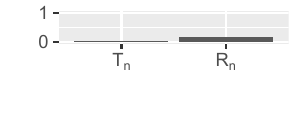}
 \caption{(c)~$r^2=10\%$}
\end{subfigure}
\\

\begin{subfigure}[b]{0.3\textwidth}
 \includegraphics[width=\textwidth]{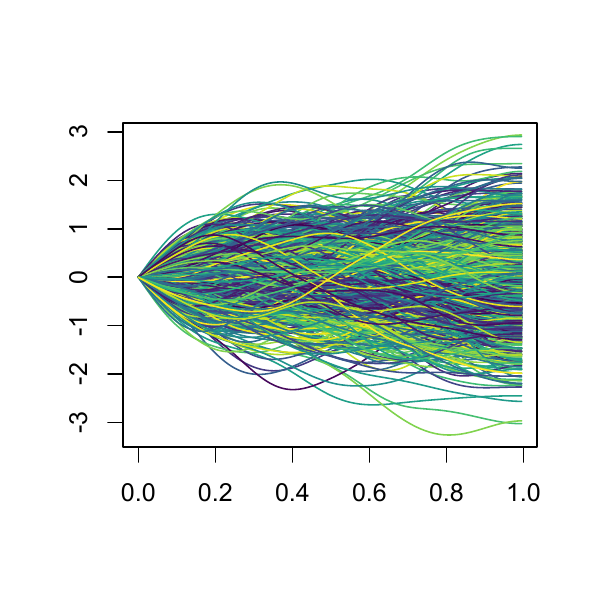}
 \includegraphics[width=\textwidth]{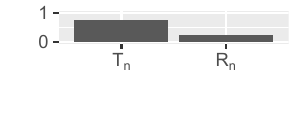}
 \caption{(a)~$r^2=100\%$}
\end{subfigure}
\begin{subfigure}[b]{0.3\textwidth}
 \includegraphics[width=\textwidth]{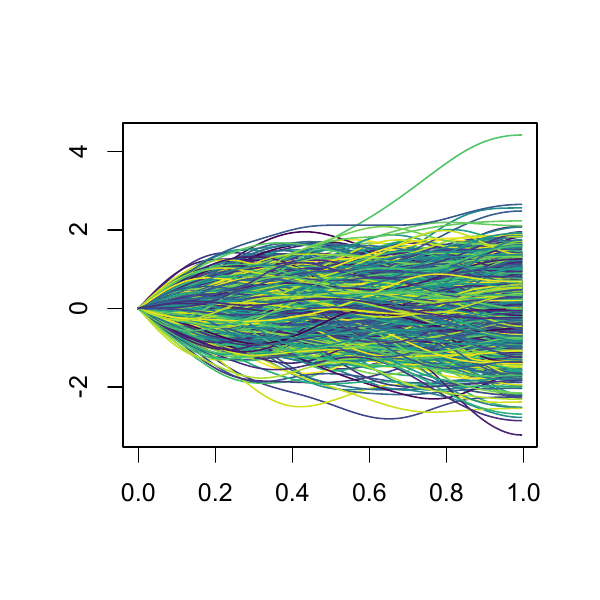}
 \includegraphics[width=\textwidth]{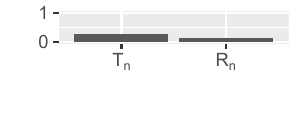}
 \caption{(b)~$r^2=50\%$}
\end{subfigure}
 \begin{subfigure}[b]{0.3\textwidth}
 \includegraphics[width=\textwidth]{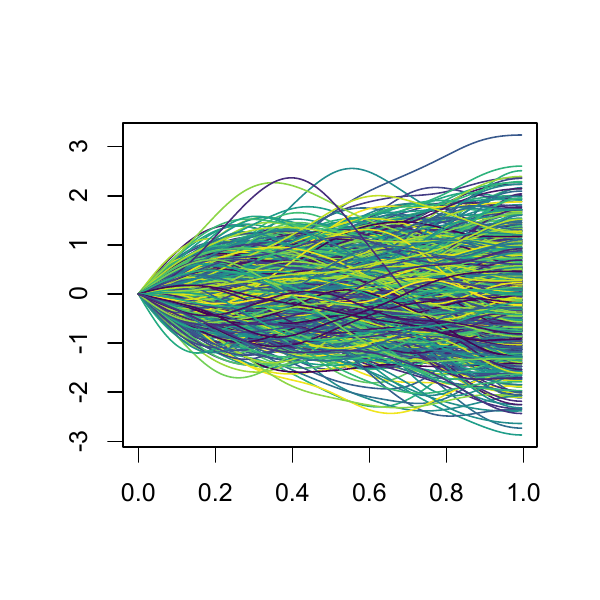}
 \includegraphics[width=\textwidth]{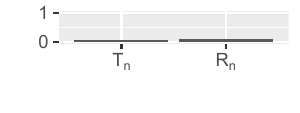}
 \caption{(c)~$r^2=10\%$}
\end{subfigure}
\\
\caption{Visualization of the \texttt{sqnorm} (top) and \texttt{sin} (bottom) relationships. The~$n=1000$ curves~$X$ are colored according to~$Y$. The bar plots compare the values of~$\widehat T_n$ and~$\widehat R_n$.}\label{fig:eval-visualisation}
\label{fig:visu}
\end{figure}

\begin{table}
\begin{center}
 \begin{tabular}{l|cccc|cccc}
 \hline
 & \multicolumn{4}{c}{$\widehat T_n$} & \multicolumn{4}{c}{$\widehat R_n$}\\
 \hline
 {\tt ind} & \multicolumn{4}{c}{0} & \multicolumn{4}{c}{0.19}\\
 & \multicolumn{4}{c}{(0.12)} & \multicolumn{4}{c}{(0.02)}\\
 \hline
~$r^2$ & 1 & 0.9 & 0.6 & 0.1 & 1 & 0.9 & 0.5 & 0.1\\
 \hline
 \input{means-sds-100-exp_decay.csv}
\end{tabular}
\end{center}
\caption{Comparison of~$\widehat T_n$ and~$\widehat R_n$ with a sample size of~$n=100$. Standard deviations are given in brackets.}\label{tab:coeff-comparision-100-exp_decay}
\end{table}

We can observe that~$\widehat R_n$ takes larger values than~$\widehat T_n$ in most cases, especially for high levels of noise. This indicates a higher ability to detect noisy relationships. A notable exception is the sine of integral relationship {\tt sin}, for which~$\widehat T_n$ takes higher values than~$\widehat R_n$ at all levels of noise. This is in line with the observation in \cite{chatterjee_new_2020}, stating that the approach discussed here is powerful in situations where the relation between~$X$ and~$Y$ is `oscillatory' in nature. Note that for small and medium sample sizes, the large values of~$\widehat R_n$ seem to be in part due to a strong bias---observe the high value of~$\widehat R_n=0.194$ for independent data. This bias has mostly vanished at the large sample size~$n=1000$.

\subsection{Finite sample distribution and running times}\label{ss:finsamp}
Next, we investigate the distribution of the test statistic~$\mathcal{I}_n$ under independence (setting {\tt ind}). In Figure~\ref{fig:dist-independence}, we created 5000 samples of~$X$ -- under Setup~(b) -- and~$Y$ -- uniformly distributed on~$\lbrack0,1\rbrack$ and independent of~$X$ -- and plot the histogram of the~$\mathcal{I}_n$'s along with the density of a standard normal distribution. 
\begin{figure}[t!]
\centering
  \begin{subfigure}[b]{0.45\textwidth}
   \resizebox{\textwidth}{!}{\input{density-20.tex}}
   \caption{$n=20$}
   \end{subfigure}
  \begin{subfigure}[b]{0.45\textwidth}
   \resizebox{\textwidth}{!}{\input{density-100.tex}}
   \caption{$n=100$}
   \end{subfigure}
    \caption{Histogram of~$\mathcal I_n$ under independence and comparison to the standard normal density.}
    \label{fig:dist-independence}
\end{figure}
Even for sample size~$n=20$, the quality of the normal approximation is quite reasonable. The~$p$-values of the Shapiro-Wilk test for~$n=20,100,1000$ are 0.00001, 0.34 and 0.89, respectively, confirming the good approximation by the limiting law.

One of the advantages of~$\mathcal I_n$ compared to the competing methods is that it is an asymptotic test and thus comes with significantly shorter running time. While it is known that~$\mathcal{I}^\mathrm{DC}_n$ and $\mathcal{I}^\mathrm{BC}_n$ asymptotically follow infinite mixtures of~$\chi^2$-distributions, the mixing weights depend on the marginal distributions of~$X$ and~$Y$ and can therefore not be used for an independence test. With regard to the sample size, the number of operations for calculation of~$\mathcal{I}^\mathrm{DC}_n$ and~$\mathcal{I}^\mathrm{DC}_n$ is of order~$n^2$ and~$n^3$, respectively, while it is only of order~$n\log n$ for computing~$\mathcal{I}_n$. For the simulation, we chose~$200$ resampling steps for the permutation and bootstrap tests. This is usually not considered enough (\cite{garcia-portugues_goodness--fit_2014} recommend at least 5000 steps for their test). This small number was chosen to keep the computation times reasonable for the simulation experiment. Table~\ref{tab:runtimes} shows the results. As expected, the computation time of~$\mathcal{I}_n$ increases considerably slower than the time required to perform the other tests. For~$n=10000$, the~$\mathcal{I}^\mathrm{CvM}_n$-based test was aborted after~90 minutes.
\begin{table}
 \centering
 \begin{tabular}{c|cccc}
 \hline
~$n$ &~$\mathcal{I}_n$ &~$\mathcal{I}^\mathrm{DC}_n$ &~$\mathcal{I}^\mathrm{BC}_n$ &~$\mathcal{I}^\mathrm{CvM}_n$\\
 \hline
 100 & <0.01& <0.01 & 0.11 & 0.98\\
 500 & 0.01& 0.07 & 3.00 & 4.15\\
 1000& 0.06& 0.38& 12.94 & 15.59\\
 2000 & 0.25& 1.75& 54.89 & 100.03\\
 10000 & 5.52& 52.84& 1591.2 & >90 mins
 \end{tabular}
\caption{Running times (in seconds) of the three independence tests for different sample sizes.}\label{tab:runtimes}
\end{table}

\subsection{Power study}\label{s:independence-test}\label{ss:power}
In this section, we compare the independence tests~$\mathcal{I}_n$,~$\mathcal{I}^\mathrm{DC}_n$,~$\mathcal{I}^\mathrm{BC}_n$ and~$\mathcal{I}^\mathrm{CvM}_n$ in terms of power. We estimated the powers of the respective tests for the models {\tt int}, {\tt sqnorm}, {\tt weight}, {\tt sin}, {\tt max}, {\tt range} and {\tt eval}. We create~$B= 500$ samples of medium size~$n=100$ for each type of relationship and~${r_i^2=i/10}$,~${0\leq~i\leq~10}$. The results can be seen in Figure~\ref{fig:powers-exp_decay}.
\begin{figure}[h!]
\resizebox{0.45\textwidth}{!}{
 \input{powers-int-exp_decay.tex}
}
\resizebox{0.45\textwidth}{!}{
 \input{powers-sqnorm-exp_decay.tex}
}\\

\resizebox{0.45\textwidth}{!}{
 \input{powers-weight-exp_decay.tex}
}
\resizebox{0.45\textwidth}{!}{
 \input{powers-sin-exp_decay.tex}
}\\

\resizebox{0.45\textwidth}{!}{
 \input{powers-max-exp_decay.tex}
}
\resizebox{0.45\textwidth}{!}{
 \input{powers-range-exp_decay.tex}
}\\

\centering
\resizebox{0.45\textwidth}{!}{
 \input{powers-eval-exp_decay.tex}
}

\caption{Estimated powers of the three tests for independence at different levels of noise.}\label{fig:powers-exp_decay}
\end{figure}
We can observe that in most settings~$\mathcal{I}^\mathrm{DC}_n$ has the highest power. In absence of noise~$\mathcal{I}_n$,~$\mathcal{I}^\mathrm{DC}_n$ and~$\mathcal{I}^\mathrm{BC}_n$ have power close to one in most considered settings. It seems that~$\mathcal{I}^\mathrm{DC}_n$ and~$\mathcal{I}^\mathrm{BC}_n$ are less sensitive to noise, in the sense that the power of the test decreases slower with increasing noise level. The power of~$\mathcal{I}^\mathrm{CvM}_n$ can be rather low if the relationship between~$X$ and~$Y$ is non-linear, even in the absence of noise. For {\tt sin} and {\tt range} the power is close to the nominal level of the test in the absence of noise. For {\tt sin}, the tests based on~$\mathcal{I}_n$ and~$\mathcal{I}^\mathrm{BC}_n$ performs remarkably well, even in comparison to the one based on~$\mathcal{I}^\mathrm{DC}_n$. Moreover, these two are the only tests that achieve (near) perfect power in the absence of noise in all considered settings. The only drawback of~$\mathcal{I}^\mathrm{BC}_n$ seems to be the lack of an asymptotic theory and the resulting high computational cost of the resampling based test. This can be a significant disadvantage when dealing with very large amounts of data.

%% file: means-sds-100-exp_decay.csv
\texttt{int}&0.86&0.66&0.29&0.05&0.94&0.88&0.65&0.34\\
&(0.02)&(0.05)&(0.1)&(0.12)&(0.01)&(0.02)&(0.05)&(0.07)\\\hline
\texttt{sqnorm}&0.64&0.4&0.14&0.01&0.5&0.47&0.36&0.25\\
&(0.05)&(0.08)&(0.11)&(0.11)&(0.03)&(0.04)&(0.04)&(0.03)\\\hline
\texttt{weight}&0.86&0.67&0.29&0.05&0.94&0.88&0.65&0.34\\
&(0.02)&(0.05)&(0.1)&(0.11)&(0.01)&(0.02)&(0.05)&(0.07)\\\hline
\texttt{sin}&0.57&0.49&0.27&0.04&0.27&0.27&0.25&0.22\\
&(0.06)&(0.07)&(0.1)&(0.11)&(0.04)&(0.03)&(0.03)&(0.03)\\\hline
\texttt{max}&0.69&0.52&0.22&0.03&0.84&0.79&0.59&0.32\\
&(0.05)&(0.07)&(0.1)&(0.12)&(0.03)&(0.03)&(0.06)&(0.06)\\\hline
\texttt{range}&0.44&0.34&0.12&-0.01&0.41&0.39&0.32&0.24\\
&(0.07)&(0.09)&(0.11)&(0.11)&(0.03)&(0.03)&(0.04)&(0.03)\\\hline
\texttt{eval}&0.78&0.63&0.28&0.05&0.83&0.79&0.58&0.32\\
&(0.03)&(0.03)&(0.03)&(0.03)&(0.03)&(0.03)&(0.03)&(0.03)

%% file: density-20.tex
\begin{tikzpicture}[x=1pt,y=1pt]
\definecolor{fillColor}{RGB}{255,255,255}
\path[use as bounding box,fill=fillColor,fill opacity=0.00] (0,0) rectangle (325.21,325.21);
\begin{scope}
\path[clip] (  0.00,  0.00) rectangle (325.21,325.21);
\definecolor{drawColor}{RGB}{0,0,0}

\node[text=drawColor,anchor=base,inner sep=0pt, outer sep=0pt, scale=  1.00] at (174.61, 15.60) {$\mathcal I_n$};

\node[text=drawColor,rotate= 90.00,anchor=base,inner sep=0pt, outer sep=0pt, scale=  1.00] at ( 10.80,168.61) {Density};
\end{scope}
\begin{scope}
\path[clip] (  0.00,  0.00) rectangle (325.21,325.21);
\definecolor{drawColor}{RGB}{0,0,0}

\path[draw=drawColor,line width= 0.4pt,line join=round,line cap=round] ( 58.49, 61.20) -- (290.73, 61.20);

\path[draw=drawColor,line width= 0.4pt,line join=round,line cap=round] ( 58.49, 61.20) -- ( 58.49, 55.20);

\path[draw=drawColor,line width= 0.4pt,line join=round,line cap=round] ( 97.20, 61.20) -- ( 97.20, 55.20);

\path[draw=drawColor,line width= 0.4pt,line join=round,line cap=round] (135.90, 61.20) -- (135.90, 55.20);

\path[draw=drawColor,line width= 0.4pt,line join=round,line cap=round] (174.61, 61.20) -- (174.61, 55.20);

\path[draw=drawColor,line width= 0.4pt,line join=round,line cap=round] (213.31, 61.20) -- (213.31, 55.20);

\path[draw=drawColor,line width= 0.4pt,line join=round,line cap=round] (252.02, 61.20) -- (252.02, 55.20);

\path[draw=drawColor,line width= 0.4pt,line join=round,line cap=round] (290.73, 61.20) -- (290.73, 55.20);

\node[text=drawColor,anchor=base,inner sep=0pt, outer sep=0pt, scale=  1.00] at ( 58.49, 39.60) {-3};

\node[text=drawColor,anchor=base,inner sep=0pt, outer sep=0pt, scale=  1.00] at ( 97.20, 39.60) {-2};

\node[text=drawColor,anchor=base,inner sep=0pt, outer sep=0pt, scale=  1.00] at (135.90, 39.60) {-1};

\node[text=drawColor,anchor=base,inner sep=0pt, outer sep=0pt, scale=  1.00] at (174.61, 39.60) {0};

\node[text=drawColor,anchor=base,inner sep=0pt, outer sep=0pt, scale=  1.00] at (213.31, 39.60) {1};

\node[text=drawColor,anchor=base,inner sep=0pt, outer sep=0pt, scale=  1.00] at (252.02, 39.60) {2};

\node[text=drawColor,anchor=base,inner sep=0pt, outer sep=0pt, scale=  1.00] at (290.73, 39.60) {3};

\path[draw=drawColor,line width= 0.4pt,line join=round,line cap=round] ( 49.20, 69.16) -- ( 49.20,245.96);

\path[draw=drawColor,line width= 0.4pt,line join=round,line cap=round] ( 49.20, 69.16) -- ( 43.20, 69.16);

\path[draw=drawColor,line width= 0.4pt,line join=round,line cap=round] ( 49.20,113.36) -- ( 43.20,113.36);

\path[draw=drawColor,line width= 0.4pt,line join=round,line cap=round] ( 49.20,157.56) -- ( 43.20,157.56);

\path[draw=drawColor,line width= 0.4pt,line join=round,line cap=round] ( 49.20,201.76) -- ( 43.20,201.76);

\path[draw=drawColor,line width= 0.4pt,line join=round,line cap=round] ( 49.20,245.96) -- ( 43.20,245.96);

\node[text=drawColor,rotate= 90.00,anchor=base,inner sep=0pt, outer sep=0pt, scale=  1.00] at ( 34.80, 69.16) {0.0};

\node[text=drawColor,rotate= 90.00,anchor=base,inner sep=0pt, outer sep=0pt, scale=  1.00] at ( 34.80,113.36) {0.1};

\node[text=drawColor,rotate= 90.00,anchor=base,inner sep=0pt, outer sep=0pt, scale=  1.00] at ( 34.80,157.56) {0.2};

\node[text=drawColor,rotate= 90.00,anchor=base,inner sep=0pt, outer sep=0pt, scale=  1.00] at ( 34.80,201.76) {0.3};

\node[text=drawColor,rotate= 90.00,anchor=base,inner sep=0pt, outer sep=0pt, scale=  1.00] at ( 34.80,245.96) {0.4};
\end{scope}
\begin{scope}
\path[clip] ( 49.20, 61.20) rectangle (300.01,276.01);
\definecolor{drawColor}{RGB}{0,0,0}
\definecolor{fillColor}{RGB}{211,211,211}

\path[draw=drawColor,line width= 0.4pt,line join=round,line cap=round,fill=fillColor] ( 50.75, 69.16) rectangle ( 58.49, 71.37);

\path[draw=drawColor,line width= 0.4pt,line join=round,line cap=round,fill=fillColor] ( 58.49, 69.16) rectangle ( 66.23, 70.92);

\path[draw=drawColor,line width= 0.4pt,line join=round,line cap=round,fill=fillColor] ( 66.23, 69.16) rectangle ( 73.97, 75.79);

\path[draw=drawColor,line width= 0.4pt,line join=round,line cap=round,fill=fillColor] ( 73.97, 69.16) rectangle ( 81.71, 78.44);

\path[draw=drawColor,line width= 0.4pt,line join=round,line cap=round,fill=fillColor] ( 81.71, 69.16) rectangle ( 89.45, 81.53);

\path[draw=drawColor,line width= 0.4pt,line join=round,line cap=round,fill=fillColor] ( 89.45, 69.16) rectangle ( 97.20, 89.93);

\path[draw=drawColor,line width= 0.4pt,line join=round,line cap=round,fill=fillColor] ( 97.20, 69.16) rectangle (104.94, 95.68);

\path[draw=drawColor,line width= 0.4pt,line join=round,line cap=round,fill=fillColor] (104.94, 69.16) rectangle (112.68,123.08);

\path[draw=drawColor,line width= 0.4pt,line join=round,line cap=round,fill=fillColor] (112.68, 69.16) rectangle (120.42,129.71);

\path[draw=drawColor,line width= 0.4pt,line join=round,line cap=round,fill=fillColor] (120.42, 69.16) rectangle (128.16,147.39);

\path[draw=drawColor,line width= 0.4pt,line join=round,line cap=round,fill=fillColor] (128.16, 69.16) rectangle (135.90,165.07);

\path[draw=drawColor,line width= 0.4pt,line join=round,line cap=round,fill=fillColor] (135.90, 69.16) rectangle (143.64,182.75);

\path[draw=drawColor,line width= 0.4pt,line join=round,line cap=round,fill=fillColor] (143.64, 69.16) rectangle (151.38,209.27);

\path[draw=drawColor,line width= 0.4pt,line join=round,line cap=round,fill=fillColor] (151.38, 69.16) rectangle (159.13,198.22);

\path[draw=drawColor,line width= 0.4pt,line join=round,line cap=round,fill=fillColor] (159.13, 69.16) rectangle (166.87,222.09);

\path[draw=drawColor,line width= 0.4pt,line join=round,line cap=round,fill=fillColor] (166.87, 69.16) rectangle (174.61,265.41);

\path[draw=drawColor,line width= 0.4pt,line join=round,line cap=round,fill=fillColor] (174.61, 69.16) rectangle (182.35,225.18);

\path[draw=drawColor,line width= 0.4pt,line join=round,line cap=round,fill=fillColor] (182.35, 69.16) rectangle (190.09,229.16);

\path[draw=drawColor,line width= 0.4pt,line join=round,line cap=round,fill=fillColor] (190.09, 69.16) rectangle (197.83,226.07);

\path[draw=drawColor,line width= 0.4pt,line join=round,line cap=round,fill=fillColor] (197.83, 69.16) rectangle (205.57,215.90);

\path[draw=drawColor,line width= 0.4pt,line join=round,line cap=round,fill=fillColor] (205.57, 69.16) rectangle (213.31,184.08);

\path[draw=drawColor,line width= 0.4pt,line join=round,line cap=round,fill=fillColor] (213.31, 69.16) rectangle (221.05,162.86);

\path[draw=drawColor,line width= 0.4pt,line join=round,line cap=round,fill=fillColor] (221.05, 69.16) rectangle (228.80,152.70);

\path[draw=drawColor,line width= 0.4pt,line join=round,line cap=round,fill=fillColor] (228.80, 69.16) rectangle (236.54,142.09);

\path[draw=drawColor,line width= 0.4pt,line join=round,line cap=round,fill=fillColor] (236.54, 69.16) rectangle (244.28,116.45);

\path[draw=drawColor,line width= 0.4pt,line join=round,line cap=round,fill=fillColor] (244.28, 69.16) rectangle (252.02,101.42);

\path[draw=drawColor,line width= 0.4pt,line join=round,line cap=round,fill=fillColor] (252.02, 69.16) rectangle (259.76, 89.93);

\path[draw=drawColor,line width= 0.4pt,line join=round,line cap=round,fill=fillColor] (259.76, 69.16) rectangle (267.50, 82.86);

\path[draw=drawColor,line width= 0.4pt,line join=round,line cap=round,fill=fillColor] (267.50, 69.16) rectangle (275.24, 74.90);

\path[draw=drawColor,line width= 0.4pt,line join=round,line cap=round,fill=fillColor] (275.24, 69.16) rectangle (282.98, 73.13);

\path[draw=drawColor,line width= 0.4pt,line join=round,line cap=round,fill=fillColor] (282.98, 69.16) rectangle (290.73, 70.04);

\path[draw=drawColor,line width= 0.4pt,line join=round,line cap=round,fill=fillColor] (290.73, 69.16) rectangle (298.47, 69.60);
\definecolor{drawColor}{RGB}{255,0,0}

\path[draw=drawColor,line width= 0.4pt,line join=round,line cap=round] ( 19.78, 69.22) --
	( 20.17, 69.22) --
	( 20.56, 69.22) --
	( 20.94, 69.22) --
	( 21.33, 69.23) --
	( 21.72, 69.23) --
	( 22.11, 69.23) --
	( 22.49, 69.23) --
	( 22.88, 69.24) --
	( 23.27, 69.24) --
	( 23.65, 69.24) --
	( 24.04, 69.25) --
	( 24.43, 69.25) --
	( 24.82, 69.25) --
	( 25.20, 69.26) --
	( 25.59, 69.26) --
	( 25.98, 69.27) --
	( 26.36, 69.27) --
	( 26.75, 69.28) --
	( 27.14, 69.28) --
	( 27.52, 69.29) --
	( 27.91, 69.29) --
	( 28.30, 69.30) --
	( 28.69, 69.30) --
	( 29.07, 69.31) --
	( 29.46, 69.31) --
	( 29.85, 69.32) --
	( 30.23, 69.32) --
	( 30.62, 69.33) --
	( 31.01, 69.34) --
	( 31.40, 69.34) --
	( 31.78, 69.35) --
	( 32.17, 69.36) --
	( 32.56, 69.37) --
	( 32.94, 69.37) --
	( 33.33, 69.38) --
	( 33.72, 69.39) --
	( 34.10, 69.40) --
	( 34.49, 69.41) --
	( 34.88, 69.42) --
	( 35.27, 69.43) --
	( 35.65, 69.44) --
	( 36.04, 69.45) --
	( 36.43, 69.46) --
	( 36.81, 69.47) --
	( 37.20, 69.48) --
	( 37.59, 69.49) --
	( 37.98, 69.50) --
	( 38.36, 69.52) --
	( 38.75, 69.53) --
	( 39.14, 69.54) --
	( 39.52, 69.56) --
	( 39.91, 69.57) --
	( 40.30, 69.58) --
	( 40.68, 69.60) --
	( 41.07, 69.62) --
	( 41.46, 69.63) --
	( 41.85, 69.65) --
	( 42.23, 69.66) --
	( 42.62, 69.68) --
	( 43.01, 69.70) --
	( 43.39, 69.72) --
	( 43.78, 69.74) --
	( 44.17, 69.76) --
	( 44.56, 69.78) --
	( 44.94, 69.80) --
	( 45.33, 69.82) --
	( 45.72, 69.85) --
	( 46.10, 69.87) --
	( 46.49, 69.89) --
	( 46.88, 69.92) --
	( 47.26, 69.94) --
	( 47.65, 69.97) --
	( 48.04, 70.00) --
	( 48.43, 70.02) --
	( 48.81, 70.05) --
	( 49.20, 70.08) --
	( 49.59, 70.11) --
	( 49.97, 70.14) --
	( 50.36, 70.18) --
	( 50.75, 70.21) --
	( 51.14, 70.24) --
	( 51.52, 70.28) --
	( 51.91, 70.32) --
	( 52.30, 70.35) --
	( 52.68, 70.39) --
	( 53.07, 70.43) --
	( 53.46, 70.47) --
	( 53.84, 70.51) --
	( 54.23, 70.56) --
	( 54.62, 70.60) --
	( 55.01, 70.65) --
	( 55.39, 70.69) --
	( 55.78, 70.74) --
	( 56.17, 70.79) --
	( 56.55, 70.84) --
	( 56.94, 70.89) --
	( 57.33, 70.95) --
	( 57.72, 71.00) --
	( 58.10, 71.06) --
	( 58.49, 71.12) --
	( 58.88, 71.17) --
	( 59.26, 71.24) --
	( 59.65, 71.30) --
	( 60.04, 71.36) --
	( 60.42, 71.43) --
	( 60.81, 71.50) --
	( 61.20, 71.57) --
	( 61.59, 71.64) --
	( 61.97, 71.71) --
	( 62.36, 71.79) --
	( 62.75, 71.86) --
	( 63.13, 71.94) --
	( 63.52, 72.03) --
	( 63.91, 72.11) --
	( 64.30, 72.19) --
	( 64.68, 72.28) --
	( 65.07, 72.37) --
	( 65.46, 72.46) --
	( 65.84, 72.56) --
	( 66.23, 72.65) --
	( 66.62, 72.75) --
	( 67.00, 72.86) --
	( 67.39, 72.96) --
	( 67.78, 73.07) --
	( 68.17, 73.18) --
	( 68.55, 73.29) --
	( 68.94, 73.40) --
	( 69.33, 73.52) --
	( 69.71, 73.64) --
	( 70.10, 73.76) --
	( 70.49, 73.89) --
	( 70.88, 74.02) --
	( 71.26, 74.15) --
	( 71.65, 74.28) --
	( 72.04, 74.42) --
	( 72.42, 74.56) --
	( 72.81, 74.71) --
	( 73.20, 74.85) --
	( 73.58, 75.01) --
	( 73.97, 75.16) --
	( 74.36, 75.32) --
	( 74.75, 75.48) --
	( 75.13, 75.64) --
	( 75.52, 75.81) --
	( 75.91, 75.98) --
	( 76.29, 76.16) --
	( 76.68, 76.34) --
	( 77.07, 76.52) --
	( 77.46, 76.71) --
	( 77.84, 76.90) --
	( 78.23, 77.10) --
	( 78.62, 77.30) --
	( 79.00, 77.50) --
	( 79.39, 77.71) --
	( 79.78, 77.92) --
	( 80.16, 78.14) --
	( 80.55, 78.36) --
	( 80.94, 78.59) --
	( 81.33, 78.82) --
	( 81.71, 79.05) --
	( 82.10, 79.29) --
	( 82.49, 79.54) --
	( 82.87, 79.79) --
	( 83.26, 80.04) --
	( 83.65, 80.30) --
	( 84.04, 80.57) --
	( 84.42, 80.84) --
	( 84.81, 81.11) --
	( 85.20, 81.39) --
	( 85.58, 81.68) --
	( 85.97, 81.97) --
	( 86.36, 82.26) --
	( 86.74, 82.57) --
	( 87.13, 82.87) --
	( 87.52, 83.19) --
	( 87.91, 83.50) --
	( 88.29, 83.83) --
	( 88.68, 84.16) --
	( 89.07, 84.49) --
	( 89.45, 84.84) --
	( 89.84, 85.18) --
	( 90.23, 85.54) --
	( 90.62, 85.90) --
	( 91.00, 86.26) --
	( 91.39, 86.64) --
	( 91.78, 87.02) --
	( 92.16, 87.40) --
	( 92.55, 87.79) --
	( 92.94, 88.19) --
	( 93.32, 88.60) --
	( 93.71, 89.01) --
	( 94.10, 89.43) --
	( 94.49, 89.85) --
	( 94.87, 90.28) --
	( 95.26, 90.72) --
	( 95.65, 91.17) --
	( 96.03, 91.62) --
	( 96.42, 92.08) --
	( 96.81, 92.55) --
	( 97.20, 93.02) --
	( 97.58, 93.50) --
	( 97.97, 93.99) --
	( 98.36, 94.48) --
	( 98.74, 94.99) --
	( 99.13, 95.50) --
	( 99.52, 96.01) --
	( 99.90, 96.54) --
	(100.29, 97.07) --
	(100.68, 97.61) --
	(101.07, 98.16) --
	(101.45, 98.71) --
	(101.84, 99.28) --
	(102.23, 99.85) --
	(102.61,100.42) --
	(103.00,101.01) --
	(103.39,101.60) --
	(103.78,102.20) --
	(104.16,102.81) --
	(104.55,103.43) --
	(104.94,104.05) --
	(105.32,104.68) --
	(105.71,105.32) --
	(106.10,105.97) --
	(106.48,106.63) --
	(106.87,107.29) --
	(107.26,107.96) --
	(107.65,108.64) --
	(108.03,109.33) --
	(108.42,110.02) --
	(108.81,110.73) --
	(109.19,111.44) --
	(109.58,112.16) --
	(109.97,112.88) --
	(110.36,113.62) --
	(110.74,114.36) --
	(111.13,115.11) --
	(111.52,115.87) --
	(111.90,116.63) --
	(112.29,117.40) --
	(112.68,118.18) --
	(113.06,118.97) --
	(113.45,119.77) --
	(113.84,120.57) --
	(114.23,121.38) --
	(114.61,122.20) --
	(115.00,123.03) --
	(115.39,123.86) --
	(115.77,124.70) --
	(116.16,125.55) --
	(116.55,126.40) --
	(116.94,127.27) --
	(117.32,128.14) --
	(117.71,129.01) --
	(118.10,129.90) --
	(118.48,130.79) --
	(118.87,131.68) --
	(119.26,132.59) --
	(119.64,133.50) --
	(120.03,134.41) --
	(120.42,135.34) --
	(120.81,136.27) --
	(121.19,137.20) --
	(121.58,138.14) --
	(121.97,139.09) --
	(122.35,140.05) --
	(122.74,141.01) --
	(123.13,141.97) --
	(123.52,142.94) --
	(123.90,143.92) --
	(124.29,144.90) --
	(124.68,145.89) --
	(125.06,146.88) --
	(125.45,147.88) --
	(125.84,148.88) --
	(126.22,149.89) --
	(126.61,150.90) --
	(127.00,151.92) --
	(127.39,152.94) --
	(127.77,153.96) --
	(128.16,154.99) --
	(128.55,156.02) --
	(128.93,157.05) --
	(129.32,158.09) --
	(129.71,159.14) --
	(130.10,160.18) --
	(130.48,161.23) --
	(130.87,162.28) --
	(131.26,163.33) --
	(131.64,164.39) --
	(132.03,165.45) --
	(132.42,166.51) --
	(132.81,167.57) --
	(133.19,168.63) --
	(133.58,169.70) --
	(133.97,170.77) --
	(134.35,171.83) --
	(134.74,172.90) --
	(135.13,173.97) --
	(135.51,175.04) --
	(135.90,176.11) --
	(136.29,177.18) --
	(136.68,178.25) --
	(137.06,179.32) --
	(137.45,180.38) --
	(137.84,181.45) --
	(138.22,182.52) --
	(138.61,183.58) --
	(139.00,184.65) --
	(139.39,185.71) --
	(139.77,186.77) --
	(140.16,187.82) --
	(140.55,188.88) --
	(140.93,189.93) --
	(141.32,190.98) --
	(141.71,192.03) --
	(142.09,193.07) --
	(142.48,194.11) --
	(142.87,195.14) --
	(143.26,196.17) --
	(143.64,197.20) --
	(144.03,198.22) --
	(144.42,199.24) --
	(144.80,200.25) --
	(145.19,201.26) --
	(145.58,202.26) --
	(145.97,203.26) --
	(146.35,204.25) --
	(146.74,205.23) --
	(147.13,206.20) --
	(147.51,207.17) --
	(147.90,208.14) --
	(148.29,209.09) --
	(148.67,210.04) --
	(149.06,210.98) --
	(149.45,211.91) --
	(149.84,212.84) --
	(150.22,213.75) --
	(150.61,214.66) --
	(151.00,215.56) --
	(151.38,216.44) --
	(151.77,217.32) --
	(152.16,218.19) --
	(152.55,219.05) --
	(152.93,219.90) --
	(153.32,220.74) --
	(153.71,221.57) --
	(154.09,222.39) --
	(154.48,223.19) --
	(154.87,223.99) --
	(155.25,224.77) --
	(155.64,225.54) --
	(156.03,226.30) --
	(156.42,227.05) --
	(156.80,227.79) --
	(157.19,228.51) --
	(157.58,229.22) --
	(157.96,229.92) --
	(158.35,230.60) --
	(158.74,231.28) --
	(159.13,231.93) --
	(159.51,232.58) --
	(159.90,233.21) --
	(160.29,233.82) --
	(160.67,234.43) --
	(161.06,235.01) --
	(161.45,235.59) --
	(161.83,236.15) --
	(162.22,236.69) --
	(162.61,237.22) --
	(163.00,237.73) --
	(163.38,238.23) --
	(163.77,238.71) --
	(164.16,239.18) --
	(164.54,239.63) --
	(164.93,240.07) --
	(165.32,240.49) --
	(165.71,240.89) --
	(166.09,241.27) --
	(166.48,241.65) --
	(166.87,242.00) --
	(167.25,242.34) --
	(167.64,242.66) --
	(168.03,242.96) --
	(168.41,243.25) --
	(168.80,243.52) --
	(169.19,243.77) --
	(169.58,244.01) --
	(169.96,244.23) --
	(170.35,244.43) --
	(170.74,244.61) --
	(171.12,244.78) --
	(171.51,244.93) --
	(171.90,245.06) --
	(172.29,245.17) --
	(172.67,245.27) --
	(173.06,245.35) --
	(173.45,245.41) --
	(173.83,245.46) --
	(174.22,245.48) --
	(174.61,245.49) --
	(174.99,245.48) --
	(175.38,245.46) --
	(175.77,245.41) --
	(176.16,245.35) --
	(176.54,245.27) --
	(176.93,245.17) --
	(177.32,245.06) --
	(177.70,244.93) --
	(178.09,244.78) --
	(178.48,244.61) --
	(178.87,244.43) --
	(179.25,244.23) --
	(179.64,244.01) --
	(180.03,243.77) --
	(180.41,243.52) --
	(180.80,243.25) --
	(181.19,242.96) --
	(181.57,242.66) --
	(181.96,242.34) --
	(182.35,242.00) --
	(182.74,241.65) --
	(183.12,241.27) --
	(183.51,240.89) --
	(183.90,240.49) --
	(184.28,240.07) --
	(184.67,239.63) --
	(185.06,239.18) --
	(185.45,238.71) --
	(185.83,238.23) --
	(186.22,237.73) --
	(186.61,237.22) --
	(186.99,236.69) --
	(187.38,236.15) --
	(187.77,235.59) --
	(188.15,235.01) --
	(188.54,234.43) --
	(188.93,233.82) --
	(189.32,233.21) --
	(189.70,232.58) --
	(190.09,231.93) --
	(190.48,231.28) --
	(190.86,230.60) --
	(191.25,229.92) --
	(191.64,229.22) --
	(192.03,228.51) --
	(192.41,227.79) --
	(192.80,227.05) --
	(193.19,226.30) --
	(193.57,225.54) --
	(193.96,224.77) --
	(194.35,223.99) --
	(194.73,223.19) --
	(195.12,222.39) --
	(195.51,221.57) --
	(195.90,220.74) --
	(196.28,219.90) --
	(196.67,219.05) --
	(197.06,218.19) --
	(197.44,217.32) --
	(197.83,216.44) --
	(198.22,215.56) --
	(198.61,214.66) --
	(198.99,213.75) --
	(199.38,212.84) --
	(199.77,211.91) --
	(200.15,210.98) --
	(200.54,210.04) --
	(200.93,209.09) --
	(201.31,208.14) --
	(201.70,207.17) --
	(202.09,206.20) --
	(202.48,205.23) --
	(202.86,204.25) --
	(203.25,203.26) --
	(203.64,202.26) --
	(204.02,201.26) --
	(204.41,200.25) --
	(204.80,199.24) --
	(205.19,198.22) --
	(205.57,197.20) --
	(205.96,196.17) --
	(206.35,195.14) --
	(206.73,194.11) --
	(207.12,193.07) --
	(207.51,192.03) --
	(207.89,190.98) --
	(208.28,189.93) --
	(208.67,188.88) --
	(209.06,187.82) --
	(209.44,186.77) --
	(209.83,185.71) --
	(210.22,184.65) --
	(210.60,183.58) --
	(210.99,182.52) --
	(211.38,181.45) --
	(211.77,180.38) --
	(212.15,179.32) --
	(212.54,178.25) --
	(212.93,177.18) --
	(213.31,176.11) --
	(213.70,175.04) --
	(214.09,173.97) --
	(214.47,172.90) --
	(214.86,171.83) --
	(215.25,170.77) --
	(215.64,169.70) --
	(216.02,168.63) --
	(216.41,167.57) --
	(216.80,166.51) --
	(217.18,165.45) --
	(217.57,164.39) --
	(217.96,163.33) --
	(218.35,162.28) --
	(218.73,161.23) --
	(219.12,160.18) --
	(219.51,159.14) --
	(219.89,158.09) --
	(220.28,157.05) --
	(220.67,156.02) --
	(221.05,154.99) --
	(221.44,153.96) --
	(221.83,152.94) --
	(222.22,151.92) --
	(222.60,150.90) --
	(222.99,149.89) --
	(223.38,148.88) --
	(223.76,147.88) --
	(224.15,146.88) --
	(224.54,145.89) --
	(224.93,144.90) --
	(225.31,143.92) --
	(225.70,142.94) --
	(226.09,141.97) --
	(226.47,141.01) --
	(226.86,140.05) --
	(227.25,139.09) --
	(227.63,138.14) --
	(228.02,137.20) --
	(228.41,136.27) --
	(228.80,135.34) --
	(229.18,134.41) --
	(229.57,133.50) --
	(229.96,132.59) --
	(230.34,131.68) --
	(230.73,130.79) --
	(231.12,129.90) --
	(231.51,129.01) --
	(231.89,128.14) --
	(232.28,127.27) --
	(232.67,126.40) --
	(233.05,125.55) --
	(233.44,124.70) --
	(233.83,123.86) --
	(234.21,123.03) --
	(234.60,122.20) --
	(234.99,121.38) --
	(235.38,120.57) --
	(235.76,119.77) --
	(236.15,118.97) --
	(236.54,118.18) --
	(236.92,117.40) --
	(237.31,116.63) --
	(237.70,115.87) --
	(238.09,115.11) --
	(238.47,114.36) --
	(238.86,113.62) --
	(239.25,112.88) --
	(239.63,112.16) --
	(240.02,111.44) --
	(240.41,110.73) --
	(240.79,110.02) --
	(241.18,109.33) --
	(241.57,108.64) --
	(241.96,107.96) --
	(242.34,107.29) --
	(242.73,106.63) --
	(243.12,105.97) --
	(243.50,105.32) --
	(243.89,104.68) --
	(244.28,104.05) --
	(244.67,103.43) --
	(245.05,102.81) --
	(245.44,102.20) --
	(245.83,101.60) --
	(246.21,101.01) --
	(246.60,100.42) --
	(246.99, 99.85) --
	(247.37, 99.28) --
	(247.76, 98.71) --
	(248.15, 98.16) --
	(248.54, 97.61) --
	(248.92, 97.07) --
	(249.31, 96.54) --
	(249.70, 96.01) --
	(250.08, 95.50) --
	(250.47, 94.99) --
	(250.86, 94.48) --
	(251.25, 93.99) --
	(251.63, 93.50) --
	(252.02, 93.02) --
	(252.41, 92.55) --
	(252.79, 92.08) --
	(253.18, 91.62) --
	(253.57, 91.17) --
	(253.95, 90.72) --
	(254.34, 90.28) --
	(254.73, 89.85) --
	(255.12, 89.43) --
	(255.50, 89.01) --
	(255.89, 88.60) --
	(256.28, 88.19) --
	(256.66, 87.79) --
	(257.05, 87.40) --
	(257.44, 87.02) --
	(257.83, 86.64) --
	(258.21, 86.26) --
	(258.60, 85.90) --
	(258.99, 85.54) --
	(259.37, 85.18) --
	(259.76, 84.84) --
	(260.15, 84.49) --
	(260.53, 84.16) --
	(260.92, 83.83) --
	(261.31, 83.50) --
	(261.70, 83.19) --
	(262.08, 82.87) --
	(262.47, 82.57) --
	(262.86, 82.26) --
	(263.24, 81.97) --
	(263.63, 81.68) --
	(264.02, 81.39) --
	(264.41, 81.11) --
	(264.79, 80.84) --
	(265.18, 80.57) --
	(265.57, 80.30) --
	(265.95, 80.04) --
	(266.34, 79.79) --
	(266.73, 79.54) --
	(267.11, 79.29) --
	(267.50, 79.05) --
	(267.89, 78.82) --
	(268.28, 78.59) --
	(268.66, 78.36) --
	(269.05, 78.14) --
	(269.44, 77.92) --
	(269.82, 77.71) --
	(270.21, 77.50) --
	(270.60, 77.30) --
	(270.99, 77.10) --
	(271.37, 76.90) --
	(271.76, 76.71) --
	(272.15, 76.52) --
	(272.53, 76.34) --
	(272.92, 76.16) --
	(273.31, 75.98) --
	(273.69, 75.81) --
	(274.08, 75.64) --
	(274.47, 75.48) --
	(274.86, 75.32) --
	(275.24, 75.16) --
	(275.63, 75.01) --
	(276.02, 74.85) --
	(276.40, 74.71) --
	(276.79, 74.56) --
	(277.18, 74.42) --
	(277.57, 74.28) --
	(277.95, 74.15) --
	(278.34, 74.02) --
	(278.73, 73.89) --
	(279.11, 73.76) --
	(279.50, 73.64) --
	(279.89, 73.52) --
	(280.27, 73.40) --
	(280.66, 73.29) --
	(281.05, 73.18) --
	(281.44, 73.07) --
	(281.82, 72.96) --
	(282.21, 72.86) --
	(282.60, 72.75) --
	(282.98, 72.65) --
	(283.37, 72.56) --
	(283.76, 72.46) --
	(284.15, 72.37) --
	(284.53, 72.28) --
	(284.92, 72.19) --
	(285.31, 72.11) --
	(285.69, 72.03) --
	(286.08, 71.94) --
	(286.47, 71.86) --
	(286.85, 71.79) --
	(287.24, 71.71) --
	(287.63, 71.64) --
	(288.02, 71.57) --
	(288.40, 71.50) --
	(288.79, 71.43) --
	(289.18, 71.36) --
	(289.56, 71.30) --
	(289.95, 71.24) --
	(290.34, 71.17) --
	(290.73, 71.12) --
	(291.11, 71.06) --
	(291.50, 71.00) --
	(291.89, 70.95) --
	(292.27, 70.89) --
	(292.66, 70.84) --
	(293.05, 70.79) --
	(293.43, 70.74) --
	(293.82, 70.69) --
	(294.21, 70.65) --
	(294.60, 70.60) --
	(294.98, 70.56) --
	(295.37, 70.51) --
	(295.76, 70.47) --
	(296.14, 70.43) --
	(296.53, 70.39) --
	(296.92, 70.35) --
	(297.31, 70.32) --
	(297.69, 70.28) --
	(298.08, 70.24) --
	(298.47, 70.21) --
	(298.85, 70.18) --
	(299.24, 70.14) --
	(299.63, 70.11) --
	(300.01, 70.08) --
	(300.40, 70.05) --
	(300.79, 70.02) --
	(301.18, 70.00) --
	(301.56, 69.97) --
	(301.95, 69.94) --
	(302.34, 69.92) --
	(302.72, 69.89) --
	(303.11, 69.87) --
	(303.50, 69.85) --
	(303.89, 69.82) --
	(304.27, 69.80) --
	(304.66, 69.78) --
	(305.05, 69.76) --
	(305.43, 69.74) --
	(305.82, 69.72) --
	(306.21, 69.70) --
	(306.60, 69.68) --
	(306.98, 69.66) --
	(307.37, 69.65) --
	(307.76, 69.63) --
	(308.14, 69.62) --
	(308.53, 69.60) --
	(308.92, 69.58) --
	(309.30, 69.57) --
	(309.69, 69.56) --
	(310.08, 69.54) --
	(310.47, 69.53) --
	(310.85, 69.52) --
	(311.24, 69.50) --
	(311.63, 69.49) --
	(312.01, 69.48) --
	(312.40, 69.47) --
	(312.79, 69.46) --
	(313.18, 69.45) --
	(313.56, 69.44) --
	(313.95, 69.43) --
	(314.34, 69.42) --
	(314.72, 69.41) --
	(315.11, 69.40) --
	(315.50, 69.39) --
	(315.88, 69.38) --
	(316.27, 69.37) --
	(316.66, 69.37) --
	(317.05, 69.36) --
	(317.43, 69.35) --
	(317.82, 69.34) --
	(318.21, 69.34) --
	(318.59, 69.33) --
	(318.98, 69.32) --
	(319.37, 69.32) --
	(319.76, 69.31) --
	(320.14, 69.31) --
	(320.53, 69.30) --
	(320.92, 69.30) --
	(321.30, 69.29) --
	(321.69, 69.29) --
	(322.08, 69.28) --
	(322.46, 69.28) --
	(322.85, 69.27) --
	(323.24, 69.27) --
	(323.63, 69.26) --
	(324.01, 69.26) --
	(324.40, 69.25) --
	(324.79, 69.25) --
	(325.17, 69.25) --
	(325.56, 69.24) --
	(325.95, 69.24) --
	(326.34, 69.24) --
	(326.72, 69.23) --
	(327.11, 69.23) --
	(327.50, 69.23) --
	(327.88, 69.23) --
	(328.27, 69.22) --
	(328.66, 69.22) --
	(329.04, 69.22) --
	(329.43, 69.22);
\end{scope}
\end{tikzpicture}

%% file: density-100.tex
\begin{tikzpicture}[x=1pt,y=1pt]
\definecolor{fillColor}{RGB}{255,255,255}
\path[use as bounding box,fill=fillColor,fill opacity=0.00] (0,0) rectangle (325.21,325.21);
\begin{scope}
\path[clip] (  0.00,  0.00) rectangle (325.21,325.21);
\definecolor{drawColor}{RGB}{0,0,0}

\node[text=drawColor,anchor=base,inner sep=0pt, outer sep=0pt, scale=  1.00] at (174.61, 15.60) {$\mathcal I_n$};

\node[text=drawColor,rotate= 90.00,anchor=base,inner sep=0pt, outer sep=0pt, scale=  1.00] at ( 10.80,168.61) {Density};
\end{scope}
\begin{scope}
\path[clip] (  0.00,  0.00) rectangle (325.21,325.21);
\definecolor{drawColor}{RGB}{0,0,0}

\path[draw=drawColor,line width= 0.4pt,line join=round,line cap=round] ( 58.49, 61.20) -- (290.73, 61.20);

\path[draw=drawColor,line width= 0.4pt,line join=round,line cap=round] ( 58.49, 61.20) -- ( 58.49, 55.20);

\path[draw=drawColor,line width= 0.4pt,line join=round,line cap=round] ( 97.20, 61.20) -- ( 97.20, 55.20);

\path[draw=drawColor,line width= 0.4pt,line join=round,line cap=round] (135.90, 61.20) -- (135.90, 55.20);

\path[draw=drawColor,line width= 0.4pt,line join=round,line cap=round] (174.61, 61.20) -- (174.61, 55.20);

\path[draw=drawColor,line width= 0.4pt,line join=round,line cap=round] (213.31, 61.20) -- (213.31, 55.20);

\path[draw=drawColor,line width= 0.4pt,line join=round,line cap=round] (252.02, 61.20) -- (252.02, 55.20);

\path[draw=drawColor,line width= 0.4pt,line join=round,line cap=round] (290.73, 61.20) -- (290.73, 55.20);

\node[text=drawColor,anchor=base,inner sep=0pt, outer sep=0pt, scale=  1.00] at ( 58.49, 39.60) {-3};

\node[text=drawColor,anchor=base,inner sep=0pt, outer sep=0pt, scale=  1.00] at ( 97.20, 39.60) {-2};

\node[text=drawColor,anchor=base,inner sep=0pt, outer sep=0pt, scale=  1.00] at (135.90, 39.60) {-1};

\node[text=drawColor,anchor=base,inner sep=0pt, outer sep=0pt, scale=  1.00] at (174.61, 39.60) {0};

\node[text=drawColor,anchor=base,inner sep=0pt, outer sep=0pt, scale=  1.00] at (213.31, 39.60) {1};

\node[text=drawColor,anchor=base,inner sep=0pt, outer sep=0pt, scale=  1.00] at (252.02, 39.60) {2};

\node[text=drawColor,anchor=base,inner sep=0pt, outer sep=0pt, scale=  1.00] at (290.73, 39.60) {3};

\path[draw=drawColor,line width= 0.4pt,line join=round,line cap=round] ( 49.20, 69.16) -- ( 49.20,245.96);

\path[draw=drawColor,line width= 0.4pt,line join=round,line cap=round] ( 49.20, 69.16) -- ( 43.20, 69.16);

\path[draw=drawColor,line width= 0.4pt,line join=round,line cap=round] ( 49.20,113.36) -- ( 43.20,113.36);

\path[draw=drawColor,line width= 0.4pt,line join=round,line cap=round] ( 49.20,157.56) -- ( 43.20,157.56);

\path[draw=drawColor,line width= 0.4pt,line join=round,line cap=round] ( 49.20,201.76) -- ( 43.20,201.76);

\path[draw=drawColor,line width= 0.4pt,line join=round,line cap=round] ( 49.20,245.96) -- ( 43.20,245.96);

\node[text=drawColor,rotate= 90.00,anchor=base,inner sep=0pt, outer sep=0pt, scale=  1.00] at ( 34.80, 69.16) {0.0};

\node[text=drawColor,rotate= 90.00,anchor=base,inner sep=0pt, outer sep=0pt, scale=  1.00] at ( 34.80,113.36) {0.1};

\node[text=drawColor,rotate= 90.00,anchor=base,inner sep=0pt, outer sep=0pt, scale=  1.00] at ( 34.80,157.56) {0.2};

\node[text=drawColor,rotate= 90.00,anchor=base,inner sep=0pt, outer sep=0pt, scale=  1.00] at ( 34.80,201.76) {0.3};

\node[text=drawColor,rotate= 90.00,anchor=base,inner sep=0pt, outer sep=0pt, scale=  1.00] at ( 34.80,245.96) {0.4};
\end{scope}
\begin{scope}
\path[clip] ( 49.20, 61.20) rectangle (300.01,276.01);
\definecolor{drawColor}{RGB}{0,0,0}
\definecolor{fillColor}{RGB}{211,211,211}

\path[draw=drawColor,line width= 0.4pt,line join=round,line cap=round,fill=fillColor] ( 35.27, 69.16) rectangle ( 43.01, 69.60);

\path[draw=drawColor,line width= 0.4pt,line join=round,line cap=round,fill=fillColor] ( 43.01, 69.16) rectangle ( 50.75, 70.48);

\path[draw=drawColor,line width= 0.4pt,line join=round,line cap=round,fill=fillColor] ( 50.75, 69.16) rectangle ( 58.49, 71.81);

\path[draw=drawColor,line width= 0.4pt,line join=round,line cap=round,fill=fillColor] ( 58.49, 69.16) rectangle ( 66.23, 72.25);

\path[draw=drawColor,line width= 0.4pt,line join=round,line cap=round,fill=fillColor] ( 66.23, 69.16) rectangle ( 73.97, 76.23);

\path[draw=drawColor,line width= 0.4pt,line join=round,line cap=round,fill=fillColor] ( 73.97, 69.16) rectangle ( 81.71, 77.11);

\path[draw=drawColor,line width= 0.4pt,line join=round,line cap=round,fill=fillColor] ( 81.71, 69.16) rectangle ( 89.45, 82.42);

\path[draw=drawColor,line width= 0.4pt,line join=round,line cap=round,fill=fillColor] ( 89.45, 69.16) rectangle ( 97.20, 86.39);

\path[draw=drawColor,line width= 0.4pt,line join=round,line cap=round,fill=fillColor] ( 97.20, 69.16) rectangle (104.94,103.63);

\path[draw=drawColor,line width= 0.4pt,line join=round,line cap=round,fill=fillColor] (104.94, 69.16) rectangle (112.68,102.31);

\path[draw=drawColor,line width= 0.4pt,line join=round,line cap=round,fill=fillColor] (112.68, 69.16) rectangle (120.42,126.17);

\path[draw=drawColor,line width= 0.4pt,line join=round,line cap=round,fill=fillColor] (120.42, 69.16) rectangle (128.16,145.62);

\path[draw=drawColor,line width= 0.4pt,line join=round,line cap=round,fill=fillColor] (128.16, 69.16) rectangle (135.90,155.35);

\path[draw=drawColor,line width= 0.4pt,line join=round,line cap=round,fill=fillColor] (135.90, 69.16) rectangle (143.64,180.54);

\path[draw=drawColor,line width= 0.4pt,line join=round,line cap=round,fill=fillColor] (143.64, 69.16) rectangle (151.38,222.09);

\path[draw=drawColor,line width= 0.4pt,line join=round,line cap=round,fill=fillColor] (151.38, 69.16) rectangle (159.13,223.86);

\path[draw=drawColor,line width= 0.4pt,line join=round,line cap=round,fill=fillColor] (159.13, 69.16) rectangle (166.87,240.21);

\path[draw=drawColor,line width= 0.4pt,line join=round,line cap=round,fill=fillColor] (166.87, 69.16) rectangle (174.61,252.59);

\path[draw=drawColor,line width= 0.4pt,line join=round,line cap=round,fill=fillColor] (174.61, 69.16) rectangle (182.35,239.77);

\path[draw=drawColor,line width= 0.4pt,line join=round,line cap=round,fill=fillColor] (182.35, 69.16) rectangle (190.09,238.44);

\path[draw=drawColor,line width= 0.4pt,line join=round,line cap=round,fill=fillColor] (190.09, 69.16) rectangle (197.83,219.44);

\path[draw=drawColor,line width= 0.4pt,line join=round,line cap=round,fill=fillColor] (197.83, 69.16) rectangle (205.57,212.37);

\path[draw=drawColor,line width= 0.4pt,line join=round,line cap=round,fill=fillColor] (205.57, 69.16) rectangle (213.31,196.01);

\path[draw=drawColor,line width= 0.4pt,line join=round,line cap=round,fill=fillColor] (213.31, 69.16) rectangle (221.05,159.33);

\path[draw=drawColor,line width= 0.4pt,line join=round,line cap=round,fill=fillColor] (221.05, 69.16) rectangle (228.80,136.34);

\path[draw=drawColor,line width= 0.4pt,line join=round,line cap=round,fill=fillColor] (228.80, 69.16) rectangle (236.54,125.29);

\path[draw=drawColor,line width= 0.4pt,line join=round,line cap=round,fill=fillColor] (236.54, 69.16) rectangle (244.28,112.47);

\path[draw=drawColor,line width= 0.4pt,line join=round,line cap=round,fill=fillColor] (244.28, 69.16) rectangle (252.02,100.54);

\path[draw=drawColor,line width= 0.4pt,line join=round,line cap=round,fill=fillColor] (252.02, 69.16) rectangle (259.76, 88.60);

\path[draw=drawColor,line width= 0.4pt,line join=round,line cap=round,fill=fillColor] (259.76, 69.16) rectangle (267.50, 82.86);

\path[draw=drawColor,line width= 0.4pt,line join=round,line cap=round,fill=fillColor] (267.50, 69.16) rectangle (275.24, 75.34);

\path[draw=drawColor,line width= 0.4pt,line join=round,line cap=round,fill=fillColor] (275.24, 69.16) rectangle (282.98, 73.58);

\path[draw=drawColor,line width= 0.4pt,line join=round,line cap=round,fill=fillColor] (282.98, 69.16) rectangle (290.73, 71.37);

\path[draw=drawColor,line width= 0.4pt,line join=round,line cap=round,fill=fillColor] (290.73, 69.16) rectangle (298.47, 69.16);

\path[draw=drawColor,line width= 0.4pt,line join=round,line cap=round,fill=fillColor] (298.47, 69.16) rectangle (306.21, 70.48);

\path[draw=drawColor,line width= 0.4pt,line join=round,line cap=round,fill=fillColor] (306.21, 69.16) rectangle (313.95, 69.60);
\definecolor{drawColor}{RGB}{255,0,0}

\path[draw=drawColor,line width= 0.4pt,line join=round,line cap=round] ( 19.78, 69.22) --
	( 20.17, 69.22) --
	( 20.56, 69.22) --
	( 20.94, 69.22) --
	( 21.33, 69.23) --
	( 21.72, 69.23) --
	( 22.11, 69.23) --
	( 22.49, 69.23) --
	( 22.88, 69.24) --
	( 23.27, 69.24) --
	( 23.65, 69.24) --
	( 24.04, 69.25) --
	( 24.43, 69.25) --
	( 24.82, 69.25) --
	( 25.20, 69.26) --
	( 25.59, 69.26) --
	( 25.98, 69.27) --
	( 26.36, 69.27) --
	( 26.75, 69.28) --
	( 27.14, 69.28) --
	( 27.52, 69.29) --
	( 27.91, 69.29) --
	( 28.30, 69.30) --
	( 28.69, 69.30) --
	( 29.07, 69.31) --
	( 29.46, 69.31) --
	( 29.85, 69.32) --
	( 30.23, 69.32) --
	( 30.62, 69.33) --
	( 31.01, 69.34) --
	( 31.40, 69.34) --
	( 31.78, 69.35) --
	( 32.17, 69.36) --
	( 32.56, 69.37) --
	( 32.94, 69.37) --
	( 33.33, 69.38) --
	( 33.72, 69.39) --
	( 34.10, 69.40) --
	( 34.49, 69.41) --
	( 34.88, 69.42) --
	( 35.27, 69.43) --
	( 35.65, 69.44) --
	( 36.04, 69.45) --
	( 36.43, 69.46) --
	( 36.81, 69.47) --
	( 37.20, 69.48) --
	( 37.59, 69.49) --
	( 37.98, 69.50) --
	( 38.36, 69.52) --
	( 38.75, 69.53) --
	( 39.14, 69.54) --
	( 39.52, 69.56) --
	( 39.91, 69.57) --
	( 40.30, 69.58) --
	( 40.68, 69.60) --
	( 41.07, 69.62) --
	( 41.46, 69.63) --
	( 41.85, 69.65) --
	( 42.23, 69.66) --
	( 42.62, 69.68) --
	( 43.01, 69.70) --
	( 43.39, 69.72) --
	( 43.78, 69.74) --
	( 44.17, 69.76) --
	( 44.56, 69.78) --
	( 44.94, 69.80) --
	( 45.33, 69.82) --
	( 45.72, 69.85) --
	( 46.10, 69.87) --
	( 46.49, 69.89) --
	( 46.88, 69.92) --
	( 47.26, 69.94) --
	( 47.65, 69.97) --
	( 48.04, 70.00) --
	( 48.43, 70.02) --
	( 48.81, 70.05) --
	( 49.20, 70.08) --
	( 49.59, 70.11) --
	( 49.97, 70.14) --
	( 50.36, 70.18) --
	( 50.75, 70.21) --
	( 51.14, 70.24) --
	( 51.52, 70.28) --
	( 51.91, 70.32) --
	( 52.30, 70.35) --
	( 52.68, 70.39) --
	( 53.07, 70.43) --
	( 53.46, 70.47) --
	( 53.84, 70.51) --
	( 54.23, 70.56) --
	( 54.62, 70.60) --
	( 55.01, 70.65) --
	( 55.39, 70.69) --
	( 55.78, 70.74) --
	( 56.17, 70.79) --
	( 56.55, 70.84) --
	( 56.94, 70.89) --
	( 57.33, 70.95) --
	( 57.72, 71.00) --
	( 58.10, 71.06) --
	( 58.49, 71.12) --
	( 58.88, 71.17) --
	( 59.26, 71.24) --
	( 59.65, 71.30) --
	( 60.04, 71.36) --
	( 60.42, 71.43) --
	( 60.81, 71.50) --
	( 61.20, 71.57) --
	( 61.59, 71.64) --
	( 61.97, 71.71) --
	( 62.36, 71.79) --
	( 62.75, 71.86) --
	( 63.13, 71.94) --
	( 63.52, 72.03) --
	( 63.91, 72.11) --
	( 64.30, 72.19) --
	( 64.68, 72.28) --
	( 65.07, 72.37) --
	( 65.46, 72.46) --
	( 65.84, 72.56) --
	( 66.23, 72.65) --
	( 66.62, 72.75) --
	( 67.00, 72.86) --
	( 67.39, 72.96) --
	( 67.78, 73.07) --
	( 68.17, 73.18) --
	( 68.55, 73.29) --
	( 68.94, 73.40) --
	( 69.33, 73.52) --
	( 69.71, 73.64) --
	( 70.10, 73.76) --
	( 70.49, 73.89) --
	( 70.88, 74.02) --
	( 71.26, 74.15) --
	( 71.65, 74.28) --
	( 72.04, 74.42) --
	( 72.42, 74.56) --
	( 72.81, 74.71) --
	( 73.20, 74.85) --
	( 73.58, 75.01) --
	( 73.97, 75.16) --
	( 74.36, 75.32) --
	( 74.75, 75.48) --
	( 75.13, 75.64) --
	( 75.52, 75.81) --
	( 75.91, 75.98) --
	( 76.29, 76.16) --
	( 76.68, 76.34) --
	( 77.07, 76.52) --
	( 77.46, 76.71) --
	( 77.84, 76.90) --
	( 78.23, 77.10) --
	( 78.62, 77.30) --
	( 79.00, 77.50) --
	( 79.39, 77.71) --
	( 79.78, 77.92) --
	( 80.16, 78.14) --
	( 80.55, 78.36) --
	( 80.94, 78.59) --
	( 81.33, 78.82) --
	( 81.71, 79.05) --
	( 82.10, 79.29) --
	( 82.49, 79.54) --
	( 82.87, 79.79) --
	( 83.26, 80.04) --
	( 83.65, 80.30) --
	( 84.04, 80.57) --
	( 84.42, 80.84) --
	( 84.81, 81.11) --
	( 85.20, 81.39) --
	( 85.58, 81.68) --
	( 85.97, 81.97) --
	( 86.36, 82.26) --
	( 86.74, 82.57) --
	( 87.13, 82.87) --
	( 87.52, 83.19) --
	( 87.91, 83.50) --
	( 88.29, 83.83) --
	( 88.68, 84.16) --
	( 89.07, 84.49) --
	( 89.45, 84.84) --
	( 89.84, 85.18) --
	( 90.23, 85.54) --
	( 90.62, 85.90) --
	( 91.00, 86.26) --
	( 91.39, 86.64) --
	( 91.78, 87.02) --
	( 92.16, 87.40) --
	( 92.55, 87.79) --
	( 92.94, 88.19) --
	( 93.32, 88.60) --
	( 93.71, 89.01) --
	( 94.10, 89.43) --
	( 94.49, 89.85) --
	( 94.87, 90.28) --
	( 95.26, 90.72) --
	( 95.65, 91.17) --
	( 96.03, 91.62) --
	( 96.42, 92.08) --
	( 96.81, 92.55) --
	( 97.20, 93.02) --
	( 97.58, 93.50) --
	( 97.97, 93.99) --
	( 98.36, 94.48) --
	( 98.74, 94.99) --
	( 99.13, 95.50) --
	( 99.52, 96.01) --
	( 99.90, 96.54) --
	(100.29, 97.07) --
	(100.68, 97.61) --
	(101.07, 98.16) --
	(101.45, 98.71) --
	(101.84, 99.28) --
	(102.23, 99.85) --
	(102.61,100.42) --
	(103.00,101.01) --
	(103.39,101.60) --
	(103.78,102.20) --
	(104.16,102.81) --
	(104.55,103.43) --
	(104.94,104.05) --
	(105.32,104.68) --
	(105.71,105.32) --
	(106.10,105.97) --
	(106.48,106.63) --
	(106.87,107.29) --
	(107.26,107.96) --
	(107.65,108.64) --
	(108.03,109.33) --
	(108.42,110.02) --
	(108.81,110.73) --
	(109.19,111.44) --
	(109.58,112.16) --
	(109.97,112.88) --
	(110.36,113.62) --
	(110.74,114.36) --
	(111.13,115.11) --
	(111.52,115.87) --
	(111.90,116.63) --
	(112.29,117.40) --
	(112.68,118.18) --
	(113.06,118.97) --
	(113.45,119.77) --
	(113.84,120.57) --
	(114.23,121.38) --
	(114.61,122.20) --
	(115.00,123.03) --
	(115.39,123.86) --
	(115.77,124.70) --
	(116.16,125.55) --
	(116.55,126.40) --
	(116.94,127.27) --
	(117.32,128.14) --
	(117.71,129.01) --
	(118.10,129.90) --
	(118.48,130.79) --
	(118.87,131.68) --
	(119.26,132.59) --
	(119.64,133.50) --
	(120.03,134.41) --
	(120.42,135.34) --
	(120.81,136.27) --
	(121.19,137.20) --
	(121.58,138.14) --
	(121.97,139.09) --
	(122.35,140.05) --
	(122.74,141.01) --
	(123.13,141.97) --
	(123.52,142.94) --
	(123.90,143.92) --
	(124.29,144.90) --
	(124.68,145.89) --
	(125.06,146.88) --
	(125.45,147.88) --
	(125.84,148.88) --
	(126.22,149.89) --
	(126.61,150.90) --
	(127.00,151.92) --
	(127.39,152.94) --
	(127.77,153.96) --
	(128.16,154.99) --
	(128.55,156.02) --
	(128.93,157.05) --
	(129.32,158.09) --
	(129.71,159.14) --
	(130.10,160.18) --
	(130.48,161.23) --
	(130.87,162.28) --
	(131.26,163.33) --
	(131.64,164.39) --
	(132.03,165.45) --
	(132.42,166.51) --
	(132.81,167.57) --
	(133.19,168.63) --
	(133.58,169.70) --
	(133.97,170.77) --
	(134.35,171.83) --
	(134.74,172.90) --
	(135.13,173.97) --
	(135.51,175.04) --
	(135.90,176.11) --
	(136.29,177.18) --
	(136.68,178.25) --
	(137.06,179.32) --
	(137.45,180.38) --
	(137.84,181.45) --
	(138.22,182.52) --
	(138.61,183.58) --
	(139.00,184.65) --
	(139.39,185.71) --
	(139.77,186.77) --
	(140.16,187.82) --
	(140.55,188.88) --
	(140.93,189.93) --
	(141.32,190.98) --
	(141.71,192.03) --
	(142.09,193.07) --
	(142.48,194.11) --
	(142.87,195.14) --
	(143.26,196.17) --
	(143.64,197.20) --
	(144.03,198.22) --
	(144.42,199.24) --
	(144.80,200.25) --
	(145.19,201.26) --
	(145.58,202.26) --
	(145.97,203.26) --
	(146.35,204.25) --
	(146.74,205.23) --
	(147.13,206.20) --
	(147.51,207.17) --
	(147.90,208.14) --
	(148.29,209.09) --
	(148.67,210.04) --
	(149.06,210.98) --
	(149.45,211.91) --
	(149.84,212.84) --
	(150.22,213.75) --
	(150.61,214.66) --
	(151.00,215.56) --
	(151.38,216.44) --
	(151.77,217.32) --
	(152.16,218.19) --
	(152.55,219.05) --
	(152.93,219.90) --
	(153.32,220.74) --
	(153.71,221.57) --
	(154.09,222.39) --
	(154.48,223.19) --
	(154.87,223.99) --
	(155.25,224.77) --
	(155.64,225.54) --
	(156.03,226.30) --
	(156.42,227.05) --
	(156.80,227.79) --
	(157.19,228.51) --
	(157.58,229.22) --
	(157.96,229.92) --
	(158.35,230.60) --
	(158.74,231.28) --
	(159.13,231.93) --
	(159.51,232.58) --
	(159.90,233.21) --
	(160.29,233.82) --
	(160.67,234.43) --
	(161.06,235.01) --
	(161.45,235.59) --
	(161.83,236.15) --
	(162.22,236.69) --
	(162.61,237.22) --
	(163.00,237.73) --
	(163.38,238.23) --
	(163.77,238.71) --
	(164.16,239.18) --
	(164.54,239.63) --
	(164.93,240.07) --
	(165.32,240.49) --
	(165.71,240.89) --
	(166.09,241.27) --
	(166.48,241.65) --
	(166.87,242.00) --
	(167.25,242.34) --
	(167.64,242.66) --
	(168.03,242.96) --
	(168.41,243.25) --
	(168.80,243.52) --
	(169.19,243.77) --
	(169.58,244.01) --
	(169.96,244.23) --
	(170.35,244.43) --
	(170.74,244.61) --
	(171.12,244.78) --
	(171.51,244.93) --
	(171.90,245.06) --
	(172.29,245.17) --
	(172.67,245.27) --
	(173.06,245.35) --
	(173.45,245.41) --
	(173.83,245.46) --
	(174.22,245.48) --
	(174.61,245.49) --
	(174.99,245.48) --
	(175.38,245.46) --
	(175.77,245.41) --
	(176.16,245.35) --
	(176.54,245.27) --
	(176.93,245.17) --
	(177.32,245.06) --
	(177.70,244.93) --
	(178.09,244.78) --
	(178.48,244.61) --
	(178.87,244.43) --
	(179.25,244.23) --
	(179.64,244.01) --
	(180.03,243.77) --
	(180.41,243.52) --
	(180.80,243.25) --
	(181.19,242.96) --
	(181.57,242.66) --
	(181.96,242.34) --
	(182.35,242.00) --
	(182.74,241.65) --
	(183.12,241.27) --
	(183.51,240.89) --
	(183.90,240.49) --
	(184.28,240.07) --
	(184.67,239.63) --
	(185.06,239.18) --
	(185.45,238.71) --
	(185.83,238.23) --
	(186.22,237.73) --
	(186.61,237.22) --
	(186.99,236.69) --
	(187.38,236.15) --
	(187.77,235.59) --
	(188.15,235.01) --
	(188.54,234.43) --
	(188.93,233.82) --
	(189.32,233.21) --
	(189.70,232.58) --
	(190.09,231.93) --
	(190.48,231.28) --
	(190.86,230.60) --
	(191.25,229.92) --
	(191.64,229.22) --
	(192.03,228.51) --
	(192.41,227.79) --
	(192.80,227.05) --
	(193.19,226.30) --
	(193.57,225.54) --
	(193.96,224.77) --
	(194.35,223.99) --
	(194.73,223.19) --
	(195.12,222.39) --
	(195.51,221.57) --
	(195.90,220.74) --
	(196.28,219.90) --
	(196.67,219.05) --
	(197.06,218.19) --
	(197.44,217.32) --
	(197.83,216.44) --
	(198.22,215.56) --
	(198.61,214.66) --
	(198.99,213.75) --
	(199.38,212.84) --
	(199.77,211.91) --
	(200.15,210.98) --
	(200.54,210.04) --
	(200.93,209.09) --
	(201.31,208.14) --
	(201.70,207.17) --
	(202.09,206.20) --
	(202.48,205.23) --
	(202.86,204.25) --
	(203.25,203.26) --
	(203.64,202.26) --
	(204.02,201.26) --
	(204.41,200.25) --
	(204.80,199.24) --
	(205.19,198.22) --
	(205.57,197.20) --
	(205.96,196.17) --
	(206.35,195.14) --
	(206.73,194.11) --
	(207.12,193.07) --
	(207.51,192.03) --
	(207.89,190.98) --
	(208.28,189.93) --
	(208.67,188.88) --
	(209.06,187.82) --
	(209.44,186.77) --
	(209.83,185.71) --
	(210.22,184.65) --
	(210.60,183.58) --
	(210.99,182.52) --
	(211.38,181.45) --
	(211.77,180.38) --
	(212.15,179.32) --
	(212.54,178.25) --
	(212.93,177.18) --
	(213.31,176.11) --
	(213.70,175.04) --
	(214.09,173.97) --
	(214.47,172.90) --
	(214.86,171.83) --
	(215.25,170.77) --
	(215.64,169.70) --
	(216.02,168.63) --
	(216.41,167.57) --
	(216.80,166.51) --
	(217.18,165.45) --
	(217.57,164.39) --
	(217.96,163.33) --
	(218.35,162.28) --
	(218.73,161.23) --
	(219.12,160.18) --
	(219.51,159.14) --
	(219.89,158.09) --
	(220.28,157.05) --
	(220.67,156.02) --
	(221.05,154.99) --
	(221.44,153.96) --
	(221.83,152.94) --
	(222.22,151.92) --
	(222.60,150.90) --
	(222.99,149.89) --
	(223.38,148.88) --
	(223.76,147.88) --
	(224.15,146.88) --
	(224.54,145.89) --
	(224.93,144.90) --
	(225.31,143.92) --
	(225.70,142.94) --
	(226.09,141.97) --
	(226.47,141.01) --
	(226.86,140.05) --
	(227.25,139.09) --
	(227.63,138.14) --
	(228.02,137.20) --
	(228.41,136.27) --
	(228.80,135.34) --
	(229.18,134.41) --
	(229.57,133.50) --
	(229.96,132.59) --
	(230.34,131.68) --
	(230.73,130.79) --
	(231.12,129.90) --
	(231.51,129.01) --
	(231.89,128.14) --
	(232.28,127.27) --
	(232.67,126.40) --
	(233.05,125.55) --
	(233.44,124.70) --
	(233.83,123.86) --
	(234.21,123.03) --
	(234.60,122.20) --
	(234.99,121.38) --
	(235.38,120.57) --
	(235.76,119.77) --
	(236.15,118.97) --
	(236.54,118.18) --
	(236.92,117.40) --
	(237.31,116.63) --
	(237.70,115.87) --
	(238.09,115.11) --
	(238.47,114.36) --
	(238.86,113.62) --
	(239.25,112.88) --
	(239.63,112.16) --
	(240.02,111.44) --
	(240.41,110.73) --
	(240.79,110.02) --
	(241.18,109.33) --
	(241.57,108.64) --
	(241.96,107.96) --
	(242.34,107.29) --
	(242.73,106.63) --
	(243.12,105.97) --
	(243.50,105.32) --
	(243.89,104.68) --
	(244.28,104.05) --
	(244.67,103.43) --
	(245.05,102.81) --
	(245.44,102.20) --
	(245.83,101.60) --
	(246.21,101.01) --
	(246.60,100.42) --
	(246.99, 99.85) --
	(247.37, 99.28) --
	(247.76, 98.71) --
	(248.15, 98.16) --
	(248.54, 97.61) --
	(248.92, 97.07) --
	(249.31, 96.54) --
	(249.70, 96.01) --
	(250.08, 95.50) --
	(250.47, 94.99) --
	(250.86, 94.48) --
	(251.25, 93.99) --
	(251.63, 93.50) --
	(252.02, 93.02) --
	(252.41, 92.55) --
	(252.79, 92.08) --
	(253.18, 91.62) --
	(253.57, 91.17) --
	(253.95, 90.72) --
	(254.34, 90.28) --
	(254.73, 89.85) --
	(255.12, 89.43) --
	(255.50, 89.01) --
	(255.89, 88.60) --
	(256.28, 88.19) --
	(256.66, 87.79) --
	(257.05, 87.40) --
	(257.44, 87.02) --
	(257.83, 86.64) --
	(258.21, 86.26) --
	(258.60, 85.90) --
	(258.99, 85.54) --
	(259.37, 85.18) --
	(259.76, 84.84) --
	(260.15, 84.49) --
	(260.53, 84.16) --
	(260.92, 83.83) --
	(261.31, 83.50) --
	(261.70, 83.19) --
	(262.08, 82.87) --
	(262.47, 82.57) --
	(262.86, 82.26) --
	(263.24, 81.97) --
	(263.63, 81.68) --
	(264.02, 81.39) --
	(264.41, 81.11) --
	(264.79, 80.84) --
	(265.18, 80.57) --
	(265.57, 80.30) --
	(265.95, 80.04) --
	(266.34, 79.79) --
	(266.73, 79.54) --
	(267.11, 79.29) --
	(267.50, 79.05) --
	(267.89, 78.82) --
	(268.28, 78.59) --
	(268.66, 78.36) --
	(269.05, 78.14) --
	(269.44, 77.92) --
	(269.82, 77.71) --
	(270.21, 77.50) --
	(270.60, 77.30) --
	(270.99, 77.10) --
	(271.37, 76.90) --
	(271.76, 76.71) --
	(272.15, 76.52) --
	(272.53, 76.34) --
	(272.92, 76.16) --
	(273.31, 75.98) --
	(273.69, 75.81) --
	(274.08, 75.64) --
	(274.47, 75.48) --
	(274.86, 75.32) --
	(275.24, 75.16) --
	(275.63, 75.01) --
	(276.02, 74.85) --
	(276.40, 74.71) --
	(276.79, 74.56) --
	(277.18, 74.42) --
	(277.57, 74.28) --
	(277.95, 74.15) --
	(278.34, 74.02) --
	(278.73, 73.89) --
	(279.11, 73.76) --
	(279.50, 73.64) --
	(279.89, 73.52) --
	(280.27, 73.40) --
	(280.66, 73.29) --
	(281.05, 73.18) --
	(281.44, 73.07) --
	(281.82, 72.96) --
	(282.21, 72.86) --
	(282.60, 72.75) --
	(282.98, 72.65) --
	(283.37, 72.56) --
	(283.76, 72.46) --
	(284.15, 72.37) --
	(284.53, 72.28) --
	(284.92, 72.19) --
	(285.31, 72.11) --
	(285.69, 72.03) --
	(286.08, 71.94) --
	(286.47, 71.86) --
	(286.85, 71.79) --
	(287.24, 71.71) --
	(287.63, 71.64) --
	(288.02, 71.57) --
	(288.40, 71.50) --
	(288.79, 71.43) --
	(289.18, 71.36) --
	(289.56, 71.30) --
	(289.95, 71.24) --
	(290.34, 71.17) --
	(290.73, 71.12) --
	(291.11, 71.06) --
	(291.50, 71.00) --
	(291.89, 70.95) --
	(292.27, 70.89) --
	(292.66, 70.84) --
	(293.05, 70.79) --
	(293.43, 70.74) --
	(293.82, 70.69) --
	(294.21, 70.65) --
	(294.60, 70.60) --
	(294.98, 70.56) --
	(295.37, 70.51) --
	(295.76, 70.47) --
	(296.14, 70.43) --
	(296.53, 70.39) --
	(296.92, 70.35) --
	(297.31, 70.32) --
	(297.69, 70.28) --
	(298.08, 70.24) --
	(298.47, 70.21) --
	(298.85, 70.18) --
	(299.24, 70.14) --
	(299.63, 70.11) --
	(300.01, 70.08) --
	(300.40, 70.05) --
	(300.79, 70.02) --
	(301.18, 70.00) --
	(301.56, 69.97) --
	(301.95, 69.94) --
	(302.34, 69.92) --
	(302.72, 69.89) --
	(303.11, 69.87) --
	(303.50, 69.85) --
	(303.89, 69.82) --
	(304.27, 69.80) --
	(304.66, 69.78) --
	(305.05, 69.76) --
	(305.43, 69.74) --
	(305.82, 69.72) --
	(306.21, 69.70) --
	(306.60, 69.68) --
	(306.98, 69.66) --
	(307.37, 69.65) --
	(307.76, 69.63) --
	(308.14, 69.62) --
	(308.53, 69.60) --
	(308.92, 69.58) --
	(309.30, 69.57) --
	(309.69, 69.56) --
	(310.08, 69.54) --
	(310.47, 69.53) --
	(310.85, 69.52) --
	(311.24, 69.50) --
	(311.63, 69.49) --
	(312.01, 69.48) --
	(312.40, 69.47) --
	(312.79, 69.46) --
	(313.18, 69.45) --
	(313.56, 69.44) --
	(313.95, 69.43) --
	(314.34, 69.42) --
	(314.72, 69.41) --
	(315.11, 69.40) --
	(315.50, 69.39) --
	(315.88, 69.38) --
	(316.27, 69.37) --
	(316.66, 69.37) --
	(317.05, 69.36) --
	(317.43, 69.35) --
	(317.82, 69.34) --
	(318.21, 69.34) --
	(318.59, 69.33) --
	(318.98, 69.32) --
	(319.37, 69.32) --
	(319.76, 69.31) --
	(320.14, 69.31) --
	(320.53, 69.30) --
	(320.92, 69.30) --
	(321.30, 69.29) --
	(321.69, 69.29) --
	(322.08, 69.28) --
	(322.46, 69.28) --
	(322.85, 69.27) --
	(323.24, 69.27) --
	(323.63, 69.26) --
	(324.01, 69.26) --
	(324.40, 69.25) --
	(324.79, 69.25) --
	(325.17, 69.25) --
	(325.56, 69.24) --
	(325.95, 69.24) --
	(326.34, 69.24) --
	(326.72, 69.23) --
	(327.11, 69.23) --
	(327.50, 69.23) --
	(327.88, 69.23) --
	(328.27, 69.22) --
	(328.66, 69.22) --
	(329.04, 69.22) --
	(329.43, 69.22);
\end{scope}
\end{tikzpicture}

%% file: powers-int-exp_decay.tex
\begin{tikzpicture}[x=1pt,y=1pt]
\definecolor{fillColor}{RGB}{255,255,255}
\path[use as bounding box,fill=fillColor,fill opacity=0.00] (0,0) rectangle (325.21,216.81);
\begin{scope}
\path[clip] ( 49.20, 61.20) rectangle (300.01,167.61);
\definecolor{drawColor}{RGB}{0,0,0}

\path[draw=drawColor,line width= 0.4pt,line join=round,line cap=round] ( 64.35, 71.75) -- ( 75.85, 74.29);

\path[draw=drawColor,line width= 0.4pt,line join=round,line cap=round] ( 86.79, 78.78) -- ( 99.86, 86.98);

\path[draw=drawColor,line width= 0.4pt,line join=round,line cap=round] (109.61, 93.93) -- (123.49,105.12);

\path[draw=drawColor,line width= 0.4pt,line join=round,line cap=round] (132.31,113.22) -- (147.23,128.79);

\path[draw=drawColor,line width= 0.4pt,line join=round,line cap=round] (156.25,136.64) -- (169.74,146.37);

\path[draw=drawColor,line width= 0.4pt,line join=round,line cap=round] (180.03,152.45) -- (192.41,158.34);

\path[draw=drawColor,line width= 0.4pt,line join=round,line cap=round] (203.79,161.57) -- (215.09,162.81);

\path[draw=drawColor,line width= 0.4pt,line join=round,line cap=round] (227.05,163.52) -- (238.28,163.62);

\path[draw=drawColor,line width= 0.4pt,line join=round,line cap=round] (250.28,163.67) -- (261.50,163.67);

\path[draw=drawColor,line width= 0.4pt,line join=round,line cap=round] (273.50,163.67) -- (284.73,163.67);

\path[draw=drawColor,line width= 0.4pt,line join=round,line cap=round] ( 58.49, 70.46) circle (  2.25);

\path[draw=drawColor,line width= 0.4pt,line join=round,line cap=round] ( 81.71, 75.59) circle (  2.25);

\path[draw=drawColor,line width= 0.4pt,line join=round,line cap=round] (104.94, 90.17) circle (  2.25);

\path[draw=drawColor,line width= 0.4pt,line join=round,line cap=round] (128.16,108.89) circle (  2.25);

\path[draw=drawColor,line width= 0.4pt,line join=round,line cap=round] (151.38,133.13) circle (  2.25);

\path[draw=drawColor,line width= 0.4pt,line join=round,line cap=round] (174.61,149.88) circle (  2.25);

\path[draw=drawColor,line width= 0.4pt,line join=round,line cap=round] (197.83,160.91) circle (  2.25);

\path[draw=drawColor,line width= 0.4pt,line join=round,line cap=round] (221.05,163.47) circle (  2.25);

\path[draw=drawColor,line width= 0.4pt,line join=round,line cap=round] (244.28,163.67) circle (  2.25);

\path[draw=drawColor,line width= 0.4pt,line join=round,line cap=round] (267.50,163.67) circle (  2.25);

\path[draw=drawColor,line width= 0.4pt,line join=round,line cap=round] (290.73,163.67) circle (  2.25);
\end{scope}
\begin{scope}
\path[clip] (  0.00,  0.00) rectangle (325.21,216.81);
\definecolor{drawColor}{RGB}{0,0,0}

\path[draw=drawColor,line width= 0.4pt,line join=round,line cap=round] ( 58.49, 61.20) -- (290.73, 61.20);

\path[draw=drawColor,line width= 0.4pt,line join=round,line cap=round] ( 58.49, 61.20) -- ( 58.49, 55.20);

\path[draw=drawColor,line width= 0.4pt,line join=round,line cap=round] (104.94, 61.20) -- (104.94, 55.20);

\path[draw=drawColor,line width= 0.4pt,line join=round,line cap=round] (151.38, 61.20) -- (151.38, 55.20);

\path[draw=drawColor,line width= 0.4pt,line join=round,line cap=round] (197.83, 61.20) -- (197.83, 55.20);

\path[draw=drawColor,line width= 0.4pt,line join=round,line cap=round] (244.28, 61.20) -- (244.28, 55.20);

\path[draw=drawColor,line width= 0.4pt,line join=round,line cap=round] (290.73, 61.20) -- (290.73, 55.20);

\node[text=drawColor,anchor=base,inner sep=0pt, outer sep=0pt, scale=  1.00] at ( 58.49, 39.60) {0.0};

\node[text=drawColor,anchor=base,inner sep=0pt, outer sep=0pt, scale=  1.00] at (104.94, 39.60) {0.2};

\node[text=drawColor,anchor=base,inner sep=0pt, outer sep=0pt, scale=  1.00] at (151.38, 39.60) {0.4};

\node[text=drawColor,anchor=base,inner sep=0pt, outer sep=0pt, scale=  1.00] at (197.83, 39.60) {0.6};

\node[text=drawColor,anchor=base,inner sep=0pt, outer sep=0pt, scale=  1.00] at (244.28, 39.60) {0.8};

\node[text=drawColor,anchor=base,inner sep=0pt, outer sep=0pt, scale=  1.00] at (290.73, 39.60) {1.0};

\path[draw=drawColor,line width= 0.4pt,line join=round,line cap=round] ( 49.20, 65.14) -- ( 49.20,163.67);

\path[draw=drawColor,line width= 0.4pt,line join=round,line cap=round] ( 49.20, 65.14) -- ( 43.20, 65.14);

\path[draw=drawColor,line width= 0.4pt,line join=round,line cap=round] ( 49.20, 84.85) -- ( 43.20, 84.85);

\path[draw=drawColor,line width= 0.4pt,line join=round,line cap=round] ( 49.20,104.55) -- ( 43.20,104.55);

\path[draw=drawColor,line width= 0.4pt,line join=round,line cap=round] ( 49.20,124.26) -- ( 43.20,124.26);

\path[draw=drawColor,line width= 0.4pt,line join=round,line cap=round] ( 49.20,143.96) -- ( 43.20,143.96);

\path[draw=drawColor,line width= 0.4pt,line join=round,line cap=round] ( 49.20,163.67) -- ( 43.20,163.67);

\node[text=drawColor,rotate= 90.00,anchor=base,inner sep=0pt, outer sep=0pt, scale=  1.00] at ( 34.80, 65.14) {0.0};

\node[text=drawColor,rotate= 90.00,anchor=base,inner sep=0pt, outer sep=0pt, scale=  1.00] at ( 34.80,104.55) {0.4};

\node[text=drawColor,rotate= 90.00,anchor=base,inner sep=0pt, outer sep=0pt, scale=  1.00] at ( 34.80,143.96) {0.8};

\path[draw=drawColor,line width= 0.4pt,line join=round,line cap=round] ( 49.20, 61.20) --
	(300.01, 61.20) --
	(300.01,167.61) --
	( 49.20,167.61) --
	cycle;
\end{scope}
\begin{scope}
\path[clip] (  0.00,  0.00) rectangle (325.21,216.81);
\definecolor{drawColor}{RGB}{0,0,0}

\node[text=drawColor,anchor=base,inner sep=0pt, outer sep=0pt, scale=  1.00] at (174.61, 15.60) {$r^2$};

\node[text=drawColor,rotate= 90.00,anchor=base,inner sep=0pt, outer sep=0pt, scale=  1.00] at ( 10.80,114.41) {Power};

\node[text=drawColor,anchor=base,inner sep=0pt, outer sep=0pt, scale=  1.20] at (174.61,188.07) {\bfseries \texttt{int}};
\end{scope}
\begin{scope}
\path[clip] ( 49.20, 61.20) rectangle (300.01,167.61);
\definecolor{drawColor}{RGB}{255,0,0}

\path[draw=drawColor,line width= 0.4pt,dash pattern=on 4pt off 4pt ,line join=round,line cap=round] ( 60.17, 77.01) -- ( 80.04,145.30);

\path[draw=drawColor,line width= 0.4pt,dash pattern=on 4pt off 4pt ,line join=round,line cap=round] ( 87.08,153.74) -- ( 99.57,160.00);

\path[draw=drawColor,line width= 0.4pt,dash pattern=on 4pt off 4pt ,line join=round,line cap=round] (110.93,162.89) -- (122.16,163.27);

\path[draw=drawColor,line width= 0.4pt,dash pattern=on 4pt off 4pt ,line join=round,line cap=round] (134.16,163.52) -- (145.38,163.62);

\path[draw=drawColor,line width= 0.4pt,dash pattern=on 4pt off 4pt ,line join=round,line cap=round] (157.38,163.67) -- (168.61,163.67);

\path[draw=drawColor,line width= 0.4pt,dash pattern=on 4pt off 4pt ,line join=round,line cap=round] (180.61,163.67) -- (191.83,163.67);

\path[draw=drawColor,line width= 0.4pt,dash pattern=on 4pt off 4pt ,line join=round,line cap=round] (203.83,163.67) -- (215.05,163.67);

\path[draw=drawColor,line width= 0.4pt,dash pattern=on 4pt off 4pt ,line join=round,line cap=round] (227.05,163.67) -- (238.28,163.67);

\path[draw=drawColor,line width= 0.4pt,dash pattern=on 4pt off 4pt ,line join=round,line cap=round] (250.28,163.67) -- (261.50,163.67);

\path[draw=drawColor,line width= 0.4pt,dash pattern=on 4pt off 4pt ,line join=round,line cap=round] (273.50,163.67) -- (284.73,163.67);

\path[draw=drawColor,line width= 0.4pt,line join=round,line cap=round] ( 58.49, 74.75) --
	( 61.52, 69.50) --
	( 55.46, 69.50) --
	cycle;

\path[draw=drawColor,line width= 0.4pt,line join=round,line cap=round] ( 81.71,154.56) --
	( 84.74,149.31) --
	( 78.68,149.31) --
	cycle;

\path[draw=drawColor,line width= 0.4pt,line join=round,line cap=round] (104.94,166.18) --
	(107.97,160.93) --
	(101.91,160.93) --
	cycle;

\path[draw=drawColor,line width= 0.4pt,line join=round,line cap=round] (128.16,166.97) --
	(131.19,161.72) --
	(125.13,161.72) --
	cycle;

\path[draw=drawColor,line width= 0.4pt,line join=round,line cap=round] (151.38,167.17) --
	(154.41,161.92) --
	(148.35,161.92) --
	cycle;

\path[draw=drawColor,line width= 0.4pt,line join=round,line cap=round] (174.61,167.17) --
	(177.64,161.92) --
	(171.58,161.92) --
	cycle;

\path[draw=drawColor,line width= 0.4pt,line join=round,line cap=round] (197.83,167.17) --
	(200.86,161.92) --
	(194.80,161.92) --
	cycle;

\path[draw=drawColor,line width= 0.4pt,line join=round,line cap=round] (221.05,167.17) --
	(224.08,161.92) --
	(218.02,161.92) --
	cycle;

\path[draw=drawColor,line width= 0.4pt,line join=round,line cap=round] (244.28,167.17) --
	(247.31,161.92) --
	(241.25,161.92) --
	cycle;

\path[draw=drawColor,line width= 0.4pt,line join=round,line cap=round] (267.50,167.17) --
	(270.53,161.92) --
	(264.47,161.92) --
	cycle;

\path[draw=drawColor,line width= 0.4pt,line join=round,line cap=round] (290.73,167.17) --
	(293.76,161.92) --
	(287.70,161.92) --
	cycle;
\definecolor{drawColor}{RGB}{0,0,255}

\path[draw=drawColor,line width= 0.4pt,dash pattern=on 1pt off 3pt on 4pt off 3pt ,line join=round,line cap=round] ( 63.94, 72.76) -- ( 76.26, 78.41);

\path[draw=drawColor,line width= 0.4pt,dash pattern=on 1pt off 3pt on 4pt off 3pt ,line join=round,line cap=round] ( 85.73, 85.37) -- (100.92,102.26);

\path[draw=drawColor,line width= 0.4pt,dash pattern=on 1pt off 3pt on 4pt off 3pt ,line join=round,line cap=round] (108.40,111.62) -- (124.70,134.73);

\path[draw=drawColor,line width= 0.4pt,dash pattern=on 1pt off 3pt on 4pt off 3pt ,line join=round,line cap=round] (132.75,143.49) -- (146.79,155.28);

\path[draw=drawColor,line width= 0.4pt,dash pattern=on 1pt off 3pt on 4pt off 3pt ,line join=round,line cap=round] (157.28,160.24) -- (168.71,162.37);

\path[draw=drawColor,line width= 0.4pt,dash pattern=on 1pt off 3pt on 4pt off 3pt ,line join=round,line cap=round] (180.61,163.52) -- (191.83,163.62);

\path[draw=drawColor,line width= 0.4pt,dash pattern=on 1pt off 3pt on 4pt off 3pt ,line join=round,line cap=round] (203.83,163.67) -- (215.05,163.67);

\path[draw=drawColor,line width= 0.4pt,dash pattern=on 1pt off 3pt on 4pt off 3pt ,line join=round,line cap=round] (227.05,163.67) -- (238.28,163.67);

\path[draw=drawColor,line width= 0.4pt,dash pattern=on 1pt off 3pt on 4pt off 3pt ,line join=round,line cap=round] (250.28,163.67) -- (261.50,163.67);

\path[draw=drawColor,line width= 0.4pt,dash pattern=on 1pt off 3pt on 4pt off 3pt ,line join=round,line cap=round] (273.50,163.67) -- (284.73,163.67);

\path[draw=drawColor,line width= 0.4pt,line join=round,line cap=round] ( 55.31, 70.26) -- ( 61.67, 70.26);

\path[draw=drawColor,line width= 0.4pt,line join=round,line cap=round] ( 58.49, 67.08) -- ( 58.49, 73.45);

\path[draw=drawColor,line width= 0.4pt,line join=round,line cap=round] ( 78.53, 80.91) -- ( 84.90, 80.91);

\path[draw=drawColor,line width= 0.4pt,line join=round,line cap=round] ( 81.71, 77.72) -- ( 81.71, 84.09);

\path[draw=drawColor,line width= 0.4pt,line join=round,line cap=round] (101.75,106.72) -- (108.12,106.72);

\path[draw=drawColor,line width= 0.4pt,line join=round,line cap=round] (104.94,103.54) -- (104.94,109.90);

\path[draw=drawColor,line width= 0.4pt,line join=round,line cap=round] (124.98,139.63) -- (131.34,139.63);

\path[draw=drawColor,line width= 0.4pt,line join=round,line cap=round] (128.16,136.45) -- (128.16,142.81);

\path[draw=drawColor,line width= 0.4pt,line join=round,line cap=round] (148.20,159.14) -- (154.57,159.14);

\path[draw=drawColor,line width= 0.4pt,line join=round,line cap=round] (151.38,155.95) -- (151.38,162.32);

\path[draw=drawColor,line width= 0.4pt,line join=round,line cap=round] (171.43,163.47) -- (177.79,163.47);

\path[draw=drawColor,line width= 0.4pt,line join=round,line cap=round] (174.61,160.29) -- (174.61,166.65);

\path[draw=drawColor,line width= 0.4pt,line join=round,line cap=round] (194.65,163.67) -- (201.01,163.67);

\path[draw=drawColor,line width= 0.4pt,line join=round,line cap=round] (197.83,160.49) -- (197.83,166.85);

\path[draw=drawColor,line width= 0.4pt,line join=round,line cap=round] (217.87,163.67) -- (224.24,163.67);

\path[draw=drawColor,line width= 0.4pt,line join=round,line cap=round] (221.05,160.49) -- (221.05,166.85);

\path[draw=drawColor,line width= 0.4pt,line join=round,line cap=round] (241.10,163.67) -- (247.46,163.67);

\path[draw=drawColor,line width= 0.4pt,line join=round,line cap=round] (244.28,160.49) -- (244.28,166.85);

\path[draw=drawColor,line width= 0.4pt,line join=round,line cap=round] (264.32,163.67) -- (270.68,163.67);

\path[draw=drawColor,line width= 0.4pt,line join=round,line cap=round] (267.50,160.49) -- (267.50,166.85);

\path[draw=drawColor,line width= 0.4pt,line join=round,line cap=round] (287.54,163.67) -- (293.91,163.67);

\path[draw=drawColor,line width= 0.4pt,line join=round,line cap=round] (290.73,160.49) -- (290.73,166.85);
\definecolor{drawColor}{RGB}{0,255,0}

\path[draw=drawColor,line width= 0.4pt,dash pattern=on 1pt off 3pt on 4pt off 3pt ,line join=round,line cap=round] ( 60.98, 73.75) -- ( 79.22,113.68);

\path[draw=drawColor,line width= 0.4pt,dash pattern=on 1pt off 3pt on 4pt off 3pt ,line join=round,line cap=round] ( 84.80,124.28) -- (101.84,152.62);

\path[draw=drawColor,line width= 0.4pt,dash pattern=on 1pt off 3pt on 4pt off 3pt ,line join=round,line cap=round] (110.82,158.95) -- (122.28,161.29);

\path[draw=drawColor,line width= 0.4pt,dash pattern=on 1pt off 3pt on 4pt off 3pt ,line join=round,line cap=round] (134.15,162.74) -- (145.39,163.22);

\path[draw=drawColor,line width= 0.4pt,dash pattern=on 1pt off 3pt on 4pt off 3pt ,line join=round,line cap=round] (157.38,163.52) -- (168.61,163.62);

\path[draw=drawColor,line width= 0.4pt,dash pattern=on 1pt off 3pt on 4pt off 3pt ,line join=round,line cap=round] (180.61,163.67) -- (191.83,163.67);

\path[draw=drawColor,line width= 0.4pt,dash pattern=on 1pt off 3pt on 4pt off 3pt ,line join=round,line cap=round] (203.83,163.67) -- (215.05,163.67);

\path[draw=drawColor,line width= 0.4pt,dash pattern=on 1pt off 3pt on 4pt off 3pt ,line join=round,line cap=round] (227.05,163.67) -- (238.28,163.67);

\path[draw=drawColor,line width= 0.4pt,dash pattern=on 1pt off 3pt on 4pt off 3pt ,line join=round,line cap=round] (250.28,163.67) -- (261.50,163.67);

\path[draw=drawColor,line width= 0.4pt,dash pattern=on 1pt off 3pt on 4pt off 3pt ,line join=round,line cap=round] (273.50,163.67) -- (284.73,163.67);

\path[draw=drawColor,line width= 0.4pt,line join=round,line cap=round] ( 56.24, 66.04) -- ( 60.74, 70.54);

\path[draw=drawColor,line width= 0.4pt,line join=round,line cap=round] ( 56.24, 70.54) -- ( 60.74, 66.04);

\path[draw=drawColor,line width= 0.4pt,line join=round,line cap=round] ( 79.46,116.88) -- ( 83.96,121.38);

\path[draw=drawColor,line width= 0.4pt,line join=round,line cap=round] ( 79.46,121.38) -- ( 83.96,116.88);

\path[draw=drawColor,line width= 0.4pt,line join=round,line cap=round] (102.69,155.51) -- (107.19,160.01);

\path[draw=drawColor,line width= 0.4pt,line join=round,line cap=round] (102.69,160.01) -- (107.19,155.51);

\path[draw=drawColor,line width= 0.4pt,line join=round,line cap=round] (125.91,160.24) -- (130.41,164.74);

\path[draw=drawColor,line width= 0.4pt,line join=round,line cap=round] (125.91,164.74) -- (130.41,160.24);

\path[draw=drawColor,line width= 0.4pt,line join=round,line cap=round] (149.13,161.22) -- (153.63,165.72);

\path[draw=drawColor,line width= 0.4pt,line join=round,line cap=round] (149.13,165.72) -- (153.63,161.22);

\path[draw=drawColor,line width= 0.4pt,line join=round,line cap=round] (172.36,161.42) -- (176.86,165.92);

\path[draw=drawColor,line width= 0.4pt,line join=round,line cap=round] (172.36,165.92) -- (176.86,161.42);

\path[draw=drawColor,line width= 0.4pt,line join=round,line cap=round] (195.58,161.42) -- (200.08,165.92);

\path[draw=drawColor,line width= 0.4pt,line join=round,line cap=round] (195.58,165.92) -- (200.08,161.42);

\path[draw=drawColor,line width= 0.4pt,line join=round,line cap=round] (218.80,161.42) -- (223.30,165.92);

\path[draw=drawColor,line width= 0.4pt,line join=round,line cap=round] (218.80,165.92) -- (223.30,161.42);

\path[draw=drawColor,line width= 0.4pt,line join=round,line cap=round] (242.03,161.42) -- (246.53,165.92);

\path[draw=drawColor,line width= 0.4pt,line join=round,line cap=round] (242.03,165.92) -- (246.53,161.42);

\path[draw=drawColor,line width= 0.4pt,line join=round,line cap=round] (265.25,161.42) -- (269.75,165.92);

\path[draw=drawColor,line width= 0.4pt,line join=round,line cap=round] (265.25,165.92) -- (269.75,161.42);

\path[draw=drawColor,line width= 0.4pt,line join=round,line cap=round] (288.48,161.42) -- (292.98,165.92);

\path[draw=drawColor,line width= 0.4pt,line join=round,line cap=round] (288.48,165.92) -- (292.98,161.42);
\definecolor{drawColor}{RGB}{0,0,0}

\path[draw=drawColor,line width= 0.4pt,line join=round,line cap=round] (253.73,121.20) rectangle (300.01, 61.20);

\path[draw=drawColor,line width= 0.4pt,line join=round,line cap=round] (262.73,109.20) circle (  2.25);
\definecolor{drawColor}{RGB}{255,0,0}

\path[draw=drawColor,line width= 0.4pt,line join=round,line cap=round] (262.73,100.70) --
	(265.76, 95.45) --
	(259.70, 95.45) --
	cycle;
\definecolor{drawColor}{RGB}{0,255,0}

\path[draw=drawColor,line width= 0.4pt,line join=round,line cap=round] (259.55, 85.20) -- (265.91, 85.20);

\path[draw=drawColor,line width= 0.4pt,line join=round,line cap=round] (262.73, 82.02) -- (262.73, 88.38);
\definecolor{drawColor}{RGB}{0,0,255}

\path[draw=drawColor,line width= 0.4pt,line join=round,line cap=round] (260.48, 70.95) -- (264.98, 75.45);

\path[draw=drawColor,line width= 0.4pt,line join=round,line cap=round] (260.48, 75.45) -- (264.98, 70.95);
\definecolor{drawColor}{RGB}{0,0,0}

\node[text=drawColor,anchor=base west,inner sep=0pt, outer sep=0pt, scale=  1.00] at (271.73,105.76) {\itshape $\mathcal I_n$};

\node[text=drawColor,anchor=base west,inner sep=0pt, outer sep=0pt, scale=  1.00] at (271.73, 93.76) {\itshape \normalfont $\mathcal I_n^\text{DC}$};

\node[text=drawColor,anchor=base west,inner sep=0pt, outer sep=0pt, scale=  1.00] at (271.73, 81.76) {\itshape \normalfont $\mathcal I_n^\text{BC}$};

\node[text=drawColor,anchor=base west,inner sep=0pt, outer sep=0pt, scale=  1.00] at (271.73, 69.76) {\itshape \normalfont $\mathcal I_n^\text{CvM}$};
\end{scope}
\end{tikzpicture}

%% file: powers-sqnorm-exp_decay.tex
\begin{tikzpicture}[x=1pt,y=1pt]
\definecolor{fillColor}{RGB}{255,255,255}
\path[use as bounding box,fill=fillColor,fill opacity=0.00] (0,0) rectangle (325.21,216.81);
\begin{scope}
\path[clip] ( 49.20, 61.20) rectangle (300.01,167.61);
\definecolor{drawColor}{RGB}{0,0,0}

\path[draw=drawColor,line width= 0.4pt,line join=round,line cap=round] ( 64.49, 70.51) -- ( 75.71, 70.61);

\path[draw=drawColor,line width= 0.4pt,line join=round,line cap=round] ( 87.62, 71.71) -- ( 99.03, 73.74);

\path[draw=drawColor,line width= 0.4pt,line join=round,line cap=round] (110.62, 76.73) -- (122.48, 80.75);

\path[draw=drawColor,line width= 0.4pt,line join=round,line cap=round] (133.70, 84.98) -- (145.84, 90.03);

\path[draw=drawColor,line width= 0.4pt,line join=round,line cap=round] (156.84, 94.83) -- (169.15,100.48);

\path[draw=drawColor,line width= 0.4pt,line join=round,line cap=round] (179.69,106.17) -- (192.75,114.37);

\path[draw=drawColor,line width= 0.4pt,line join=round,line cap=round] (202.68,121.09) -- (216.21,130.97);

\path[draw=drawColor,line width= 0.4pt,line join=round,line cap=round] (225.98,137.93) -- (239.35,147.24);

\path[draw=drawColor,line width= 0.4pt,line join=round,line cap=round] (249.63,153.39) -- (262.16,159.76);

\path[draw=drawColor,line width= 0.4pt,line join=round,line cap=round] (273.49,162.79) -- (284.73,163.36);

\path[draw=drawColor,line width= 0.4pt,line join=round,line cap=round] ( 58.49, 70.46) circle (  2.25);

\path[draw=drawColor,line width= 0.4pt,line join=round,line cap=round] ( 81.71, 70.66) circle (  2.25);

\path[draw=drawColor,line width= 0.4pt,line join=round,line cap=round] (104.94, 74.80) circle (  2.25);

\path[draw=drawColor,line width= 0.4pt,line join=round,line cap=round] (128.16, 82.68) circle (  2.25);

\path[draw=drawColor,line width= 0.4pt,line join=round,line cap=round] (151.38, 92.33) circle (  2.25);

\path[draw=drawColor,line width= 0.4pt,line join=round,line cap=round] (174.61,102.98) circle (  2.25);

\path[draw=drawColor,line width= 0.4pt,line join=round,line cap=round] (197.83,117.56) circle (  2.25);

\path[draw=drawColor,line width= 0.4pt,line join=round,line cap=round] (221.05,134.50) circle (  2.25);

\path[draw=drawColor,line width= 0.4pt,line join=round,line cap=round] (244.28,150.66) circle (  2.25);

\path[draw=drawColor,line width= 0.4pt,line join=round,line cap=round] (267.50,162.49) circle (  2.25);

\path[draw=drawColor,line width= 0.4pt,line join=round,line cap=round] (290.73,163.67) circle (  2.25);
\end{scope}
\begin{scope}
\path[clip] (  0.00,  0.00) rectangle (325.21,216.81);
\definecolor{drawColor}{RGB}{0,0,0}

\path[draw=drawColor,line width= 0.4pt,line join=round,line cap=round] ( 58.49, 61.20) -- (290.73, 61.20);

\path[draw=drawColor,line width= 0.4pt,line join=round,line cap=round] ( 58.49, 61.20) -- ( 58.49, 55.20);

\path[draw=drawColor,line width= 0.4pt,line join=round,line cap=round] (104.94, 61.20) -- (104.94, 55.20);

\path[draw=drawColor,line width= 0.4pt,line join=round,line cap=round] (151.38, 61.20) -- (151.38, 55.20);

\path[draw=drawColor,line width= 0.4pt,line join=round,line cap=round] (197.83, 61.20) -- (197.83, 55.20);

\path[draw=drawColor,line width= 0.4pt,line join=round,line cap=round] (244.28, 61.20) -- (244.28, 55.20);

\path[draw=drawColor,line width= 0.4pt,line join=round,line cap=round] (290.73, 61.20) -- (290.73, 55.20);

\node[text=drawColor,anchor=base,inner sep=0pt, outer sep=0pt, scale=  1.00] at ( 58.49, 39.60) {0.0};

\node[text=drawColor,anchor=base,inner sep=0pt, outer sep=0pt, scale=  1.00] at (104.94, 39.60) {0.2};

\node[text=drawColor,anchor=base,inner sep=0pt, outer sep=0pt, scale=  1.00] at (151.38, 39.60) {0.4};

\node[text=drawColor,anchor=base,inner sep=0pt, outer sep=0pt, scale=  1.00] at (197.83, 39.60) {0.6};

\node[text=drawColor,anchor=base,inner sep=0pt, outer sep=0pt, scale=  1.00] at (244.28, 39.60) {0.8};

\node[text=drawColor,anchor=base,inner sep=0pt, outer sep=0pt, scale=  1.00] at (290.73, 39.60) {1.0};

\path[draw=drawColor,line width= 0.4pt,line join=round,line cap=round] ( 49.20, 65.14) -- ( 49.20,163.67);

\path[draw=drawColor,line width= 0.4pt,line join=round,line cap=round] ( 49.20, 65.14) -- ( 43.20, 65.14);

\path[draw=drawColor,line width= 0.4pt,line join=round,line cap=round] ( 49.20, 84.85) -- ( 43.20, 84.85);

\path[draw=drawColor,line width= 0.4pt,line join=round,line cap=round] ( 49.20,104.55) -- ( 43.20,104.55);

\path[draw=drawColor,line width= 0.4pt,line join=round,line cap=round] ( 49.20,124.26) -- ( 43.20,124.26);

\path[draw=drawColor,line width= 0.4pt,line join=round,line cap=round] ( 49.20,143.96) -- ( 43.20,143.96);

\path[draw=drawColor,line width= 0.4pt,line join=round,line cap=round] ( 49.20,163.67) -- ( 43.20,163.67);

\node[text=drawColor,rotate= 90.00,anchor=base,inner sep=0pt, outer sep=0pt, scale=  1.00] at ( 34.80, 65.14) {0.0};

\node[text=drawColor,rotate= 90.00,anchor=base,inner sep=0pt, outer sep=0pt, scale=  1.00] at ( 34.80,104.55) {0.4};

\node[text=drawColor,rotate= 90.00,anchor=base,inner sep=0pt, outer sep=0pt, scale=  1.00] at ( 34.80,143.96) {0.8};

\path[draw=drawColor,line width= 0.4pt,line join=round,line cap=round] ( 49.20, 61.20) --
	(300.01, 61.20) --
	(300.01,167.61) --
	( 49.20,167.61) --
	cycle;
\end{scope}
\begin{scope}
\path[clip] (  0.00,  0.00) rectangle (325.21,216.81);
\definecolor{drawColor}{RGB}{0,0,0}

\node[text=drawColor,anchor=base,inner sep=0pt, outer sep=0pt, scale=  1.00] at (174.61, 15.60) {$r^2$};

\node[text=drawColor,rotate= 90.00,anchor=base,inner sep=0pt, outer sep=0pt, scale=  1.00] at ( 10.80,114.41) {Power};

\node[text=drawColor,anchor=base,inner sep=0pt, outer sep=0pt, scale=  1.20] at (174.61,188.07) {\bfseries \texttt{sqnorm}};
\end{scope}
\begin{scope}
\path[clip] ( 49.20, 61.20) rectangle (300.01,167.61);
\definecolor{drawColor}{RGB}{255,0,0}

\path[draw=drawColor,line width= 0.4pt,dash pattern=on 4pt off 4pt ,line join=round,line cap=round] ( 63.39, 74.70) -- ( 76.81, 84.15);

\path[draw=drawColor,line width= 0.4pt,dash pattern=on 4pt off 4pt ,line join=round,line cap=round] ( 85.00, 92.63) -- (101.65,118.06);

\path[draw=drawColor,line width= 0.4pt,dash pattern=on 4pt off 4pt ,line join=round,line cap=round] (109.27,127.23) -- (123.83,141.19);

\path[draw=drawColor,line width= 0.4pt,dash pattern=on 4pt off 4pt ,line join=round,line cap=round] (133.51,148.06) -- (146.04,154.44);

\path[draw=drawColor,line width= 0.4pt,dash pattern=on 4pt off 4pt ,line join=round,line cap=round] (157.20,158.65) -- (168.79,161.60);

\path[draw=drawColor,line width= 0.4pt,dash pattern=on 4pt off 4pt ,line join=round,line cap=round] (180.61,163.23) -- (191.83,163.52);

\path[draw=drawColor,line width= 0.4pt,dash pattern=on 4pt off 4pt ,line join=round,line cap=round] (203.83,163.67) -- (215.05,163.67);

\path[draw=drawColor,line width= 0.4pt,dash pattern=on 4pt off 4pt ,line join=round,line cap=round] (227.05,163.67) -- (238.28,163.67);

\path[draw=drawColor,line width= 0.4pt,dash pattern=on 4pt off 4pt ,line join=round,line cap=round] (250.28,163.67) -- (261.50,163.67);

\path[draw=drawColor,line width= 0.4pt,dash pattern=on 4pt off 4pt ,line join=round,line cap=round] (273.50,163.67) -- (284.73,163.67);

\path[draw=drawColor,line width= 0.4pt,line join=round,line cap=round] ( 58.49, 74.75) --
	( 61.52, 69.50) --
	( 55.46, 69.50) --
	cycle;

\path[draw=drawColor,line width= 0.4pt,line join=round,line cap=round] ( 81.71, 91.10) --
	( 84.74, 85.86) --
	( 78.68, 85.86) --
	cycle;

\path[draw=drawColor,line width= 0.4pt,line join=round,line cap=round] (104.94,126.57) --
	(107.97,121.33) --
	(101.91,121.33) --
	cycle;

\path[draw=drawColor,line width= 0.4pt,line join=round,line cap=round] (128.16,148.84) --
	(131.19,143.59) --
	(125.13,143.59) --
	cycle;

\path[draw=drawColor,line width= 0.4pt,line join=round,line cap=round] (151.38,160.67) --
	(154.41,155.42) --
	(148.35,155.42) --
	cycle;

\path[draw=drawColor,line width= 0.4pt,line join=round,line cap=round] (174.61,166.58) --
	(177.64,161.33) --
	(171.58,161.33) --
	cycle;

\path[draw=drawColor,line width= 0.4pt,line join=round,line cap=round] (197.83,167.17) --
	(200.86,161.92) --
	(194.80,161.92) --
	cycle;

\path[draw=drawColor,line width= 0.4pt,line join=round,line cap=round] (221.05,167.17) --
	(224.08,161.92) --
	(218.02,161.92) --
	cycle;

\path[draw=drawColor,line width= 0.4pt,line join=round,line cap=round] (244.28,167.17) --
	(247.31,161.92) --
	(241.25,161.92) --
	cycle;

\path[draw=drawColor,line width= 0.4pt,line join=round,line cap=round] (267.50,167.17) --
	(270.53,161.92) --
	(264.47,161.92) --
	cycle;

\path[draw=drawColor,line width= 0.4pt,line join=round,line cap=round] (290.73,167.17) --
	(293.76,161.92) --
	(287.70,161.92) --
	cycle;
\definecolor{drawColor}{RGB}{0,0,255}

\path[draw=drawColor,line width= 0.4pt,dash pattern=on 1pt off 3pt on 4pt off 3pt ,line join=round,line cap=round] ( 64.49, 70.11) -- ( 75.71, 69.83);

\path[draw=drawColor,line width= 0.4pt,dash pattern=on 1pt off 3pt on 4pt off 3pt ,line join=round,line cap=round] ( 87.70, 70.13) -- ( 98.95, 70.99);

\path[draw=drawColor,line width= 0.4pt,dash pattern=on 1pt off 3pt on 4pt off 3pt ,line join=round,line cap=round] (110.83, 72.55) -- (122.26, 74.68);

\path[draw=drawColor,line width= 0.4pt,dash pattern=on 1pt off 3pt on 4pt off 3pt ,line join=round,line cap=round] (134.15, 76.04) -- (145.39, 76.51);

\path[draw=drawColor,line width= 0.4pt,dash pattern=on 1pt off 3pt on 4pt off 3pt ,line join=round,line cap=round] (157.38, 77.07) -- (168.62, 77.64);

\path[draw=drawColor,line width= 0.4pt,dash pattern=on 1pt off 3pt on 4pt off 3pt ,line join=round,line cap=round] (180.58, 77.34) -- (191.86, 76.19);

\path[draw=drawColor,line width= 0.4pt,dash pattern=on 1pt off 3pt on 4pt off 3pt ,line join=round,line cap=round] (203.71, 76.78) -- (215.18, 79.12);

\path[draw=drawColor,line width= 0.4pt,dash pattern=on 1pt off 3pt on 4pt off 3pt ,line join=round,line cap=round] (227.02, 80.92) -- (238.31, 82.07);

\path[draw=drawColor,line width= 0.4pt,dash pattern=on 1pt off 3pt on 4pt off 3pt ,line join=round,line cap=round] (250.25, 83.24) -- (261.53, 84.29);

\path[draw=drawColor,line width= 0.4pt,dash pattern=on 1pt off 3pt on 4pt off 3pt ,line join=round,line cap=round] (273.46, 85.55) -- (284.77, 86.90);

\path[draw=drawColor,line width= 0.4pt,line join=round,line cap=round] ( 55.31, 70.26) -- ( 61.67, 70.26);

\path[draw=drawColor,line width= 0.4pt,line join=round,line cap=round] ( 58.49, 67.08) -- ( 58.49, 73.45);

\path[draw=drawColor,line width= 0.4pt,line join=round,line cap=round] ( 78.53, 69.67) -- ( 84.90, 69.67);

\path[draw=drawColor,line width= 0.4pt,line join=round,line cap=round] ( 81.71, 66.49) -- ( 81.71, 72.86);

\path[draw=drawColor,line width= 0.4pt,line join=round,line cap=round] (101.75, 71.45) -- (108.12, 71.45);

\path[draw=drawColor,line width= 0.4pt,line join=round,line cap=round] (104.94, 68.26) -- (104.94, 74.63);

\path[draw=drawColor,line width= 0.4pt,line join=round,line cap=round] (124.98, 75.78) -- (131.34, 75.78);

\path[draw=drawColor,line width= 0.4pt,line join=round,line cap=round] (128.16, 72.60) -- (128.16, 78.96);

\path[draw=drawColor,line width= 0.4pt,line join=round,line cap=round] (148.20, 76.77) -- (154.57, 76.77);

\path[draw=drawColor,line width= 0.4pt,line join=round,line cap=round] (151.38, 73.59) -- (151.38, 79.95);

\path[draw=drawColor,line width= 0.4pt,line join=round,line cap=round] (171.43, 77.95) -- (177.79, 77.95);

\path[draw=drawColor,line width= 0.4pt,line join=round,line cap=round] (174.61, 74.77) -- (174.61, 81.13);

\path[draw=drawColor,line width= 0.4pt,line join=round,line cap=round] (194.65, 75.59) -- (201.01, 75.59);

\path[draw=drawColor,line width= 0.4pt,line join=round,line cap=round] (197.83, 72.40) -- (197.83, 78.77);

\path[draw=drawColor,line width= 0.4pt,line join=round,line cap=round] (217.87, 80.31) -- (224.24, 80.31);

\path[draw=drawColor,line width= 0.4pt,line join=round,line cap=round] (221.05, 77.13) -- (221.05, 83.50);

\path[draw=drawColor,line width= 0.4pt,line join=round,line cap=round] (241.10, 82.68) -- (247.46, 82.68);

\path[draw=drawColor,line width= 0.4pt,line join=round,line cap=round] (244.28, 79.50) -- (244.28, 85.86);

\path[draw=drawColor,line width= 0.4pt,line join=round,line cap=round] (264.32, 84.85) -- (270.68, 84.85);

\path[draw=drawColor,line width= 0.4pt,line join=round,line cap=round] (267.50, 81.66) -- (267.50, 88.03);

\path[draw=drawColor,line width= 0.4pt,line join=round,line cap=round] (287.54, 87.61) -- (293.91, 87.61);

\path[draw=drawColor,line width= 0.4pt,line join=round,line cap=round] (290.73, 84.42) -- (290.73, 90.79);
\definecolor{drawColor}{RGB}{0,255,0}

\path[draw=drawColor,line width= 0.4pt,dash pattern=on 1pt off 3pt on 4pt off 3pt ,line join=round,line cap=round] ( 61.63, 73.41) -- ( 78.57,101.02);

\path[draw=drawColor,line width= 0.4pt,dash pattern=on 1pt off 3pt on 4pt off 3pt ,line join=round,line cap=round] ( 85.19,111.02) -- (101.46,133.95);

\path[draw=drawColor,line width= 0.4pt,dash pattern=on 1pt off 3pt on 4pt off 3pt ,line join=round,line cap=round] (109.96,142.12) -- (123.14,150.73);

\path[draw=drawColor,line width= 0.4pt,dash pattern=on 1pt off 3pt on 4pt off 3pt ,line join=round,line cap=round] (133.88,155.81) -- (145.66,159.51);

\path[draw=drawColor,line width= 0.4pt,dash pattern=on 1pt off 3pt on 4pt off 3pt ,line join=round,line cap=round] (157.36,161.86) -- (168.63,162.91);

\path[draw=drawColor,line width= 0.4pt,dash pattern=on 1pt off 3pt on 4pt off 3pt ,line join=round,line cap=round] (180.61,163.47) -- (191.83,163.47);

\path[draw=drawColor,line width= 0.4pt,dash pattern=on 1pt off 3pt on 4pt off 3pt ,line join=round,line cap=round] (203.83,163.52) -- (215.05,163.62);

\path[draw=drawColor,line width= 0.4pt,dash pattern=on 1pt off 3pt on 4pt off 3pt ,line join=round,line cap=round] (227.05,163.67) -- (238.28,163.67);

\path[draw=drawColor,line width= 0.4pt,dash pattern=on 1pt off 3pt on 4pt off 3pt ,line join=round,line cap=round] (250.28,163.67) -- (261.50,163.67);

\path[draw=drawColor,line width= 0.4pt,dash pattern=on 1pt off 3pt on 4pt off 3pt ,line join=round,line cap=round] (273.50,163.67) -- (284.73,163.67);

\path[draw=drawColor,line width= 0.4pt,line join=round,line cap=round] ( 56.24, 66.04) -- ( 60.74, 70.54);

\path[draw=drawColor,line width= 0.4pt,line join=round,line cap=round] ( 56.24, 70.54) -- ( 60.74, 66.04);

\path[draw=drawColor,line width= 0.4pt,line join=round,line cap=round] ( 79.46,103.88) -- ( 83.96,108.38);

\path[draw=drawColor,line width= 0.4pt,line join=round,line cap=round] ( 79.46,108.38) -- ( 83.96,103.88);

\path[draw=drawColor,line width= 0.4pt,line join=round,line cap=round] (102.69,136.59) -- (107.19,141.09);

\path[draw=drawColor,line width= 0.4pt,line join=round,line cap=round] (102.69,141.09) -- (107.19,136.59);

\path[draw=drawColor,line width= 0.4pt,line join=round,line cap=round] (125.91,151.76) -- (130.41,156.26);

\path[draw=drawColor,line width= 0.4pt,line join=round,line cap=round] (125.91,156.26) -- (130.41,151.76);

\path[draw=drawColor,line width= 0.4pt,line join=round,line cap=round] (149.13,159.05) -- (153.63,163.55);

\path[draw=drawColor,line width= 0.4pt,line join=round,line cap=round] (149.13,163.55) -- (153.63,159.05);

\path[draw=drawColor,line width= 0.4pt,line join=round,line cap=round] (172.36,161.22) -- (176.86,165.72);

\path[draw=drawColor,line width= 0.4pt,line join=round,line cap=round] (172.36,165.72) -- (176.86,161.22);

\path[draw=drawColor,line width= 0.4pt,line join=round,line cap=round] (195.58,161.22) -- (200.08,165.72);

\path[draw=drawColor,line width= 0.4pt,line join=round,line cap=round] (195.58,165.72) -- (200.08,161.22);

\path[draw=drawColor,line width= 0.4pt,line join=round,line cap=round] (218.80,161.42) -- (223.30,165.92);

\path[draw=drawColor,line width= 0.4pt,line join=round,line cap=round] (218.80,165.92) -- (223.30,161.42);

\path[draw=drawColor,line width= 0.4pt,line join=round,line cap=round] (242.03,161.42) -- (246.53,165.92);

\path[draw=drawColor,line width= 0.4pt,line join=round,line cap=round] (242.03,165.92) -- (246.53,161.42);

\path[draw=drawColor,line width= 0.4pt,line join=round,line cap=round] (265.25,161.42) -- (269.75,165.92);

\path[draw=drawColor,line width= 0.4pt,line join=round,line cap=round] (265.25,165.92) -- (269.75,161.42);

\path[draw=drawColor,line width= 0.4pt,line join=round,line cap=round] (288.48,161.42) -- (292.98,165.92);

\path[draw=drawColor,line width= 0.4pt,line join=round,line cap=round] (288.48,165.92) -- (292.98,161.42);
\definecolor{drawColor}{RGB}{0,0,0}

\path[draw=drawColor,line width= 0.4pt,line join=round,line cap=round] ( 49.20,167.61) rectangle ( 95.48,107.61);

\path[draw=drawColor,line width= 0.4pt,line join=round,line cap=round] ( 58.20,155.61) circle (  2.25);
\definecolor{drawColor}{RGB}{255,0,0}

\path[draw=drawColor,line width= 0.4pt,line join=round,line cap=round] ( 58.20,147.11) --
	( 61.23,141.86) --
	( 55.17,141.86) --
	cycle;
\definecolor{drawColor}{RGB}{0,255,0}

\path[draw=drawColor,line width= 0.4pt,line join=round,line cap=round] ( 55.02,131.61) -- ( 61.38,131.61);

\path[draw=drawColor,line width= 0.4pt,line join=round,line cap=round] ( 58.20,128.43) -- ( 58.20,134.79);
\definecolor{drawColor}{RGB}{0,0,255}

\path[draw=drawColor,line width= 0.4pt,line join=round,line cap=round] ( 55.95,117.36) -- ( 60.45,121.86);

\path[draw=drawColor,line width= 0.4pt,line join=round,line cap=round] ( 55.95,121.86) -- ( 60.45,117.36);
\definecolor{drawColor}{RGB}{0,0,0}

\node[text=drawColor,anchor=base west,inner sep=0pt, outer sep=0pt, scale=  1.00] at ( 67.20,152.17) {\itshape $\mathcal I_n$};

\node[text=drawColor,anchor=base west,inner sep=0pt, outer sep=0pt, scale=  1.00] at ( 67.20,140.17) {\itshape \normalfont $\mathcal I_n^\text{DC}$};

\node[text=drawColor,anchor=base west,inner sep=0pt, outer sep=0pt, scale=  1.00] at ( 67.20,128.17) {\itshape \normalfont $\mathcal I_n^\text{BC}$};

\node[text=drawColor,anchor=base west,inner sep=0pt, outer sep=0pt, scale=  1.00] at ( 67.20,116.17) {\itshape \normalfont $\mathcal I_n^\text{CvM}$};
\end{scope}
\end{tikzpicture}

%% file: powers-weight-exp_decay.tex
\begin{tikzpicture}[x=1pt,y=1pt]
\definecolor{fillColor}{RGB}{255,255,255}
\path[use as bounding box,fill=fillColor,fill opacity=0.00] (0,0) rectangle (325.21,216.81);
\begin{scope}
\path[clip] ( 49.20, 61.20) rectangle (300.01,167.61);
\definecolor{drawColor}{RGB}{0,0,0}

\path[draw=drawColor,line width= 0.4pt,line join=round,line cap=round] ( 64.17, 72.39) -- ( 76.03, 76.42);

\path[draw=drawColor,line width= 0.4pt,line join=round,line cap=round] ( 86.89, 81.38) -- ( 99.76, 88.91);

\path[draw=drawColor,line width= 0.4pt,line join=round,line cap=round] (110.08, 95.04) -- (123.02,102.84);

\path[draw=drawColor,line width= 0.4pt,line join=round,line cap=round] (132.12,110.44) -- (147.42,127.83);

\path[draw=drawColor,line width= 0.4pt,line join=round,line cap=round] (156.13,136.00) -- (169.86,146.60);

\path[draw=drawColor,line width= 0.4pt,line join=round,line cap=round] (180.15,152.57) -- (192.29,157.62);

\path[draw=drawColor,line width= 0.4pt,line join=round,line cap=round] (203.75,160.88) -- (215.13,162.71);

\path[draw=drawColor,line width= 0.4pt,line join=round,line cap=round] (227.05,163.67) -- (238.28,163.67);

\path[draw=drawColor,line width= 0.4pt,line join=round,line cap=round] (250.28,163.67) -- (261.50,163.67);

\path[draw=drawColor,line width= 0.4pt,line join=round,line cap=round] (273.50,163.67) -- (284.73,163.67);

\path[draw=drawColor,line width= 0.4pt,line join=round,line cap=round] ( 58.49, 70.46) circle (  2.25);

\path[draw=drawColor,line width= 0.4pt,line join=round,line cap=round] ( 81.71, 78.34) circle (  2.25);

\path[draw=drawColor,line width= 0.4pt,line join=round,line cap=round] (104.94, 91.94) circle (  2.25);

\path[draw=drawColor,line width= 0.4pt,line join=round,line cap=round] (128.16,105.93) circle (  2.25);

\path[draw=drawColor,line width= 0.4pt,line join=round,line cap=round] (151.38,132.34) circle (  2.25);

\path[draw=drawColor,line width= 0.4pt,line join=round,line cap=round] (174.61,150.27) circle (  2.25);

\path[draw=drawColor,line width= 0.4pt,line join=round,line cap=round] (197.83,159.92) circle (  2.25);

\path[draw=drawColor,line width= 0.4pt,line join=round,line cap=round] (221.05,163.67) circle (  2.25);

\path[draw=drawColor,line width= 0.4pt,line join=round,line cap=round] (244.28,163.67) circle (  2.25);

\path[draw=drawColor,line width= 0.4pt,line join=round,line cap=round] (267.50,163.67) circle (  2.25);

\path[draw=drawColor,line width= 0.4pt,line join=round,line cap=round] (290.73,163.67) circle (  2.25);
\end{scope}
\begin{scope}
\path[clip] (  0.00,  0.00) rectangle (325.21,216.81);
\definecolor{drawColor}{RGB}{0,0,0}

\path[draw=drawColor,line width= 0.4pt,line join=round,line cap=round] ( 58.49, 61.20) -- (290.73, 61.20);

\path[draw=drawColor,line width= 0.4pt,line join=round,line cap=round] ( 58.49, 61.20) -- ( 58.49, 55.20);

\path[draw=drawColor,line width= 0.4pt,line join=round,line cap=round] (104.94, 61.20) -- (104.94, 55.20);

\path[draw=drawColor,line width= 0.4pt,line join=round,line cap=round] (151.38, 61.20) -- (151.38, 55.20);

\path[draw=drawColor,line width= 0.4pt,line join=round,line cap=round] (197.83, 61.20) -- (197.83, 55.20);

\path[draw=drawColor,line width= 0.4pt,line join=round,line cap=round] (244.28, 61.20) -- (244.28, 55.20);

\path[draw=drawColor,line width= 0.4pt,line join=round,line cap=round] (290.73, 61.20) -- (290.73, 55.20);

\node[text=drawColor,anchor=base,inner sep=0pt, outer sep=0pt, scale=  1.00] at ( 58.49, 39.60) {0.0};

\node[text=drawColor,anchor=base,inner sep=0pt, outer sep=0pt, scale=  1.00] at (104.94, 39.60) {0.2};

\node[text=drawColor,anchor=base,inner sep=0pt, outer sep=0pt, scale=  1.00] at (151.38, 39.60) {0.4};

\node[text=drawColor,anchor=base,inner sep=0pt, outer sep=0pt, scale=  1.00] at (197.83, 39.60) {0.6};

\node[text=drawColor,anchor=base,inner sep=0pt, outer sep=0pt, scale=  1.00] at (244.28, 39.60) {0.8};

\node[text=drawColor,anchor=base,inner sep=0pt, outer sep=0pt, scale=  1.00] at (290.73, 39.60) {1.0};

\path[draw=drawColor,line width= 0.4pt,line join=round,line cap=round] ( 49.20, 65.14) -- ( 49.20,163.67);

\path[draw=drawColor,line width= 0.4pt,line join=round,line cap=round] ( 49.20, 65.14) -- ( 43.20, 65.14);

\path[draw=drawColor,line width= 0.4pt,line join=round,line cap=round] ( 49.20, 84.85) -- ( 43.20, 84.85);

\path[draw=drawColor,line width= 0.4pt,line join=round,line cap=round] ( 49.20,104.55) -- ( 43.20,104.55);

\path[draw=drawColor,line width= 0.4pt,line join=round,line cap=round] ( 49.20,124.26) -- ( 43.20,124.26);

\path[draw=drawColor,line width= 0.4pt,line join=round,line cap=round] ( 49.20,143.96) -- ( 43.20,143.96);

\path[draw=drawColor,line width= 0.4pt,line join=round,line cap=round] ( 49.20,163.67) -- ( 43.20,163.67);

\node[text=drawColor,rotate= 90.00,anchor=base,inner sep=0pt, outer sep=0pt, scale=  1.00] at ( 34.80, 65.14) {0.0};

\node[text=drawColor,rotate= 90.00,anchor=base,inner sep=0pt, outer sep=0pt, scale=  1.00] at ( 34.80,104.55) {0.4};

\node[text=drawColor,rotate= 90.00,anchor=base,inner sep=0pt, outer sep=0pt, scale=  1.00] at ( 34.80,143.96) {0.8};

\path[draw=drawColor,line width= 0.4pt,line join=round,line cap=round] ( 49.20, 61.20) --
	(300.01, 61.20) --
	(300.01,167.61) --
	( 49.20,167.61) --
	cycle;
\end{scope}
\begin{scope}
\path[clip] (  0.00,  0.00) rectangle (325.21,216.81);
\definecolor{drawColor}{RGB}{0,0,0}

\node[text=drawColor,anchor=base,inner sep=0pt, outer sep=0pt, scale=  1.00] at (174.61, 15.60) {$r^2$};

\node[text=drawColor,rotate= 90.00,anchor=base,inner sep=0pt, outer sep=0pt, scale=  1.00] at ( 10.80,114.41) {Power};

\node[text=drawColor,anchor=base,inner sep=0pt, outer sep=0pt, scale=  1.20] at (174.61,188.07) {\bfseries \texttt{weight}};
\end{scope}
\begin{scope}
\path[clip] ( 49.20, 61.20) rectangle (300.01,167.61);
\definecolor{drawColor}{RGB}{255,0,0}

\path[draw=drawColor,line width= 0.4pt,dash pattern=on 4pt off 4pt ,line join=round,line cap=round] ( 60.21, 77.00) -- ( 79.99,142.95);

\path[draw=drawColor,line width= 0.4pt,dash pattern=on 4pt off 4pt ,line join=round,line cap=round] ( 86.85,151.79) -- ( 99.80,159.59);

\path[draw=drawColor,line width= 0.4pt,dash pattern=on 4pt off 4pt ,line join=round,line cap=round] (110.93,162.94) -- (122.17,163.41);

\path[draw=drawColor,line width= 0.4pt,dash pattern=on 4pt off 4pt ,line join=round,line cap=round] (134.16,163.67) -- (145.38,163.67);

\path[draw=drawColor,line width= 0.4pt,dash pattern=on 4pt off 4pt ,line join=round,line cap=round] (157.38,163.67) -- (168.61,163.67);

\path[draw=drawColor,line width= 0.4pt,dash pattern=on 4pt off 4pt ,line join=round,line cap=round] (180.61,163.67) -- (191.83,163.67);

\path[draw=drawColor,line width= 0.4pt,dash pattern=on 4pt off 4pt ,line join=round,line cap=round] (203.83,163.67) -- (215.05,163.67);

\path[draw=drawColor,line width= 0.4pt,dash pattern=on 4pt off 4pt ,line join=round,line cap=round] (227.05,163.67) -- (238.28,163.67);

\path[draw=drawColor,line width= 0.4pt,dash pattern=on 4pt off 4pt ,line join=round,line cap=round] (250.28,163.67) -- (261.50,163.67);

\path[draw=drawColor,line width= 0.4pt,dash pattern=on 4pt off 4pt ,line join=round,line cap=round] (273.50,163.67) -- (284.73,163.67);

\path[draw=drawColor,line width= 0.4pt,line join=round,line cap=round] ( 58.49, 74.75) --
	( 61.52, 69.50) --
	( 55.46, 69.50) --
	cycle;

\path[draw=drawColor,line width= 0.4pt,line join=round,line cap=round] ( 81.71,152.19) --
	( 84.74,146.94) --
	( 78.68,146.94) --
	cycle;

\path[draw=drawColor,line width= 0.4pt,line join=round,line cap=round] (104.94,166.18) --
	(107.97,160.93) --
	(101.91,160.93) --
	cycle;

\path[draw=drawColor,line width= 0.4pt,line join=round,line cap=round] (128.16,167.17) --
	(131.19,161.92) --
	(125.13,161.92) --
	cycle;

\path[draw=drawColor,line width= 0.4pt,line join=round,line cap=round] (151.38,167.17) --
	(154.41,161.92) --
	(148.35,161.92) --
	cycle;

\path[draw=drawColor,line width= 0.4pt,line join=round,line cap=round] (174.61,167.17) --
	(177.64,161.92) --
	(171.58,161.92) --
	cycle;

\path[draw=drawColor,line width= 0.4pt,line join=round,line cap=round] (197.83,167.17) --
	(200.86,161.92) --
	(194.80,161.92) --
	cycle;

\path[draw=drawColor,line width= 0.4pt,line join=round,line cap=round] (221.05,167.17) --
	(224.08,161.92) --
	(218.02,161.92) --
	cycle;

\path[draw=drawColor,line width= 0.4pt,line join=round,line cap=round] (244.28,167.17) --
	(247.31,161.92) --
	(241.25,161.92) --
	cycle;

\path[draw=drawColor,line width= 0.4pt,line join=round,line cap=round] (267.50,167.17) --
	(270.53,161.92) --
	(264.47,161.92) --
	cycle;

\path[draw=drawColor,line width= 0.4pt,line join=round,line cap=round] (290.73,167.17) --
	(293.76,161.92) --
	(287.70,161.92) --
	cycle;
\definecolor{drawColor}{RGB}{0,0,255}

\path[draw=drawColor,line width= 0.4pt,dash pattern=on 1pt off 3pt on 4pt off 3pt ,line join=round,line cap=round] ( 64.17, 72.19) -- ( 76.03, 76.22);

\path[draw=drawColor,line width= 0.4pt,dash pattern=on 1pt off 3pt on 4pt off 3pt ,line join=round,line cap=round] ( 85.53, 82.78) -- (101.12,101.70);

\path[draw=drawColor,line width= 0.4pt,dash pattern=on 1pt off 3pt on 4pt off 3pt ,line join=round,line cap=round] (108.44,111.20) -- (124.66,133.77);

\path[draw=drawColor,line width= 0.4pt,dash pattern=on 1pt off 3pt on 4pt off 3pt ,line join=round,line cap=round] (132.68,142.59) -- (146.87,154.99);

\path[draw=drawColor,line width= 0.4pt,dash pattern=on 1pt off 3pt on 4pt off 3pt ,line join=round,line cap=round] (157.28,160.04) -- (168.71,162.17);

\path[draw=drawColor,line width= 0.4pt,dash pattern=on 1pt off 3pt on 4pt off 3pt ,line join=round,line cap=round] (180.61,163.38) -- (191.83,163.57);

\path[draw=drawColor,line width= 0.4pt,dash pattern=on 1pt off 3pt on 4pt off 3pt ,line join=round,line cap=round] (203.83,163.67) -- (215.05,163.67);

\path[draw=drawColor,line width= 0.4pt,dash pattern=on 1pt off 3pt on 4pt off 3pt ,line join=round,line cap=round] (227.05,163.67) -- (238.28,163.67);

\path[draw=drawColor,line width= 0.4pt,dash pattern=on 1pt off 3pt on 4pt off 3pt ,line join=round,line cap=round] (250.28,163.67) -- (261.50,163.67);

\path[draw=drawColor,line width= 0.4pt,dash pattern=on 1pt off 3pt on 4pt off 3pt ,line join=round,line cap=round] (273.50,163.67) -- (284.73,163.67);

\path[draw=drawColor,line width= 0.4pt,line join=round,line cap=round] ( 55.31, 70.26) -- ( 61.67, 70.26);

\path[draw=drawColor,line width= 0.4pt,line join=round,line cap=round] ( 58.49, 67.08) -- ( 58.49, 73.45);

\path[draw=drawColor,line width= 0.4pt,line join=round,line cap=round] ( 78.53, 78.15) -- ( 84.90, 78.15);

\path[draw=drawColor,line width= 0.4pt,line join=round,line cap=round] ( 81.71, 74.96) -- ( 81.71, 81.33);

\path[draw=drawColor,line width= 0.4pt,line join=round,line cap=round] (101.75,106.33) -- (108.12,106.33);

\path[draw=drawColor,line width= 0.4pt,line join=round,line cap=round] (104.94,103.14) -- (104.94,109.51);

\path[draw=drawColor,line width= 0.4pt,line join=round,line cap=round] (124.98,138.64) -- (131.34,138.64);

\path[draw=drawColor,line width= 0.4pt,line join=round,line cap=round] (128.16,135.46) -- (128.16,141.82);

\path[draw=drawColor,line width= 0.4pt,line join=round,line cap=round] (148.20,158.94) -- (154.57,158.94);

\path[draw=drawColor,line width= 0.4pt,line join=round,line cap=round] (151.38,155.76) -- (151.38,162.12);

\path[draw=drawColor,line width= 0.4pt,line join=round,line cap=round] (171.43,163.27) -- (177.79,163.27);

\path[draw=drawColor,line width= 0.4pt,line join=round,line cap=round] (174.61,160.09) -- (174.61,166.46);

\path[draw=drawColor,line width= 0.4pt,line join=round,line cap=round] (194.65,163.67) -- (201.01,163.67);

\path[draw=drawColor,line width= 0.4pt,line join=round,line cap=round] (197.83,160.49) -- (197.83,166.85);

\path[draw=drawColor,line width= 0.4pt,line join=round,line cap=round] (217.87,163.67) -- (224.24,163.67);

\path[draw=drawColor,line width= 0.4pt,line join=round,line cap=round] (221.05,160.49) -- (221.05,166.85);

\path[draw=drawColor,line width= 0.4pt,line join=round,line cap=round] (241.10,163.67) -- (247.46,163.67);

\path[draw=drawColor,line width= 0.4pt,line join=round,line cap=round] (244.28,160.49) -- (244.28,166.85);

\path[draw=drawColor,line width= 0.4pt,line join=round,line cap=round] (264.32,163.67) -- (270.68,163.67);

\path[draw=drawColor,line width= 0.4pt,line join=round,line cap=round] (267.50,160.49) -- (267.50,166.85);

\path[draw=drawColor,line width= 0.4pt,line join=round,line cap=round] (287.54,163.67) -- (293.91,163.67);

\path[draw=drawColor,line width= 0.4pt,line join=round,line cap=round] (290.73,160.49) -- (290.73,166.85);
\definecolor{drawColor}{RGB}{0,255,0}

\path[draw=drawColor,line width= 0.4pt,dash pattern=on 1pt off 3pt on 4pt off 3pt ,line join=round,line cap=round] ( 63.18, 72.04) -- ( 77.02, 83.08);

\path[draw=drawColor,line width= 0.4pt,dash pattern=on 1pt off 3pt on 4pt off 3pt ,line join=round,line cap=round] ( 85.32, 91.62) -- (101.33,112.96);

\path[draw=drawColor,line width= 0.4pt,dash pattern=on 1pt off 3pt on 4pt off 3pt ,line join=round,line cap=round] (109.04,122.14) -- (124.06,138.20);

\path[draw=drawColor,line width= 0.4pt,dash pattern=on 1pt off 3pt on 4pt off 3pt ,line join=round,line cap=round] (133.36,145.58) -- (146.19,152.99);

\path[draw=drawColor,line width= 0.4pt,dash pattern=on 1pt off 3pt on 4pt off 3pt ,line join=round,line cap=round] (157.21,157.42) -- (168.78,160.26);

\path[draw=drawColor,line width= 0.4pt,dash pattern=on 1pt off 3pt on 4pt off 3pt ,line join=round,line cap=round] (180.59,162.21) -- (191.85,163.16);

\path[draw=drawColor,line width= 0.4pt,dash pattern=on 1pt off 3pt on 4pt off 3pt ,line join=round,line cap=round] (203.83,163.67) -- (215.05,163.67);

\path[draw=drawColor,line width= 0.4pt,dash pattern=on 1pt off 3pt on 4pt off 3pt ,line join=round,line cap=round] (227.05,163.67) -- (238.28,163.67);

\path[draw=drawColor,line width= 0.4pt,dash pattern=on 1pt off 3pt on 4pt off 3pt ,line join=round,line cap=round] (250.28,163.67) -- (261.50,163.67);

\path[draw=drawColor,line width= 0.4pt,dash pattern=on 1pt off 3pt on 4pt off 3pt ,line join=round,line cap=round] (273.50,163.67) -- (284.73,163.67);

\path[draw=drawColor,line width= 0.4pt,line join=round,line cap=round] ( 56.24, 66.04) -- ( 60.74, 70.54);

\path[draw=drawColor,line width= 0.4pt,line join=round,line cap=round] ( 56.24, 70.54) -- ( 60.74, 66.04);

\path[draw=drawColor,line width= 0.4pt,line join=round,line cap=round] ( 79.46, 84.57) -- ( 83.96, 89.07);

\path[draw=drawColor,line width= 0.4pt,line join=round,line cap=round] ( 79.46, 89.07) -- ( 83.96, 84.57);

\path[draw=drawColor,line width= 0.4pt,line join=round,line cap=round] (102.69,115.50) -- (107.19,120.00);

\path[draw=drawColor,line width= 0.4pt,line join=round,line cap=round] (102.69,120.00) -- (107.19,115.50);

\path[draw=drawColor,line width= 0.4pt,line join=round,line cap=round] (125.91,140.33) -- (130.41,144.83);

\path[draw=drawColor,line width= 0.4pt,line join=round,line cap=round] (125.91,144.83) -- (130.41,140.33);

\path[draw=drawColor,line width= 0.4pt,line join=round,line cap=round] (149.13,153.73) -- (153.63,158.23);

\path[draw=drawColor,line width= 0.4pt,line join=round,line cap=round] (149.13,158.23) -- (153.63,153.73);

\path[draw=drawColor,line width= 0.4pt,line join=round,line cap=round] (172.36,159.45) -- (176.86,163.95);

\path[draw=drawColor,line width= 0.4pt,line join=round,line cap=round] (172.36,163.95) -- (176.86,159.45);

\path[draw=drawColor,line width= 0.4pt,line join=round,line cap=round] (195.58,161.42) -- (200.08,165.92);

\path[draw=drawColor,line width= 0.4pt,line join=round,line cap=round] (195.58,165.92) -- (200.08,161.42);

\path[draw=drawColor,line width= 0.4pt,line join=round,line cap=round] (218.80,161.42) -- (223.30,165.92);

\path[draw=drawColor,line width= 0.4pt,line join=round,line cap=round] (218.80,165.92) -- (223.30,161.42);

\path[draw=drawColor,line width= 0.4pt,line join=round,line cap=round] (242.03,161.42) -- (246.53,165.92);

\path[draw=drawColor,line width= 0.4pt,line join=round,line cap=round] (242.03,165.92) -- (246.53,161.42);

\path[draw=drawColor,line width= 0.4pt,line join=round,line cap=round] (265.25,161.42) -- (269.75,165.92);

\path[draw=drawColor,line width= 0.4pt,line join=round,line cap=round] (265.25,165.92) -- (269.75,161.42);

\path[draw=drawColor,line width= 0.4pt,line join=round,line cap=round] (288.48,161.42) -- (292.98,165.92);

\path[draw=drawColor,line width= 0.4pt,line join=round,line cap=round] (288.48,165.92) -- (292.98,161.42);
\definecolor{drawColor}{RGB}{0,0,0}

\path[draw=drawColor,line width= 0.4pt,line join=round,line cap=round] (253.73,121.20) rectangle (300.01, 61.20);

\path[draw=drawColor,line width= 0.4pt,line join=round,line cap=round] (262.73,109.20) circle (  2.25);
\definecolor{drawColor}{RGB}{255,0,0}

\path[draw=drawColor,line width= 0.4pt,line join=round,line cap=round] (262.73,100.70) --
	(265.76, 95.45) --
	(259.70, 95.45) --
	cycle;
\definecolor{drawColor}{RGB}{0,255,0}

\path[draw=drawColor,line width= 0.4pt,line join=round,line cap=round] (259.55, 85.20) -- (265.91, 85.20);

\path[draw=drawColor,line width= 0.4pt,line join=round,line cap=round] (262.73, 82.02) -- (262.73, 88.38);
\definecolor{drawColor}{RGB}{0,0,255}

\path[draw=drawColor,line width= 0.4pt,line join=round,line cap=round] (260.48, 70.95) -- (264.98, 75.45);

\path[draw=drawColor,line width= 0.4pt,line join=round,line cap=round] (260.48, 75.45) -- (264.98, 70.95);
\definecolor{drawColor}{RGB}{0,0,0}

\node[text=drawColor,anchor=base west,inner sep=0pt, outer sep=0pt, scale=  1.00] at (271.73,105.76) {\itshape $\mathcal I_n$};

\node[text=drawColor,anchor=base west,inner sep=0pt, outer sep=0pt, scale=  1.00] at (271.73, 93.76) {\itshape \normalfont $\mathcal I_n^\text{DC}$};

\node[text=drawColor,anchor=base west,inner sep=0pt, outer sep=0pt, scale=  1.00] at (271.73, 81.76) {\itshape \normalfont $\mathcal I_n^\text{BC}$};

\node[text=drawColor,anchor=base west,inner sep=0pt, outer sep=0pt, scale=  1.00] at (271.73, 69.76) {\itshape \normalfont $\mathcal I_n^\text{CvM}$};
\end{scope}
\end{tikzpicture}

%% file: powers-sin-exp_decay.tex
\begin{tikzpicture}[x=1pt,y=1pt]
\definecolor{fillColor}{RGB}{255,255,255}
\path[use as bounding box,fill=fillColor,fill opacity=0.00] (0,0) rectangle (325.21,216.81);
\begin{scope}
\path[clip] ( 49.20, 61.20) rectangle (300.01,167.61);
\definecolor{drawColor}{RGB}{0,0,0}

\path[draw=drawColor,line width= 0.4pt,line join=round,line cap=round] ( 64.32, 71.90) -- ( 75.89, 74.74);

\path[draw=drawColor,line width= 0.4pt,line join=round,line cap=round] ( 87.08, 78.86) -- ( 99.57, 85.12);

\path[draw=drawColor,line width= 0.4pt,line join=round,line cap=round] (110.11, 90.83) -- (122.98, 98.37);

\path[draw=drawColor,line width= 0.4pt,line join=round,line cap=round] (132.75,105.26) -- (146.79,117.05);

\path[draw=drawColor,line width= 0.4pt,line join=round,line cap=round] (156.06,124.67) -- (169.94,135.86);

\path[draw=drawColor,line width= 0.4pt,line join=round,line cap=round] (179.61,142.94) -- (192.83,151.69);

\path[draw=drawColor,line width= 0.4pt,line join=round,line cap=round] (203.60,156.66) -- (215.29,160.04);

\path[draw=drawColor,line width= 0.4pt,line join=round,line cap=round] (227.03,162.21) -- (238.30,163.16);

\path[draw=drawColor,line width= 0.4pt,line join=round,line cap=round] (250.28,163.67) -- (261.50,163.67);

\path[draw=drawColor,line width= 0.4pt,line join=round,line cap=round] (273.50,163.67) -- (284.73,163.67);

\path[draw=drawColor,line width= 0.4pt,line join=round,line cap=round] ( 58.49, 70.46) circle (  2.25);

\path[draw=drawColor,line width= 0.4pt,line join=round,line cap=round] ( 81.71, 76.18) circle (  2.25);

\path[draw=drawColor,line width= 0.4pt,line join=round,line cap=round] (104.94, 87.80) circle (  2.25);

\path[draw=drawColor,line width= 0.4pt,line join=round,line cap=round] (128.16,101.40) circle (  2.25);

\path[draw=drawColor,line width= 0.4pt,line join=round,line cap=round] (151.38,120.91) circle (  2.25);

\path[draw=drawColor,line width= 0.4pt,line join=round,line cap=round] (174.61,139.63) circle (  2.25);

\path[draw=drawColor,line width= 0.4pt,line join=round,line cap=round] (197.83,155.00) circle (  2.25);

\path[draw=drawColor,line width= 0.4pt,line join=round,line cap=round] (221.05,161.70) circle (  2.25);

\path[draw=drawColor,line width= 0.4pt,line join=round,line cap=round] (244.28,163.67) circle (  2.25);

\path[draw=drawColor,line width= 0.4pt,line join=round,line cap=round] (267.50,163.67) circle (  2.25);

\path[draw=drawColor,line width= 0.4pt,line join=round,line cap=round] (290.73,163.67) circle (  2.25);
\end{scope}
\begin{scope}
\path[clip] (  0.00,  0.00) rectangle (325.21,216.81);
\definecolor{drawColor}{RGB}{0,0,0}

\path[draw=drawColor,line width= 0.4pt,line join=round,line cap=round] ( 58.49, 61.20) -- (290.73, 61.20);

\path[draw=drawColor,line width= 0.4pt,line join=round,line cap=round] ( 58.49, 61.20) -- ( 58.49, 55.20);

\path[draw=drawColor,line width= 0.4pt,line join=round,line cap=round] (104.94, 61.20) -- (104.94, 55.20);

\path[draw=drawColor,line width= 0.4pt,line join=round,line cap=round] (151.38, 61.20) -- (151.38, 55.20);

\path[draw=drawColor,line width= 0.4pt,line join=round,line cap=round] (197.83, 61.20) -- (197.83, 55.20);

\path[draw=drawColor,line width= 0.4pt,line join=round,line cap=round] (244.28, 61.20) -- (244.28, 55.20);

\path[draw=drawColor,line width= 0.4pt,line join=round,line cap=round] (290.73, 61.20) -- (290.73, 55.20);

\node[text=drawColor,anchor=base,inner sep=0pt, outer sep=0pt, scale=  1.00] at ( 58.49, 39.60) {0.0};

\node[text=drawColor,anchor=base,inner sep=0pt, outer sep=0pt, scale=  1.00] at (104.94, 39.60) {0.2};

\node[text=drawColor,anchor=base,inner sep=0pt, outer sep=0pt, scale=  1.00] at (151.38, 39.60) {0.4};

\node[text=drawColor,anchor=base,inner sep=0pt, outer sep=0pt, scale=  1.00] at (197.83, 39.60) {0.6};

\node[text=drawColor,anchor=base,inner sep=0pt, outer sep=0pt, scale=  1.00] at (244.28, 39.60) {0.8};

\node[text=drawColor,anchor=base,inner sep=0pt, outer sep=0pt, scale=  1.00] at (290.73, 39.60) {1.0};

\path[draw=drawColor,line width= 0.4pt,line join=round,line cap=round] ( 49.20, 65.14) -- ( 49.20,163.67);

\path[draw=drawColor,line width= 0.4pt,line join=round,line cap=round] ( 49.20, 65.14) -- ( 43.20, 65.14);

\path[draw=drawColor,line width= 0.4pt,line join=round,line cap=round] ( 49.20, 84.85) -- ( 43.20, 84.85);

\path[draw=drawColor,line width= 0.4pt,line join=round,line cap=round] ( 49.20,104.55) -- ( 43.20,104.55);

\path[draw=drawColor,line width= 0.4pt,line join=round,line cap=round] ( 49.20,124.26) -- ( 43.20,124.26);

\path[draw=drawColor,line width= 0.4pt,line join=round,line cap=round] ( 49.20,143.96) -- ( 43.20,143.96);

\path[draw=drawColor,line width= 0.4pt,line join=round,line cap=round] ( 49.20,163.67) -- ( 43.20,163.67);

\node[text=drawColor,rotate= 90.00,anchor=base,inner sep=0pt, outer sep=0pt, scale=  1.00] at ( 34.80, 65.14) {0.0};

\node[text=drawColor,rotate= 90.00,anchor=base,inner sep=0pt, outer sep=0pt, scale=  1.00] at ( 34.80,104.55) {0.4};

\node[text=drawColor,rotate= 90.00,anchor=base,inner sep=0pt, outer sep=0pt, scale=  1.00] at ( 34.80,143.96) {0.8};

\path[draw=drawColor,line width= 0.4pt,line join=round,line cap=round] ( 49.20, 61.20) --
	(300.01, 61.20) --
	(300.01,167.61) --
	( 49.20,167.61) --
	cycle;
\end{scope}
\begin{scope}
\path[clip] (  0.00,  0.00) rectangle (325.21,216.81);
\definecolor{drawColor}{RGB}{0,0,0}

\node[text=drawColor,anchor=base,inner sep=0pt, outer sep=0pt, scale=  1.00] at (174.61, 15.60) {$r^2$};

\node[text=drawColor,rotate= 90.00,anchor=base,inner sep=0pt, outer sep=0pt, scale=  1.00] at ( 10.80,114.41) {Power};

\node[text=drawColor,anchor=base,inner sep=0pt, outer sep=0pt, scale=  1.20] at (174.61,188.07) {\bfseries \texttt{sin}};
\end{scope}
\begin{scope}
\path[clip] ( 49.20, 61.20) rectangle (300.01,167.61);
\definecolor{drawColor}{RGB}{255,0,0}

\path[draw=drawColor,line width= 0.4pt,dash pattern=on 4pt off 4pt ,line join=round,line cap=round] ( 64.47, 71.71) -- ( 75.73, 72.57);

\path[draw=drawColor,line width= 0.4pt,dash pattern=on 4pt off 4pt ,line join=round,line cap=round] ( 87.71, 73.23) -- ( 98.94, 73.61);

\path[draw=drawColor,line width= 0.4pt,dash pattern=on 4pt off 4pt ,line join=round,line cap=round] (110.84, 74.86) -- (122.25, 76.90);

\path[draw=drawColor,line width= 0.4pt,dash pattern=on 4pt off 4pt ,line join=round,line cap=round] (134.05, 79.10) -- (145.49, 81.33);

\path[draw=drawColor,line width= 0.4pt,dash pattern=on 4pt off 4pt ,line join=round,line cap=round] (157.22, 83.87) -- (168.77, 86.61);

\path[draw=drawColor,line width= 0.4pt,dash pattern=on 4pt off 4pt ,line join=round,line cap=round] (180.57, 88.66) -- (191.87, 89.90);

\path[draw=drawColor,line width= 0.4pt,dash pattern=on 4pt off 4pt ,line join=round,line cap=round] (203.37, 92.86) -- (215.51, 97.91);

\path[draw=drawColor,line width= 0.4pt,dash pattern=on 4pt off 4pt ,line join=round,line cap=round] (226.81,101.93) -- (238.53,105.41);

\path[draw=drawColor,line width= 0.4pt,dash pattern=on 4pt off 4pt ,line join=round,line cap=round] (249.55,109.98) -- (262.23,116.86);

\path[draw=drawColor,line width= 0.4pt,dash pattern=on 4pt off 4pt ,line join=round,line cap=round] (273.48,120.23) -- (284.75,121.19);

\path[draw=drawColor,line width= 0.4pt,line join=round,line cap=round] ( 58.49, 74.75) --
	( 61.52, 69.50) --
	( 55.46, 69.50) --
	cycle;

\path[draw=drawColor,line width= 0.4pt,line join=round,line cap=round] ( 81.71, 76.52) --
	( 84.74, 71.27) --
	( 78.68, 71.27) --
	cycle;

\path[draw=drawColor,line width= 0.4pt,line join=round,line cap=round] (104.94, 77.31) --
	(107.97, 72.06) --
	(101.91, 72.06) --
	cycle;

\path[draw=drawColor,line width= 0.4pt,line join=round,line cap=round] (128.16, 81.45) --
	(131.19, 76.20) --
	(125.13, 76.20) --
	cycle;

\path[draw=drawColor,line width= 0.4pt,line join=round,line cap=round] (151.38, 85.98) --
	(154.41, 80.73) --
	(148.35, 80.73) --
	cycle;

\path[draw=drawColor,line width= 0.4pt,line join=round,line cap=round] (174.61, 91.50) --
	(177.64, 86.25) --
	(171.58, 86.25) --
	cycle;

\path[draw=drawColor,line width= 0.4pt,line join=round,line cap=round] (197.83, 94.06) --
	(200.86, 88.81) --
	(194.80, 88.81) --
	cycle;

\path[draw=drawColor,line width= 0.4pt,line join=round,line cap=round] (221.05,103.72) --
	(224.08, 98.47) --
	(218.02, 98.47) --
	cycle;

\path[draw=drawColor,line width= 0.4pt,line join=round,line cap=round] (244.28,110.61) --
	(247.31,105.36) --
	(241.25,105.36) --
	cycle;

\path[draw=drawColor,line width= 0.4pt,line join=round,line cap=round] (267.50,123.22) --
	(270.53,117.98) --
	(264.47,117.98) --
	cycle;

\path[draw=drawColor,line width= 0.4pt,line join=round,line cap=round] (290.73,125.20) --
	(293.76,119.95) --
	(287.70,119.95) --
	cycle;
\definecolor{drawColor}{RGB}{0,0,255}

\path[draw=drawColor,line width= 0.4pt,dash pattern=on 1pt off 3pt on 4pt off 3pt ,line join=round,line cap=round] ( 64.49, 70.06) -- ( 75.72, 69.68);

\path[draw=drawColor,line width= 0.4pt,dash pattern=on 1pt off 3pt on 4pt off 3pt ,line join=round,line cap=round] ( 87.71, 69.32) -- ( 98.94, 69.04);

\path[draw=drawColor,line width= 0.4pt,dash pattern=on 1pt off 3pt on 4pt off 3pt ,line join=round,line cap=round] (110.94, 68.94) -- (122.16, 69.03);

\path[draw=drawColor,line width= 0.4pt,dash pattern=on 1pt off 3pt on 4pt off 3pt ,line join=round,line cap=round] (134.15, 69.44) -- (145.39, 70.11);

\path[draw=drawColor,line width= 0.4pt,dash pattern=on 1pt off 3pt on 4pt off 3pt ,line join=round,line cap=round] (157.38, 70.61) -- (168.61, 70.90);

\path[draw=drawColor,line width= 0.4pt,dash pattern=on 1pt off 3pt on 4pt off 3pt ,line join=round,line cap=round] (180.59, 71.51) -- (191.85, 72.37);

\path[draw=drawColor,line width= 0.4pt,dash pattern=on 1pt off 3pt on 4pt off 3pt ,line join=round,line cap=round] (203.78, 72.02) -- (215.11, 70.48);

\path[draw=drawColor,line width= 0.4pt,dash pattern=on 1pt off 3pt on 4pt off 3pt ,line join=round,line cap=round] (227.03, 70.18) -- (238.30, 71.14);

\path[draw=drawColor,line width= 0.4pt,dash pattern=on 1pt off 3pt on 4pt off 3pt ,line join=round,line cap=round] (250.24, 70.99) -- (261.54, 69.74);

\path[draw=drawColor,line width= 0.4pt,dash pattern=on 1pt off 3pt on 4pt off 3pt ,line join=round,line cap=round] (273.50, 69.29) -- (284.73, 69.67);

\path[draw=drawColor,line width= 0.4pt,line join=round,line cap=round] ( 55.31, 70.26) -- ( 61.67, 70.26);

\path[draw=drawColor,line width= 0.4pt,line join=round,line cap=round] ( 58.49, 67.08) -- ( 58.49, 73.45);

\path[draw=drawColor,line width= 0.4pt,line join=round,line cap=round] ( 78.53, 69.48) -- ( 84.90, 69.48);

\path[draw=drawColor,line width= 0.4pt,line join=round,line cap=round] ( 81.71, 66.29) -- ( 81.71, 72.66);

\path[draw=drawColor,line width= 0.4pt,line join=round,line cap=round] (101.75, 68.89) -- (108.12, 68.89);

\path[draw=drawColor,line width= 0.4pt,line join=round,line cap=round] (104.94, 65.70) -- (104.94, 72.07);

\path[draw=drawColor,line width= 0.4pt,line join=round,line cap=round] (124.98, 69.08) -- (131.34, 69.08);

\path[draw=drawColor,line width= 0.4pt,line join=round,line cap=round] (128.16, 65.90) -- (128.16, 72.26);

\path[draw=drawColor,line width= 0.4pt,line join=round,line cap=round] (148.20, 70.46) -- (154.57, 70.46);

\path[draw=drawColor,line width= 0.4pt,line join=round,line cap=round] (151.38, 67.28) -- (151.38, 73.64);

\path[draw=drawColor,line width= 0.4pt,line join=round,line cap=round] (171.43, 71.05) -- (177.79, 71.05);

\path[draw=drawColor,line width= 0.4pt,line join=round,line cap=round] (174.61, 67.87) -- (174.61, 74.23);

\path[draw=drawColor,line width= 0.4pt,line join=round,line cap=round] (194.65, 72.83) -- (201.01, 72.83);

\path[draw=drawColor,line width= 0.4pt,line join=round,line cap=round] (197.83, 69.64) -- (197.83, 76.01);

\path[draw=drawColor,line width= 0.4pt,line join=round,line cap=round] (217.87, 69.67) -- (224.24, 69.67);

\path[draw=drawColor,line width= 0.4pt,line join=round,line cap=round] (221.05, 66.49) -- (221.05, 72.86);

\path[draw=drawColor,line width= 0.4pt,line join=round,line cap=round] (241.10, 71.64) -- (247.46, 71.64);

\path[draw=drawColor,line width= 0.4pt,line join=round,line cap=round] (244.28, 68.46) -- (244.28, 74.83);

\path[draw=drawColor,line width= 0.4pt,line join=round,line cap=round] (264.32, 69.08) -- (270.68, 69.08);

\path[draw=drawColor,line width= 0.4pt,line join=round,line cap=round] (267.50, 65.90) -- (267.50, 72.26);

\path[draw=drawColor,line width= 0.4pt,line join=round,line cap=round] (287.54, 69.87) -- (293.91, 69.87);

\path[draw=drawColor,line width= 0.4pt,line join=round,line cap=round] (290.73, 66.69) -- (290.73, 73.05);
\definecolor{drawColor}{RGB}{0,255,0}

\path[draw=drawColor,line width= 0.4pt,dash pattern=on 1pt off 3pt on 4pt off 3pt ,line join=round,line cap=round] ( 64.16, 70.27) -- ( 76.05, 74.40);

\path[draw=drawColor,line width= 0.4pt,dash pattern=on 1pt off 3pt on 4pt off 3pt ,line join=round,line cap=round] ( 87.38, 78.34) -- ( 99.27, 82.48);

\path[draw=drawColor,line width= 0.4pt,dash pattern=on 1pt off 3pt on 4pt off 3pt ,line join=round,line cap=round] (110.21, 87.32) -- (122.89, 94.20);

\path[draw=drawColor,line width= 0.4pt,dash pattern=on 1pt off 3pt on 4pt off 3pt ,line join=round,line cap=round] (133.51, 99.79) -- (146.04,106.17);

\path[draw=drawColor,line width= 0.4pt,dash pattern=on 1pt off 3pt on 4pt off 3pt ,line join=round,line cap=round] (155.83,112.92) -- (170.17,125.94);

\path[draw=drawColor,line width= 0.4pt,dash pattern=on 1pt off 3pt on 4pt off 3pt ,line join=round,line cap=round] (179.63,133.25) -- (192.81,141.86);

\path[draw=drawColor,line width= 0.4pt,dash pattern=on 1pt off 3pt on 4pt off 3pt ,line join=round,line cap=round] (203.42,147.33) -- (215.47,152.03);

\path[draw=drawColor,line width= 0.4pt,dash pattern=on 1pt off 3pt on 4pt off 3pt ,line join=round,line cap=round] (226.85,155.78) -- (238.49,158.94);

\path[draw=drawColor,line width= 0.4pt,dash pattern=on 1pt off 3pt on 4pt off 3pt ,line join=round,line cap=round] (250.24,161.22) -- (261.54,162.57);

\path[draw=drawColor,line width= 0.4pt,dash pattern=on 1pt off 3pt on 4pt off 3pt ,line join=round,line cap=round] (273.50,163.38) -- (284.73,163.57);

\path[draw=drawColor,line width= 0.4pt,line join=round,line cap=round] ( 56.24, 66.04) -- ( 60.74, 70.54);

\path[draw=drawColor,line width= 0.4pt,line join=round,line cap=round] ( 56.24, 70.54) -- ( 60.74, 66.04);

\path[draw=drawColor,line width= 0.4pt,line join=round,line cap=round] ( 79.46, 74.12) -- ( 83.96, 78.62);

\path[draw=drawColor,line width= 0.4pt,line join=round,line cap=round] ( 79.46, 78.62) -- ( 83.96, 74.12);

\path[draw=drawColor,line width= 0.4pt,line join=round,line cap=round] (102.69, 82.20) -- (107.19, 86.70);

\path[draw=drawColor,line width= 0.4pt,line join=round,line cap=round] (102.69, 86.70) -- (107.19, 82.20);

\path[draw=drawColor,line width= 0.4pt,line join=round,line cap=round] (125.91, 94.81) -- (130.41, 99.31);

\path[draw=drawColor,line width= 0.4pt,line join=round,line cap=round] (125.91, 99.31) -- (130.41, 94.81);

\path[draw=drawColor,line width= 0.4pt,line join=round,line cap=round] (149.13,106.64) -- (153.63,111.14);

\path[draw=drawColor,line width= 0.4pt,line join=round,line cap=round] (149.13,111.14) -- (153.63,106.64);

\path[draw=drawColor,line width= 0.4pt,line join=round,line cap=round] (172.36,127.72) -- (176.86,132.22);

\path[draw=drawColor,line width= 0.4pt,line join=round,line cap=round] (172.36,132.22) -- (176.86,127.72);

\path[draw=drawColor,line width= 0.4pt,line join=round,line cap=round] (195.58,142.90) -- (200.08,147.40);

\path[draw=drawColor,line width= 0.4pt,line join=round,line cap=round] (195.58,147.40) -- (200.08,142.90);

\path[draw=drawColor,line width= 0.4pt,line join=round,line cap=round] (218.80,151.96) -- (223.30,156.46);

\path[draw=drawColor,line width= 0.4pt,line join=round,line cap=round] (218.80,156.46) -- (223.30,151.96);

\path[draw=drawColor,line width= 0.4pt,line join=round,line cap=round] (242.03,158.27) -- (246.53,162.77);

\path[draw=drawColor,line width= 0.4pt,line join=round,line cap=round] (242.03,162.77) -- (246.53,158.27);

\path[draw=drawColor,line width= 0.4pt,line join=round,line cap=round] (265.25,161.02) -- (269.75,165.52);

\path[draw=drawColor,line width= 0.4pt,line join=round,line cap=round] (265.25,165.52) -- (269.75,161.02);

\path[draw=drawColor,line width= 0.4pt,line join=round,line cap=round] (288.48,161.42) -- (292.98,165.92);

\path[draw=drawColor,line width= 0.4pt,line join=round,line cap=round] (288.48,165.92) -- (292.98,161.42);
\definecolor{drawColor}{RGB}{0,0,0}

\path[draw=drawColor,line width= 0.4pt,line join=round,line cap=round] ( 49.20,167.61) rectangle ( 95.48,107.61);

\path[draw=drawColor,line width= 0.4pt,line join=round,line cap=round] ( 58.20,155.61) circle (  2.25);
\definecolor{drawColor}{RGB}{255,0,0}

\path[draw=drawColor,line width= 0.4pt,line join=round,line cap=round] ( 58.20,147.11) --
	( 61.23,141.86) --
	( 55.17,141.86) --
	cycle;
\definecolor{drawColor}{RGB}{0,255,0}

\path[draw=drawColor,line width= 0.4pt,line join=round,line cap=round] ( 55.02,131.61) -- ( 61.38,131.61);

\path[draw=drawColor,line width= 0.4pt,line join=round,line cap=round] ( 58.20,128.43) -- ( 58.20,134.79);
\definecolor{drawColor}{RGB}{0,0,255}

\path[draw=drawColor,line width= 0.4pt,line join=round,line cap=round] ( 55.95,117.36) -- ( 60.45,121.86);

\path[draw=drawColor,line width= 0.4pt,line join=round,line cap=round] ( 55.95,121.86) -- ( 60.45,117.36);
\definecolor{drawColor}{RGB}{0,0,0}

\node[text=drawColor,anchor=base west,inner sep=0pt, outer sep=0pt, scale=  1.00] at ( 67.20,152.17) {\itshape $\mathcal I_n$};

\node[text=drawColor,anchor=base west,inner sep=0pt, outer sep=0pt, scale=  1.00] at ( 67.20,140.17) {\itshape \normalfont $\mathcal I_n^\text{DC}$};

\node[text=drawColor,anchor=base west,inner sep=0pt, outer sep=0pt, scale=  1.00] at ( 67.20,128.17) {\itshape \normalfont $\mathcal I_n^\text{BC}$};

\node[text=drawColor,anchor=base west,inner sep=0pt, outer sep=0pt, scale=  1.00] at ( 67.20,116.17) {\itshape \normalfont $\mathcal I_n^\text{CvM}$};
\end{scope}
\end{tikzpicture}

%% file: powers-max-exp_decay.tex
\begin{tikzpicture}[x=1pt,y=1pt]
\definecolor{fillColor}{RGB}{255,255,255}
\path[use as bounding box,fill=fillColor,fill opacity=0.00] (0,0) rectangle (325.21,216.81);
\begin{scope}
\path[clip] ( 49.20, 61.20) rectangle (300.01,167.61);
\definecolor{drawColor}{RGB}{0,0,0}

\path[draw=drawColor,line width= 0.4pt,line join=round,line cap=round] ( 64.49, 70.46) -- ( 75.71, 70.46);

\path[draw=drawColor,line width= 0.4pt,line join=round,line cap=round] ( 87.32, 72.60) -- ( 99.33, 77.19);

\path[draw=drawColor,line width= 0.4pt,line join=round,line cap=round] (110.25, 82.12) -- (122.85, 88.75);

\path[draw=drawColor,line width= 0.4pt,line join=round,line cap=round] (133.18, 94.83) -- (146.36,103.44);

\path[draw=drawColor,line width= 0.4pt,line join=round,line cap=round] (156.15,110.36) -- (169.84,120.81);

\path[draw=drawColor,line width= 0.4pt,line join=round,line cap=round] (178.69,128.85) -- (193.75,145.08);

\path[draw=drawColor,line width= 0.4pt,line join=round,line cap=round] (203.40,151.70) -- (215.48,156.52);

\path[draw=drawColor,line width= 0.4pt,line join=round,line cap=round] (226.93,159.94) -- (238.40,162.27);

\path[draw=drawColor,line width= 0.4pt,line join=round,line cap=round] (250.28,163.52) -- (261.50,163.62);

\path[draw=drawColor,line width= 0.4pt,line join=round,line cap=round] (273.50,163.67) -- (284.73,163.67);

\path[draw=drawColor,line width= 0.4pt,line join=round,line cap=round] ( 58.49, 70.46) circle (  2.25);

\path[draw=drawColor,line width= 0.4pt,line join=round,line cap=round] ( 81.71, 70.46) circle (  2.25);

\path[draw=drawColor,line width= 0.4pt,line join=round,line cap=round] (104.94, 79.33) circle (  2.25);

\path[draw=drawColor,line width= 0.4pt,line join=round,line cap=round] (128.16, 91.55) circle (  2.25);

\path[draw=drawColor,line width= 0.4pt,line join=round,line cap=round] (151.38,106.72) circle (  2.25);

\path[draw=drawColor,line width= 0.4pt,line join=round,line cap=round] (174.61,124.45) circle (  2.25);

\path[draw=drawColor,line width= 0.4pt,line join=round,line cap=round] (197.83,149.48) circle (  2.25);

\path[draw=drawColor,line width= 0.4pt,line join=round,line cap=round] (221.05,158.74) circle (  2.25);

\path[draw=drawColor,line width= 0.4pt,line join=round,line cap=round] (244.28,163.47) circle (  2.25);

\path[draw=drawColor,line width= 0.4pt,line join=round,line cap=round] (267.50,163.67) circle (  2.25);

\path[draw=drawColor,line width= 0.4pt,line join=round,line cap=round] (290.73,163.67) circle (  2.25);
\end{scope}
\begin{scope}
\path[clip] (  0.00,  0.00) rectangle (325.21,216.81);
\definecolor{drawColor}{RGB}{0,0,0}

\path[draw=drawColor,line width= 0.4pt,line join=round,line cap=round] ( 58.49, 61.20) -- (290.73, 61.20);

\path[draw=drawColor,line width= 0.4pt,line join=round,line cap=round] ( 58.49, 61.20) -- ( 58.49, 55.20);

\path[draw=drawColor,line width= 0.4pt,line join=round,line cap=round] (104.94, 61.20) -- (104.94, 55.20);

\path[draw=drawColor,line width= 0.4pt,line join=round,line cap=round] (151.38, 61.20) -- (151.38, 55.20);

\path[draw=drawColor,line width= 0.4pt,line join=round,line cap=round] (197.83, 61.20) -- (197.83, 55.20);

\path[draw=drawColor,line width= 0.4pt,line join=round,line cap=round] (244.28, 61.20) -- (244.28, 55.20);

\path[draw=drawColor,line width= 0.4pt,line join=round,line cap=round] (290.73, 61.20) -- (290.73, 55.20);

\node[text=drawColor,anchor=base,inner sep=0pt, outer sep=0pt, scale=  1.00] at ( 58.49, 39.60) {0.0};

\node[text=drawColor,anchor=base,inner sep=0pt, outer sep=0pt, scale=  1.00] at (104.94, 39.60) {0.2};

\node[text=drawColor,anchor=base,inner sep=0pt, outer sep=0pt, scale=  1.00] at (151.38, 39.60) {0.4};

\node[text=drawColor,anchor=base,inner sep=0pt, outer sep=0pt, scale=  1.00] at (197.83, 39.60) {0.6};

\node[text=drawColor,anchor=base,inner sep=0pt, outer sep=0pt, scale=  1.00] at (244.28, 39.60) {0.8};

\node[text=drawColor,anchor=base,inner sep=0pt, outer sep=0pt, scale=  1.00] at (290.73, 39.60) {1.0};

\path[draw=drawColor,line width= 0.4pt,line join=round,line cap=round] ( 49.20, 65.14) -- ( 49.20,163.67);

\path[draw=drawColor,line width= 0.4pt,line join=round,line cap=round] ( 49.20, 65.14) -- ( 43.20, 65.14);

\path[draw=drawColor,line width= 0.4pt,line join=round,line cap=round] ( 49.20, 84.85) -- ( 43.20, 84.85);

\path[draw=drawColor,line width= 0.4pt,line join=round,line cap=round] ( 49.20,104.55) -- ( 43.20,104.55);

\path[draw=drawColor,line width= 0.4pt,line join=round,line cap=round] ( 49.20,124.26) -- ( 43.20,124.26);

\path[draw=drawColor,line width= 0.4pt,line join=round,line cap=round] ( 49.20,143.96) -- ( 43.20,143.96);

\path[draw=drawColor,line width= 0.4pt,line join=round,line cap=round] ( 49.20,163.67) -- ( 43.20,163.67);

\node[text=drawColor,rotate= 90.00,anchor=base,inner sep=0pt, outer sep=0pt, scale=  1.00] at ( 34.80, 65.14) {0.0};

\node[text=drawColor,rotate= 90.00,anchor=base,inner sep=0pt, outer sep=0pt, scale=  1.00] at ( 34.80,104.55) {0.4};

\node[text=drawColor,rotate= 90.00,anchor=base,inner sep=0pt, outer sep=0pt, scale=  1.00] at ( 34.80,143.96) {0.8};

\path[draw=drawColor,line width= 0.4pt,line join=round,line cap=round] ( 49.20, 61.20) --
	(300.01, 61.20) --
	(300.01,167.61) --
	( 49.20,167.61) --
	cycle;
\end{scope}
\begin{scope}
\path[clip] (  0.00,  0.00) rectangle (325.21,216.81);
\definecolor{drawColor}{RGB}{0,0,0}

\node[text=drawColor,anchor=base,inner sep=0pt, outer sep=0pt, scale=  1.00] at (174.61, 15.60) {$r^2$};

\node[text=drawColor,rotate= 90.00,anchor=base,inner sep=0pt, outer sep=0pt, scale=  1.00] at ( 10.80,114.41) {Power};

\node[text=drawColor,anchor=base,inner sep=0pt, outer sep=0pt, scale=  1.20] at (174.61,188.07) {\bfseries \texttt{max}};
\end{scope}
\begin{scope}
\path[clip] ( 49.20, 61.20) rectangle (300.01,167.61);
\definecolor{drawColor}{RGB}{255,0,0}

\path[draw=drawColor,line width= 0.4pt,dash pattern=on 4pt off 4pt ,line join=round,line cap=round] ( 60.52, 76.89) -- ( 79.68,130.04);

\path[draw=drawColor,line width= 0.4pt,dash pattern=on 4pt off 4pt ,line join=round,line cap=round] ( 85.71,140.16) -- (100.94,157.22);

\path[draw=drawColor,line width= 0.4pt,dash pattern=on 4pt off 4pt ,line join=round,line cap=round] (110.92,162.16) -- (122.18,163.01);

\path[draw=drawColor,line width= 0.4pt,dash pattern=on 4pt off 4pt ,line join=round,line cap=round] (134.16,163.52) -- (145.38,163.62);

\path[draw=drawColor,line width= 0.4pt,dash pattern=on 4pt off 4pt ,line join=round,line cap=round] (157.38,163.67) -- (168.61,163.67);

\path[draw=drawColor,line width= 0.4pt,dash pattern=on 4pt off 4pt ,line join=round,line cap=round] (180.61,163.67) -- (191.83,163.67);

\path[draw=drawColor,line width= 0.4pt,dash pattern=on 4pt off 4pt ,line join=round,line cap=round] (203.83,163.67) -- (215.05,163.67);

\path[draw=drawColor,line width= 0.4pt,dash pattern=on 4pt off 4pt ,line join=round,line cap=round] (227.05,163.67) -- (238.28,163.67);

\path[draw=drawColor,line width= 0.4pt,dash pattern=on 4pt off 4pt ,line join=round,line cap=round] (250.28,163.67) -- (261.50,163.67);

\path[draw=drawColor,line width= 0.4pt,dash pattern=on 4pt off 4pt ,line join=round,line cap=round] (273.50,163.67) -- (284.73,163.67);

\path[draw=drawColor,line width= 0.4pt,line join=round,line cap=round] ( 58.49, 74.75) --
	( 61.52, 69.50) --
	( 55.46, 69.50) --
	cycle;

\path[draw=drawColor,line width= 0.4pt,line join=round,line cap=round] ( 81.71,139.19) --
	( 84.74,133.94) --
	( 78.68,133.94) --
	cycle;

\path[draw=drawColor,line width= 0.4pt,line join=round,line cap=round] (104.94,165.20) --
	(107.97,159.95) --
	(101.91,159.95) --
	cycle;

\path[draw=drawColor,line width= 0.4pt,line join=round,line cap=round] (128.16,166.97) --
	(131.19,161.72) --
	(125.13,161.72) --
	cycle;

\path[draw=drawColor,line width= 0.4pt,line join=round,line cap=round] (151.38,167.17) --
	(154.41,161.92) --
	(148.35,161.92) --
	cycle;

\path[draw=drawColor,line width= 0.4pt,line join=round,line cap=round] (174.61,167.17) --
	(177.64,161.92) --
	(171.58,161.92) --
	cycle;

\path[draw=drawColor,line width= 0.4pt,line join=round,line cap=round] (197.83,167.17) --
	(200.86,161.92) --
	(194.80,161.92) --
	cycle;

\path[draw=drawColor,line width= 0.4pt,line join=round,line cap=round] (221.05,167.17) --
	(224.08,161.92) --
	(218.02,161.92) --
	cycle;

\path[draw=drawColor,line width= 0.4pt,line join=round,line cap=round] (244.28,167.17) --
	(247.31,161.92) --
	(241.25,161.92) --
	cycle;

\path[draw=drawColor,line width= 0.4pt,line join=round,line cap=round] (267.50,167.17) --
	(270.53,161.92) --
	(264.47,161.92) --
	cycle;

\path[draw=drawColor,line width= 0.4pt,line join=round,line cap=round] (290.73,167.17) --
	(293.76,161.92) --
	(287.70,161.92) --
	cycle;
\definecolor{drawColor}{RGB}{0,0,255}

\path[draw=drawColor,line width= 0.4pt,dash pattern=on 1pt off 3pt on 4pt off 3pt ,line join=round,line cap=round] ( 64.30, 71.74) -- ( 75.90, 74.70);

\path[draw=drawColor,line width= 0.4pt,dash pattern=on 1pt off 3pt on 4pt off 3pt ,line join=round,line cap=round] ( 86.81, 79.34) -- ( 99.84, 87.40);

\path[draw=drawColor,line width= 0.4pt,dash pattern=on 1pt off 3pt on 4pt off 3pt ,line join=round,line cap=round] (108.71, 95.23) -- (124.39,114.66);

\path[draw=drawColor,line width= 0.4pt,dash pattern=on 1pt off 3pt on 4pt off 3pt ,line join=round,line cap=round] (132.35,123.63) -- (147.20,138.88);

\path[draw=drawColor,line width= 0.4pt,dash pattern=on 1pt off 3pt on 4pt off 3pt ,line join=round,line cap=round] (156.21,146.74) -- (169.78,156.76);

\path[draw=drawColor,line width= 0.4pt,dash pattern=on 1pt off 3pt on 4pt off 3pt ,line join=round,line cap=round] (180.57,161.03) -- (191.87,162.37);

\path[draw=drawColor,line width= 0.4pt,dash pattern=on 1pt off 3pt on 4pt off 3pt ,line join=round,line cap=round] (203.83,163.23) -- (215.06,163.52);

\path[draw=drawColor,line width= 0.4pt,dash pattern=on 1pt off 3pt on 4pt off 3pt ,line join=round,line cap=round] (227.05,163.67) -- (238.28,163.67);

\path[draw=drawColor,line width= 0.4pt,dash pattern=on 1pt off 3pt on 4pt off 3pt ,line join=round,line cap=round] (250.28,163.67) -- (261.50,163.67);

\path[draw=drawColor,line width= 0.4pt,dash pattern=on 1pt off 3pt on 4pt off 3pt ,line join=round,line cap=round] (273.50,163.67) -- (284.73,163.67);

\path[draw=drawColor,line width= 0.4pt,line join=round,line cap=round] ( 55.31, 70.26) -- ( 61.67, 70.26);

\path[draw=drawColor,line width= 0.4pt,line join=round,line cap=round] ( 58.49, 67.08) -- ( 58.49, 73.45);

\path[draw=drawColor,line width= 0.4pt,line join=round,line cap=round] ( 78.53, 76.18) -- ( 84.90, 76.18);

\path[draw=drawColor,line width= 0.4pt,line join=round,line cap=round] ( 81.71, 72.99) -- ( 81.71, 79.36);

\path[draw=drawColor,line width= 0.4pt,line join=round,line cap=round] (101.75, 90.56) -- (108.12, 90.56);

\path[draw=drawColor,line width= 0.4pt,line join=round,line cap=round] (104.94, 87.38) -- (104.94, 93.74);

\path[draw=drawColor,line width= 0.4pt,line join=round,line cap=round] (124.98,119.33) -- (131.34,119.33);

\path[draw=drawColor,line width= 0.4pt,line join=round,line cap=round] (128.16,116.15) -- (128.16,122.51);

\path[draw=drawColor,line width= 0.4pt,line join=round,line cap=round] (148.20,143.18) -- (154.57,143.18);

\path[draw=drawColor,line width= 0.4pt,line join=round,line cap=round] (151.38,139.99) -- (151.38,146.36);

\path[draw=drawColor,line width= 0.4pt,line join=round,line cap=round] (171.43,160.32) -- (177.79,160.32);

\path[draw=drawColor,line width= 0.4pt,line join=round,line cap=round] (174.61,157.14) -- (174.61,163.50);

\path[draw=drawColor,line width= 0.4pt,line join=round,line cap=round] (194.65,163.08) -- (201.01,163.08);

\path[draw=drawColor,line width= 0.4pt,line join=round,line cap=round] (197.83,159.90) -- (197.83,166.26);

\path[draw=drawColor,line width= 0.4pt,line join=round,line cap=round] (217.87,163.67) -- (224.24,163.67);

\path[draw=drawColor,line width= 0.4pt,line join=round,line cap=round] (221.05,160.49) -- (221.05,166.85);

\path[draw=drawColor,line width= 0.4pt,line join=round,line cap=round] (241.10,163.67) -- (247.46,163.67);

\path[draw=drawColor,line width= 0.4pt,line join=round,line cap=round] (244.28,160.49) -- (244.28,166.85);

\path[draw=drawColor,line width= 0.4pt,line join=round,line cap=round] (264.32,163.67) -- (270.68,163.67);

\path[draw=drawColor,line width= 0.4pt,line join=round,line cap=round] (267.50,160.49) -- (267.50,166.85);

\path[draw=drawColor,line width= 0.4pt,line join=round,line cap=round] (287.54,163.67) -- (293.91,163.67);

\path[draw=drawColor,line width= 0.4pt,line join=round,line cap=round] (290.73,160.49) -- (290.73,166.85);
\definecolor{drawColor}{RGB}{0,255,0}

\path[draw=drawColor,line width= 0.4pt,dash pattern=on 1pt off 3pt on 4pt off 3pt ,line join=round,line cap=round] ( 61.02, 73.73) -- ( 79.18,112.71);

\path[draw=drawColor,line width= 0.4pt,dash pattern=on 1pt off 3pt on 4pt off 3pt ,line join=round,line cap=round] ( 85.16,123.06) -- (101.49,146.34);

\path[draw=drawColor,line width= 0.4pt,dash pattern=on 1pt off 3pt on 4pt off 3pt ,line join=round,line cap=round] (110.37,153.79) -- (122.72,159.56);

\path[draw=drawColor,line width= 0.4pt,dash pattern=on 1pt off 3pt on 4pt off 3pt ,line join=round,line cap=round] (134.15,162.45) -- (145.39,163.12);

\path[draw=drawColor,line width= 0.4pt,dash pattern=on 1pt off 3pt on 4pt off 3pt ,line join=round,line cap=round] (157.38,163.52) -- (168.61,163.62);

\path[draw=drawColor,line width= 0.4pt,dash pattern=on 1pt off 3pt on 4pt off 3pt ,line join=round,line cap=round] (180.61,163.67) -- (191.83,163.67);

\path[draw=drawColor,line width= 0.4pt,dash pattern=on 1pt off 3pt on 4pt off 3pt ,line join=round,line cap=round] (203.83,163.67) -- (215.05,163.67);

\path[draw=drawColor,line width= 0.4pt,dash pattern=on 1pt off 3pt on 4pt off 3pt ,line join=round,line cap=round] (227.05,163.67) -- (238.28,163.67);

\path[draw=drawColor,line width= 0.4pt,dash pattern=on 1pt off 3pt on 4pt off 3pt ,line join=round,line cap=round] (250.28,163.67) -- (261.50,163.67);

\path[draw=drawColor,line width= 0.4pt,dash pattern=on 1pt off 3pt on 4pt off 3pt ,line join=round,line cap=round] (273.50,163.67) -- (284.73,163.67);

\path[draw=drawColor,line width= 0.4pt,line join=round,line cap=round] ( 56.24, 66.04) -- ( 60.74, 70.54);

\path[draw=drawColor,line width= 0.4pt,line join=round,line cap=round] ( 56.24, 70.54) -- ( 60.74, 66.04);

\path[draw=drawColor,line width= 0.4pt,line join=round,line cap=round] ( 79.46,115.90) -- ( 83.96,120.40);

\path[draw=drawColor,line width= 0.4pt,line join=round,line cap=round] ( 79.46,120.40) -- ( 83.96,115.90);

\path[draw=drawColor,line width= 0.4pt,line join=round,line cap=round] (102.69,149.00) -- (107.19,153.50);

\path[draw=drawColor,line width= 0.4pt,line join=round,line cap=round] (102.69,153.50) -- (107.19,149.00);

\path[draw=drawColor,line width= 0.4pt,line join=round,line cap=round] (125.91,159.84) -- (130.41,164.34);

\path[draw=drawColor,line width= 0.4pt,line join=round,line cap=round] (125.91,164.34) -- (130.41,159.84);

\path[draw=drawColor,line width= 0.4pt,line join=round,line cap=round] (149.13,161.22) -- (153.63,165.72);

\path[draw=drawColor,line width= 0.4pt,line join=round,line cap=round] (149.13,165.72) -- (153.63,161.22);

\path[draw=drawColor,line width= 0.4pt,line join=round,line cap=round] (172.36,161.42) -- (176.86,165.92);

\path[draw=drawColor,line width= 0.4pt,line join=round,line cap=round] (172.36,165.92) -- (176.86,161.42);

\path[draw=drawColor,line width= 0.4pt,line join=round,line cap=round] (195.58,161.42) -- (200.08,165.92);

\path[draw=drawColor,line width= 0.4pt,line join=round,line cap=round] (195.58,165.92) -- (200.08,161.42);

\path[draw=drawColor,line width= 0.4pt,line join=round,line cap=round] (218.80,161.42) -- (223.30,165.92);

\path[draw=drawColor,line width= 0.4pt,line join=round,line cap=round] (218.80,165.92) -- (223.30,161.42);

\path[draw=drawColor,line width= 0.4pt,line join=round,line cap=round] (242.03,161.42) -- (246.53,165.92);

\path[draw=drawColor,line width= 0.4pt,line join=round,line cap=round] (242.03,165.92) -- (246.53,161.42);

\path[draw=drawColor,line width= 0.4pt,line join=round,line cap=round] (265.25,161.42) -- (269.75,165.92);

\path[draw=drawColor,line width= 0.4pt,line join=round,line cap=round] (265.25,165.92) -- (269.75,161.42);

\path[draw=drawColor,line width= 0.4pt,line join=round,line cap=round] (288.48,161.42) -- (292.98,165.92);

\path[draw=drawColor,line width= 0.4pt,line join=round,line cap=round] (288.48,165.92) -- (292.98,161.42);
\definecolor{drawColor}{RGB}{0,0,0}

\path[draw=drawColor,line width= 0.4pt,line join=round,line cap=round] (253.73,121.20) rectangle (300.01, 61.20);

\path[draw=drawColor,line width= 0.4pt,line join=round,line cap=round] (262.73,109.20) circle (  2.25);
\definecolor{drawColor}{RGB}{255,0,0}

\path[draw=drawColor,line width= 0.4pt,line join=round,line cap=round] (262.73,100.70) --
	(265.76, 95.45) --
	(259.70, 95.45) --
	cycle;
\definecolor{drawColor}{RGB}{0,255,0}

\path[draw=drawColor,line width= 0.4pt,line join=round,line cap=round] (259.55, 85.20) -- (265.91, 85.20);

\path[draw=drawColor,line width= 0.4pt,line join=round,line cap=round] (262.73, 82.02) -- (262.73, 88.38);
\definecolor{drawColor}{RGB}{0,0,255}

\path[draw=drawColor,line width= 0.4pt,line join=round,line cap=round] (260.48, 70.95) -- (264.98, 75.45);

\path[draw=drawColor,line width= 0.4pt,line join=round,line cap=round] (260.48, 75.45) -- (264.98, 70.95);
\definecolor{drawColor}{RGB}{0,0,0}

\node[text=drawColor,anchor=base west,inner sep=0pt, outer sep=0pt, scale=  1.00] at (271.73,105.76) {\itshape $\mathcal I_n$};

\node[text=drawColor,anchor=base west,inner sep=0pt, outer sep=0pt, scale=  1.00] at (271.73, 93.76) {\itshape \normalfont $\mathcal I_n^\text{DC}$};

\node[text=drawColor,anchor=base west,inner sep=0pt, outer sep=0pt, scale=  1.00] at (271.73, 81.76) {\itshape \normalfont $\mathcal I_n^\text{BC}$};

\node[text=drawColor,anchor=base west,inner sep=0pt, outer sep=0pt, scale=  1.00] at (271.73, 69.76) {\itshape \normalfont $\mathcal I_n^\text{CvM}$};
\end{scope}
\end{tikzpicture}

%% file: powers-range-exp_decay.tex
\begin{tikzpicture}[x=1pt,y=1pt]
\definecolor{fillColor}{RGB}{255,255,255}
\path[use as bounding box,fill=fillColor,fill opacity=0.00] (0,0) rectangle (325.21,216.81);
\begin{scope}
\path[clip] ( 49.20, 61.20) rectangle (300.01,167.61);
\definecolor{drawColor}{RGB}{0,0,0}

\path[draw=drawColor,line width= 0.4pt,line join=round,line cap=round] ( 64.49, 70.36) -- ( 75.71, 70.17);

\path[draw=drawColor,line width= 0.4pt,line join=round,line cap=round] ( 87.71, 69.81) -- ( 98.94, 69.34);

\path[draw=drawColor,line width= 0.4pt,line join=round,line cap=round] (110.76, 70.52) -- (122.33, 73.36);

\path[draw=drawColor,line width= 0.4pt,line join=round,line cap=round] (133.96, 76.32) -- (145.58, 79.38);

\path[draw=drawColor,line width= 0.4pt,line join=round,line cap=round] (156.71, 83.66) -- (169.28, 90.17);

\path[draw=drawColor,line width= 0.4pt,line join=round,line cap=round] (179.94, 95.68) -- (192.50,102.19);

\path[draw=drawColor,line width= 0.4pt,line join=round,line cap=round] (202.18,109.08) -- (216.71,122.88);

\path[draw=drawColor,line width= 0.4pt,line join=round,line cap=round] (225.80,130.68) -- (239.53,141.28);

\path[draw=drawColor,line width= 0.4pt,line join=round,line cap=round] (249.34,148.17) -- (262.44,156.51);

\path[draw=drawColor,line width= 0.4pt,line join=round,line cap=round] (273.43,160.68) -- (284.80,162.52);

\path[draw=drawColor,line width= 0.4pt,line join=round,line cap=round] ( 58.49, 70.46) circle (  2.25);

\path[draw=drawColor,line width= 0.4pt,line join=round,line cap=round] ( 81.71, 70.07) circle (  2.25);

\path[draw=drawColor,line width= 0.4pt,line join=round,line cap=round] (104.94, 69.08) circle (  2.25);

\path[draw=drawColor,line width= 0.4pt,line join=round,line cap=round] (128.16, 74.80) circle (  2.25);

\path[draw=drawColor,line width= 0.4pt,line join=round,line cap=round] (151.38, 80.91) circle (  2.25);

\path[draw=drawColor,line width= 0.4pt,line join=round,line cap=round] (174.61, 92.93) circle (  2.25);

\path[draw=drawColor,line width= 0.4pt,line join=round,line cap=round] (197.83,104.95) circle (  2.25);

\path[draw=drawColor,line width= 0.4pt,line join=round,line cap=round] (221.05,127.02) circle (  2.25);

\path[draw=drawColor,line width= 0.4pt,line join=round,line cap=round] (244.28,144.95) circle (  2.25);

\path[draw=drawColor,line width= 0.4pt,line join=round,line cap=round] (267.50,159.73) circle (  2.25);

\path[draw=drawColor,line width= 0.4pt,line join=round,line cap=round] (290.73,163.47) circle (  2.25);
\end{scope}
\begin{scope}
\path[clip] (  0.00,  0.00) rectangle (325.21,216.81);
\definecolor{drawColor}{RGB}{0,0,0}

\path[draw=drawColor,line width= 0.4pt,line join=round,line cap=round] ( 58.49, 61.20) -- (290.73, 61.20);

\path[draw=drawColor,line width= 0.4pt,line join=round,line cap=round] ( 58.49, 61.20) -- ( 58.49, 55.20);

\path[draw=drawColor,line width= 0.4pt,line join=round,line cap=round] (104.94, 61.20) -- (104.94, 55.20);

\path[draw=drawColor,line width= 0.4pt,line join=round,line cap=round] (151.38, 61.20) -- (151.38, 55.20);

\path[draw=drawColor,line width= 0.4pt,line join=round,line cap=round] (197.83, 61.20) -- (197.83, 55.20);

\path[draw=drawColor,line width= 0.4pt,line join=round,line cap=round] (244.28, 61.20) -- (244.28, 55.20);

\path[draw=drawColor,line width= 0.4pt,line join=round,line cap=round] (290.73, 61.20) -- (290.73, 55.20);

\node[text=drawColor,anchor=base,inner sep=0pt, outer sep=0pt, scale=  1.00] at ( 58.49, 39.60) {0.0};

\node[text=drawColor,anchor=base,inner sep=0pt, outer sep=0pt, scale=  1.00] at (104.94, 39.60) {0.2};

\node[text=drawColor,anchor=base,inner sep=0pt, outer sep=0pt, scale=  1.00] at (151.38, 39.60) {0.4};

\node[text=drawColor,anchor=base,inner sep=0pt, outer sep=0pt, scale=  1.00] at (197.83, 39.60) {0.6};

\node[text=drawColor,anchor=base,inner sep=0pt, outer sep=0pt, scale=  1.00] at (244.28, 39.60) {0.8};

\node[text=drawColor,anchor=base,inner sep=0pt, outer sep=0pt, scale=  1.00] at (290.73, 39.60) {1.0};

\path[draw=drawColor,line width= 0.4pt,line join=round,line cap=round] ( 49.20, 65.14) -- ( 49.20,163.67);

\path[draw=drawColor,line width= 0.4pt,line join=round,line cap=round] ( 49.20, 65.14) -- ( 43.20, 65.14);

\path[draw=drawColor,line width= 0.4pt,line join=round,line cap=round] ( 49.20, 84.85) -- ( 43.20, 84.85);

\path[draw=drawColor,line width= 0.4pt,line join=round,line cap=round] ( 49.20,104.55) -- ( 43.20,104.55);

\path[draw=drawColor,line width= 0.4pt,line join=round,line cap=round] ( 49.20,124.26) -- ( 43.20,124.26);

\path[draw=drawColor,line width= 0.4pt,line join=round,line cap=round] ( 49.20,143.96) -- ( 43.20,143.96);

\path[draw=drawColor,line width= 0.4pt,line join=round,line cap=round] ( 49.20,163.67) -- ( 43.20,163.67);

\node[text=drawColor,rotate= 90.00,anchor=base,inner sep=0pt, outer sep=0pt, scale=  1.00] at ( 34.80, 65.14) {0.0};

\node[text=drawColor,rotate= 90.00,anchor=base,inner sep=0pt, outer sep=0pt, scale=  1.00] at ( 34.80,104.55) {0.4};

\node[text=drawColor,rotate= 90.00,anchor=base,inner sep=0pt, outer sep=0pt, scale=  1.00] at ( 34.80,143.96) {0.8};

\path[draw=drawColor,line width= 0.4pt,line join=round,line cap=round] ( 49.20, 61.20) --
	(300.01, 61.20) --
	(300.01,167.61) --
	( 49.20,167.61) --
	cycle;
\end{scope}
\begin{scope}
\path[clip] (  0.00,  0.00) rectangle (325.21,216.81);
\definecolor{drawColor}{RGB}{0,0,0}

\node[text=drawColor,anchor=base,inner sep=0pt, outer sep=0pt, scale=  1.00] at (174.61, 15.60) {$r^2$};

\node[text=drawColor,rotate= 90.00,anchor=base,inner sep=0pt, outer sep=0pt, scale=  1.00] at ( 10.80,114.41) {Power};

\node[text=drawColor,anchor=base,inner sep=0pt, outer sep=0pt, scale=  1.20] at (174.61,188.07) {\bfseries \texttt{range}};
\end{scope}
\begin{scope}
\path[clip] ( 49.20, 61.20) rectangle (300.01,167.61);
\definecolor{drawColor}{RGB}{255,0,0}

\path[draw=drawColor,line width= 0.4pt,dash pattern=on 4pt off 4pt ,line join=round,line cap=round] ( 64.17, 73.18) -- ( 76.03, 77.20);

\path[draw=drawColor,line width= 0.4pt,dash pattern=on 4pt off 4pt ,line join=round,line cap=round] ( 86.29, 83.01) -- (100.36, 94.96);

\path[draw=drawColor,line width= 0.4pt,dash pattern=on 4pt off 4pt ,line join=round,line cap=round] (108.93,103.31) -- (124.16,120.37);

\path[draw=drawColor,line width= 0.4pt,dash pattern=on 4pt off 4pt ,line join=round,line cap=round] (132.81,128.64) -- (146.73,139.98);

\path[draw=drawColor,line width= 0.4pt,dash pattern=on 4pt off 4pt ,line join=round,line cap=round] (156.62,146.70) -- (169.37,153.84);

\path[draw=drawColor,line width= 0.4pt,dash pattern=on 4pt off 4pt ,line join=round,line cap=round] (180.49,157.97) -- (191.95,160.30);

\path[draw=drawColor,line width= 0.4pt,dash pattern=on 4pt off 4pt ,line join=round,line cap=round] (203.81,162.01) -- (215.08,162.96);

\path[draw=drawColor,line width= 0.4pt,dash pattern=on 4pt off 4pt ,line join=round,line cap=round] (227.05,163.52) -- (238.28,163.62);

\path[draw=drawColor,line width= 0.4pt,dash pattern=on 4pt off 4pt ,line join=round,line cap=round] (250.28,163.67) -- (261.50,163.67);

\path[draw=drawColor,line width= 0.4pt,dash pattern=on 4pt off 4pt ,line join=round,line cap=round] (273.50,163.67) -- (284.73,163.67);

\path[draw=drawColor,line width= 0.4pt,line join=round,line cap=round] ( 58.49, 74.75) --
	( 61.52, 69.50) --
	( 55.46, 69.50) --
	cycle;

\path[draw=drawColor,line width= 0.4pt,line join=round,line cap=round] ( 81.71, 82.63) --
	( 84.74, 77.38) --
	( 78.68, 77.38) --
	cycle;

\path[draw=drawColor,line width= 0.4pt,line join=round,line cap=round] (104.94,102.34) --
	(107.97, 97.09) --
	(101.91, 97.09) --
	cycle;

\path[draw=drawColor,line width= 0.4pt,line join=round,line cap=round] (128.16,128.35) --
	(131.19,123.10) --
	(125.13,123.10) --
	cycle;

\path[draw=drawColor,line width= 0.4pt,line join=round,line cap=round] (151.38,147.27) --
	(154.41,142.02) --
	(148.35,142.02) --
	cycle;

\path[draw=drawColor,line width= 0.4pt,line join=round,line cap=round] (174.61,160.27) --
	(177.64,155.02) --
	(171.58,155.02) --
	cycle;

\path[draw=drawColor,line width= 0.4pt,line join=round,line cap=round] (197.83,165.00) --
	(200.86,159.75) --
	(194.80,159.75) --
	cycle;

\path[draw=drawColor,line width= 0.4pt,line join=round,line cap=round] (221.05,166.97) --
	(224.08,161.72) --
	(218.02,161.72) --
	cycle;

\path[draw=drawColor,line width= 0.4pt,line join=round,line cap=round] (244.28,167.17) --
	(247.31,161.92) --
	(241.25,161.92) --
	cycle;

\path[draw=drawColor,line width= 0.4pt,line join=round,line cap=round] (267.50,167.17) --
	(270.53,161.92) --
	(264.47,161.92) --
	cycle;

\path[draw=drawColor,line width= 0.4pt,line join=round,line cap=round] (290.73,167.17) --
	(293.76,161.92) --
	(287.70,161.92) --
	cycle;
\definecolor{drawColor}{RGB}{0,0,255}

\path[draw=drawColor,line width= 0.4pt,dash pattern=on 1pt off 3pt on 4pt off 3pt ,line join=round,line cap=round] ( 64.49, 70.26) -- ( 75.71, 70.26);

\path[draw=drawColor,line width= 0.4pt,dash pattern=on 1pt off 3pt on 4pt off 3pt ,line join=round,line cap=round] ( 87.71, 70.11) -- ( 98.94, 69.83);

\path[draw=drawColor,line width= 0.4pt,dash pattern=on 1pt off 3pt on 4pt off 3pt ,line join=round,line cap=round] (110.91, 70.28) -- (122.19, 71.43);

\path[draw=drawColor,line width= 0.4pt,dash pattern=on 1pt off 3pt on 4pt off 3pt ,line join=round,line cap=round] (134.16, 72.19) -- (145.39, 72.48);

\path[draw=drawColor,line width= 0.4pt,dash pattern=on 1pt off 3pt on 4pt off 3pt ,line join=round,line cap=round] (157.38, 72.73) -- (168.61, 72.92);

\path[draw=drawColor,line width= 0.4pt,dash pattern=on 1pt off 3pt on 4pt off 3pt ,line join=round,line cap=round] (180.59, 73.53) -- (191.85, 74.49);

\path[draw=drawColor,line width= 0.4pt,dash pattern=on 1pt off 3pt on 4pt off 3pt ,line join=round,line cap=round] (203.83, 75.20) -- (215.06, 75.58);

\path[draw=drawColor,line width= 0.4pt,dash pattern=on 1pt off 3pt on 4pt off 3pt ,line join=round,line cap=round] (227.05, 76.09) -- (238.29, 76.66);

\path[draw=drawColor,line width= 0.4pt,dash pattern=on 1pt off 3pt on 4pt off 3pt ,line join=round,line cap=round] (250.20, 77.92) -- (261.58, 79.75);

\path[draw=drawColor,line width= 0.4pt,dash pattern=on 1pt off 3pt on 4pt off 3pt ,line join=round,line cap=round] (273.50, 80.50) -- (284.73, 80.12);

\path[draw=drawColor,line width= 0.4pt,line join=round,line cap=round] ( 55.31, 70.26) -- ( 61.67, 70.26);

\path[draw=drawColor,line width= 0.4pt,line join=round,line cap=round] ( 58.49, 67.08) -- ( 58.49, 73.45);

\path[draw=drawColor,line width= 0.4pt,line join=round,line cap=round] ( 78.53, 70.26) -- ( 84.90, 70.26);

\path[draw=drawColor,line width= 0.4pt,line join=round,line cap=round] ( 81.71, 67.08) -- ( 81.71, 73.45);

\path[draw=drawColor,line width= 0.4pt,line join=round,line cap=round] (101.75, 69.67) -- (108.12, 69.67);

\path[draw=drawColor,line width= 0.4pt,line join=round,line cap=round] (104.94, 66.49) -- (104.94, 72.86);

\path[draw=drawColor,line width= 0.4pt,line join=round,line cap=round] (124.98, 72.04) -- (131.34, 72.04);

\path[draw=drawColor,line width= 0.4pt,line join=round,line cap=round] (128.16, 68.86) -- (128.16, 75.22);

\path[draw=drawColor,line width= 0.4pt,line join=round,line cap=round] (148.20, 72.63) -- (154.57, 72.63);

\path[draw=drawColor,line width= 0.4pt,line join=round,line cap=round] (151.38, 69.45) -- (151.38, 75.81);

\path[draw=drawColor,line width= 0.4pt,line join=round,line cap=round] (171.43, 73.02) -- (177.79, 73.02);

\path[draw=drawColor,line width= 0.4pt,line join=round,line cap=round] (174.61, 69.84) -- (174.61, 76.21);

\path[draw=drawColor,line width= 0.4pt,line join=round,line cap=round] (194.65, 74.99) -- (201.01, 74.99);

\path[draw=drawColor,line width= 0.4pt,line join=round,line cap=round] (197.83, 71.81) -- (197.83, 78.18);

\path[draw=drawColor,line width= 0.4pt,line join=round,line cap=round] (217.87, 75.78) -- (224.24, 75.78);

\path[draw=drawColor,line width= 0.4pt,line join=round,line cap=round] (221.05, 72.60) -- (221.05, 78.96);

\path[draw=drawColor,line width= 0.4pt,line join=round,line cap=round] (241.10, 76.96) -- (247.46, 76.96);

\path[draw=drawColor,line width= 0.4pt,line join=round,line cap=round] (244.28, 73.78) -- (244.28, 80.15);

\path[draw=drawColor,line width= 0.4pt,line join=round,line cap=round] (264.32, 80.71) -- (270.68, 80.71);

\path[draw=drawColor,line width= 0.4pt,line join=round,line cap=round] (267.50, 77.53) -- (267.50, 83.89);

\path[draw=drawColor,line width= 0.4pt,line join=round,line cap=round] (287.54, 79.92) -- (293.91, 79.92);

\path[draw=drawColor,line width= 0.4pt,line join=round,line cap=round] (290.73, 76.74) -- (290.73, 83.10);
\definecolor{drawColor}{RGB}{0,255,0}

\path[draw=drawColor,line width= 0.4pt,dash pattern=on 1pt off 3pt on 4pt off 3pt ,line join=round,line cap=round] ( 62.21, 73.00) -- ( 77.99, 92.95);

\path[draw=drawColor,line width= 0.4pt,dash pattern=on 1pt off 3pt on 4pt off 3pt ,line join=round,line cap=round] ( 85.47,102.34) -- (101.18,121.94);

\path[draw=drawColor,line width= 0.4pt,dash pattern=on 1pt off 3pt on 4pt off 3pt ,line join=round,line cap=round] (109.29,130.76) -- (123.81,144.56);

\path[draw=drawColor,line width= 0.4pt,dash pattern=on 1pt off 3pt on 4pt off 3pt ,line join=round,line cap=round] (133.70,151.00) -- (145.84,156.04);

\path[draw=drawColor,line width= 0.4pt,dash pattern=on 1pt off 3pt on 4pt off 3pt ,line join=round,line cap=round] (157.27,159.50) -- (168.72,161.73);

\path[draw=drawColor,line width= 0.4pt,dash pattern=on 1pt off 3pt on 4pt off 3pt ,line join=round,line cap=round] (180.60,163.08) -- (191.83,163.47);

\path[draw=drawColor,line width= 0.4pt,dash pattern=on 1pt off 3pt on 4pt off 3pt ,line join=round,line cap=round] (203.83,163.67) -- (215.05,163.67);

\path[draw=drawColor,line width= 0.4pt,dash pattern=on 1pt off 3pt on 4pt off 3pt ,line join=round,line cap=round] (227.05,163.67) -- (238.28,163.67);

\path[draw=drawColor,line width= 0.4pt,dash pattern=on 1pt off 3pt on 4pt off 3pt ,line join=round,line cap=round] (250.28,163.67) -- (261.50,163.67);

\path[draw=drawColor,line width= 0.4pt,dash pattern=on 1pt off 3pt on 4pt off 3pt ,line join=round,line cap=round] (273.50,163.67) -- (284.73,163.67);

\path[draw=drawColor,line width= 0.4pt,line join=round,line cap=round] ( 56.24, 66.04) -- ( 60.74, 70.54);

\path[draw=drawColor,line width= 0.4pt,line join=round,line cap=round] ( 56.24, 70.54) -- ( 60.74, 66.04);

\path[draw=drawColor,line width= 0.4pt,line join=round,line cap=round] ( 79.46, 95.41) -- ( 83.96, 99.91);

\path[draw=drawColor,line width= 0.4pt,line join=round,line cap=round] ( 79.46, 99.91) -- ( 83.96, 95.41);

\path[draw=drawColor,line width= 0.4pt,line join=round,line cap=round] (102.69,124.37) -- (107.19,128.87);

\path[draw=drawColor,line width= 0.4pt,line join=round,line cap=round] (102.69,128.87) -- (107.19,124.37);

\path[draw=drawColor,line width= 0.4pt,line join=round,line cap=round] (125.91,146.44) -- (130.41,150.94);

\path[draw=drawColor,line width= 0.4pt,line join=round,line cap=round] (125.91,150.94) -- (130.41,146.44);

\path[draw=drawColor,line width= 0.4pt,line join=round,line cap=round] (149.13,156.10) -- (153.63,160.60);

\path[draw=drawColor,line width= 0.4pt,line join=round,line cap=round] (149.13,160.60) -- (153.63,156.10);

\path[draw=drawColor,line width= 0.4pt,line join=round,line cap=round] (172.36,160.63) -- (176.86,165.13);

\path[draw=drawColor,line width= 0.4pt,line join=round,line cap=round] (172.36,165.13) -- (176.86,160.63);

\path[draw=drawColor,line width= 0.4pt,line join=round,line cap=round] (195.58,161.42) -- (200.08,165.92);

\path[draw=drawColor,line width= 0.4pt,line join=round,line cap=round] (195.58,165.92) -- (200.08,161.42);

\path[draw=drawColor,line width= 0.4pt,line join=round,line cap=round] (218.80,161.42) -- (223.30,165.92);

\path[draw=drawColor,line width= 0.4pt,line join=round,line cap=round] (218.80,165.92) -- (223.30,161.42);

\path[draw=drawColor,line width= 0.4pt,line join=round,line cap=round] (242.03,161.42) -- (246.53,165.92);

\path[draw=drawColor,line width= 0.4pt,line join=round,line cap=round] (242.03,165.92) -- (246.53,161.42);

\path[draw=drawColor,line width= 0.4pt,line join=round,line cap=round] (265.25,161.42) -- (269.75,165.92);

\path[draw=drawColor,line width= 0.4pt,line join=round,line cap=round] (265.25,165.92) -- (269.75,161.42);

\path[draw=drawColor,line width= 0.4pt,line join=round,line cap=round] (288.48,161.42) -- (292.98,165.92);

\path[draw=drawColor,line width= 0.4pt,line join=round,line cap=round] (288.48,165.92) -- (292.98,161.42);
\definecolor{drawColor}{RGB}{0,0,0}

\path[draw=drawColor,line width= 0.4pt,line join=round,line cap=round] ( 49.20,167.61) rectangle ( 95.48,107.61);

\path[draw=drawColor,line width= 0.4pt,line join=round,line cap=round] ( 58.20,155.61) circle (  2.25);
\definecolor{drawColor}{RGB}{255,0,0}

\path[draw=drawColor,line width= 0.4pt,line join=round,line cap=round] ( 58.20,147.11) --
	( 61.23,141.86) --
	( 55.17,141.86) --
	cycle;
\definecolor{drawColor}{RGB}{0,255,0}

\path[draw=drawColor,line width= 0.4pt,line join=round,line cap=round] ( 55.02,131.61) -- ( 61.38,131.61);

\path[draw=drawColor,line width= 0.4pt,line join=round,line cap=round] ( 58.20,128.43) -- ( 58.20,134.79);
\definecolor{drawColor}{RGB}{0,0,255}

\path[draw=drawColor,line width= 0.4pt,line join=round,line cap=round] ( 55.95,117.36) -- ( 60.45,121.86);

\path[draw=drawColor,line width= 0.4pt,line join=round,line cap=round] ( 55.95,121.86) -- ( 60.45,117.36);
\definecolor{drawColor}{RGB}{0,0,0}

\node[text=drawColor,anchor=base west,inner sep=0pt, outer sep=0pt, scale=  1.00] at ( 67.20,152.17) {\itshape $\mathcal I_n$};

\node[text=drawColor,anchor=base west,inner sep=0pt, outer sep=0pt, scale=  1.00] at ( 67.20,140.17) {\itshape \normalfont $\mathcal I_n^\text{DC}$};

\node[text=drawColor,anchor=base west,inner sep=0pt, outer sep=0pt, scale=  1.00] at ( 67.20,128.17) {\itshape \normalfont $\mathcal I_n^\text{BC}$};

\node[text=drawColor,anchor=base west,inner sep=0pt, outer sep=0pt, scale=  1.00] at ( 67.20,116.17) {\itshape \normalfont $\mathcal I_n^\text{CvM}$};
\end{scope}
\end{tikzpicture}

%% file: powers-eval-exp_decay.tex
\begin{tikzpicture}[x=1pt,y=1pt]
\definecolor{fillColor}{RGB}{255,255,255}
\path[use as bounding box,fill=fillColor,fill opacity=0.00] (0,0) rectangle (325.21,216.81);
\begin{scope}
\path[clip] ( 49.20, 61.20) rectangle (300.01,167.61);
\definecolor{drawColor}{RGB}{0,0,0}

\path[draw=drawColor,line width= 0.4pt,line join=round,line cap=round] ( 64.24, 72.17) -- ( 75.96, 75.65);

\path[draw=drawColor,line width= 0.4pt,line join=round,line cap=round] ( 87.15, 79.90) -- ( 99.50, 85.66);

\path[draw=drawColor,line width= 0.4pt,line join=round,line cap=round] (109.25, 92.37) -- (123.85,106.49);

\path[draw=drawColor,line width= 0.4pt,line join=round,line cap=round] (132.83,114.43) -- (146.71,125.62);

\path[draw=drawColor,line width= 0.4pt,line join=round,line cap=round] (155.86,133.37) -- (170.13,146.08);

\path[draw=drawColor,line width= 0.4pt,line join=round,line cap=round] (180.13,152.42) -- (192.31,157.58);

\path[draw=drawColor,line width= 0.4pt,line join=round,line cap=round] (203.76,160.83) -- (215.12,162.57);

\path[draw=drawColor,line width= 0.4pt,line join=round,line cap=round] (227.05,163.52) -- (238.28,163.62);

\path[draw=drawColor,line width= 0.4pt,line join=round,line cap=round] (250.28,163.67) -- (261.50,163.67);

\path[draw=drawColor,line width= 0.4pt,line join=round,line cap=round] (273.50,163.67) -- (284.73,163.67);

\path[draw=drawColor,line width= 0.4pt,line join=round,line cap=round] ( 58.49, 70.46) circle (  2.25);

\path[draw=drawColor,line width= 0.4pt,line join=round,line cap=round] ( 81.71, 77.36) circle (  2.25);

\path[draw=drawColor,line width= 0.4pt,line join=round,line cap=round] (104.94, 88.20) circle (  2.25);

\path[draw=drawColor,line width= 0.4pt,line join=round,line cap=round] (128.16,110.66) circle (  2.25);

\path[draw=drawColor,line width= 0.4pt,line join=round,line cap=round] (151.38,129.38) circle (  2.25);

\path[draw=drawColor,line width= 0.4pt,line join=round,line cap=round] (174.61,150.07) circle (  2.25);

\path[draw=drawColor,line width= 0.4pt,line join=round,line cap=round] (197.83,159.92) circle (  2.25);

\path[draw=drawColor,line width= 0.4pt,line join=round,line cap=round] (221.05,163.47) circle (  2.25);

\path[draw=drawColor,line width= 0.4pt,line join=round,line cap=round] (244.28,163.67) circle (  2.25);

\path[draw=drawColor,line width= 0.4pt,line join=round,line cap=round] (267.50,163.67) circle (  2.25);

\path[draw=drawColor,line width= 0.4pt,line join=round,line cap=round] (290.73,163.67) circle (  2.25);
\end{scope}
\begin{scope}
\path[clip] (  0.00,  0.00) rectangle (325.21,216.81);
\definecolor{drawColor}{RGB}{0,0,0}

\path[draw=drawColor,line width= 0.4pt,line join=round,line cap=round] ( 58.49, 61.20) -- (290.73, 61.20);

\path[draw=drawColor,line width= 0.4pt,line join=round,line cap=round] ( 58.49, 61.20) -- ( 58.49, 55.20);

\path[draw=drawColor,line width= 0.4pt,line join=round,line cap=round] (104.94, 61.20) -- (104.94, 55.20);

\path[draw=drawColor,line width= 0.4pt,line join=round,line cap=round] (151.38, 61.20) -- (151.38, 55.20);

\path[draw=drawColor,line width= 0.4pt,line join=round,line cap=round] (197.83, 61.20) -- (197.83, 55.20);

\path[draw=drawColor,line width= 0.4pt,line join=round,line cap=round] (244.28, 61.20) -- (244.28, 55.20);

\path[draw=drawColor,line width= 0.4pt,line join=round,line cap=round] (290.73, 61.20) -- (290.73, 55.20);

\node[text=drawColor,anchor=base,inner sep=0pt, outer sep=0pt, scale=  1.00] at ( 58.49, 39.60) {0.0};

\node[text=drawColor,anchor=base,inner sep=0pt, outer sep=0pt, scale=  1.00] at (104.94, 39.60) {0.2};

\node[text=drawColor,anchor=base,inner sep=0pt, outer sep=0pt, scale=  1.00] at (151.38, 39.60) {0.4};

\node[text=drawColor,anchor=base,inner sep=0pt, outer sep=0pt, scale=  1.00] at (197.83, 39.60) {0.6};

\node[text=drawColor,anchor=base,inner sep=0pt, outer sep=0pt, scale=  1.00] at (244.28, 39.60) {0.8};

\node[text=drawColor,anchor=base,inner sep=0pt, outer sep=0pt, scale=  1.00] at (290.73, 39.60) {1.0};

\path[draw=drawColor,line width= 0.4pt,line join=round,line cap=round] ( 49.20, 65.14) -- ( 49.20,163.67);

\path[draw=drawColor,line width= 0.4pt,line join=round,line cap=round] ( 49.20, 65.14) -- ( 43.20, 65.14);

\path[draw=drawColor,line width= 0.4pt,line join=round,line cap=round] ( 49.20, 84.85) -- ( 43.20, 84.85);

\path[draw=drawColor,line width= 0.4pt,line join=round,line cap=round] ( 49.20,104.55) -- ( 43.20,104.55);

\path[draw=drawColor,line width= 0.4pt,line join=round,line cap=round] ( 49.20,124.26) -- ( 43.20,124.26);

\path[draw=drawColor,line width= 0.4pt,line join=round,line cap=round] ( 49.20,143.96) -- ( 43.20,143.96);

\path[draw=drawColor,line width= 0.4pt,line join=round,line cap=round] ( 49.20,163.67) -- ( 43.20,163.67);

\node[text=drawColor,rotate= 90.00,anchor=base,inner sep=0pt, outer sep=0pt, scale=  1.00] at ( 34.80, 65.14) {0.0};

\node[text=drawColor,rotate= 90.00,anchor=base,inner sep=0pt, outer sep=0pt, scale=  1.00] at ( 34.80,104.55) {0.4};

\node[text=drawColor,rotate= 90.00,anchor=base,inner sep=0pt, outer sep=0pt, scale=  1.00] at ( 34.80,143.96) {0.8};

\path[draw=drawColor,line width= 0.4pt,line join=round,line cap=round] ( 49.20, 61.20) --
	(300.01, 61.20) --
	(300.01,167.61) --
	( 49.20,167.61) --
	cycle;
\end{scope}
\begin{scope}
\path[clip] (  0.00,  0.00) rectangle (325.21,216.81);
\definecolor{drawColor}{RGB}{0,0,0}

\node[text=drawColor,anchor=base,inner sep=0pt, outer sep=0pt, scale=  1.00] at (174.61, 15.60) {$r^2$};

\node[text=drawColor,rotate= 90.00,anchor=base,inner sep=0pt, outer sep=0pt, scale=  1.00] at ( 10.80,114.41) {Power};

\node[text=drawColor,anchor=base,inner sep=0pt, outer sep=0pt, scale=  1.20] at (174.61,188.07) {\bfseries \texttt{eval}};
\end{scope}
\begin{scope}
\path[clip] ( 49.20, 61.20) rectangle (300.01,167.61);
\definecolor{drawColor}{RGB}{255,0,0}

\path[draw=drawColor,line width= 0.4pt,dash pattern=on 4pt off 4pt ,line join=round,line cap=round] ( 60.39, 76.94) -- ( 79.81,135.12);

\path[draw=drawColor,line width= 0.4pt,dash pattern=on 4pt off 4pt ,line join=round,line cap=round] ( 86.16,144.84) -- (100.49,157.86);

\path[draw=drawColor,line width= 0.4pt,dash pattern=on 4pt off 4pt ,line join=round,line cap=round] (110.92,162.35) -- (122.18,163.21);

\path[draw=drawColor,line width= 0.4pt,dash pattern=on 4pt off 4pt ,line join=round,line cap=round] (134.16,163.67) -- (145.38,163.67);

\path[draw=drawColor,line width= 0.4pt,dash pattern=on 4pt off 4pt ,line join=round,line cap=round] (157.38,163.67) -- (168.61,163.67);

\path[draw=drawColor,line width= 0.4pt,dash pattern=on 4pt off 4pt ,line join=round,line cap=round] (180.61,163.67) -- (191.83,163.67);

\path[draw=drawColor,line width= 0.4pt,dash pattern=on 4pt off 4pt ,line join=round,line cap=round] (203.83,163.67) -- (215.05,163.67);

\path[draw=drawColor,line width= 0.4pt,dash pattern=on 4pt off 4pt ,line join=round,line cap=round] (227.05,163.67) -- (238.28,163.67);

\path[draw=drawColor,line width= 0.4pt,dash pattern=on 4pt off 4pt ,line join=round,line cap=round] (250.28,163.67) -- (261.50,163.67);

\path[draw=drawColor,line width= 0.4pt,dash pattern=on 4pt off 4pt ,line join=round,line cap=round] (273.50,163.67) -- (284.73,163.67);

\path[draw=drawColor,line width= 0.4pt,line join=round,line cap=round] ( 58.49, 74.75) --
	( 61.52, 69.50) --
	( 55.46, 69.50) --
	cycle;

\path[draw=drawColor,line width= 0.4pt,line join=round,line cap=round] ( 81.71,144.31) --
	( 84.74,139.06) --
	( 78.68,139.06) --
	cycle;

\path[draw=drawColor,line width= 0.4pt,line join=round,line cap=round] (104.94,165.39) --
	(107.97,160.15) --
	(101.91,160.15) --
	cycle;

\path[draw=drawColor,line width= 0.4pt,line join=round,line cap=round] (128.16,167.17) --
	(131.19,161.92) --
	(125.13,161.92) --
	cycle;

\path[draw=drawColor,line width= 0.4pt,line join=round,line cap=round] (151.38,167.17) --
	(154.41,161.92) --
	(148.35,161.92) --
	cycle;

\path[draw=drawColor,line width= 0.4pt,line join=round,line cap=round] (174.61,167.17) --
	(177.64,161.92) --
	(171.58,161.92) --
	cycle;

\path[draw=drawColor,line width= 0.4pt,line join=round,line cap=round] (197.83,167.17) --
	(200.86,161.92) --
	(194.80,161.92) --
	cycle;

\path[draw=drawColor,line width= 0.4pt,line join=round,line cap=round] (221.05,167.17) --
	(224.08,161.92) --
	(218.02,161.92) --
	cycle;

\path[draw=drawColor,line width= 0.4pt,line join=round,line cap=round] (244.28,167.17) --
	(247.31,161.92) --
	(241.25,161.92) --
	cycle;

\path[draw=drawColor,line width= 0.4pt,line join=round,line cap=round] (267.50,167.17) --
	(270.53,161.92) --
	(264.47,161.92) --
	cycle;

\path[draw=drawColor,line width= 0.4pt,line join=round,line cap=round] (290.73,167.17) --
	(293.76,161.92) --
	(287.70,161.92) --
	cycle;
\definecolor{drawColor}{RGB}{0,0,255}

\path[draw=drawColor,line width= 0.4pt,dash pattern=on 1pt off 3pt on 4pt off 3pt ,line join=round,line cap=round] ( 64.30, 71.74) -- ( 75.90, 74.70);

\path[draw=drawColor,line width= 0.4pt,dash pattern=on 1pt off 3pt on 4pt off 3pt ,line join=round,line cap=round] ( 86.38, 79.94) -- (100.27, 91.13);

\path[draw=drawColor,line width= 0.4pt,dash pattern=on 1pt off 3pt on 4pt off 3pt ,line join=round,line cap=round] (108.33, 99.85) -- (124.77,123.84);

\path[draw=drawColor,line width= 0.4pt,dash pattern=on 1pt off 3pt on 4pt off 3pt ,line join=round,line cap=round] (132.38,133.05) -- (147.16,147.98);

\path[draw=drawColor,line width= 0.4pt,dash pattern=on 1pt off 3pt on 4pt off 3pt ,line join=round,line cap=round] (156.91,154.58) -- (169.08,159.75);

\path[draw=drawColor,line width= 0.4pt,dash pattern=on 1pt off 3pt on 4pt off 3pt ,line join=round,line cap=round] (180.60,162.45) -- (191.84,163.12);

\path[draw=drawColor,line width= 0.4pt,dash pattern=on 1pt off 3pt on 4pt off 3pt ,line join=round,line cap=round] (203.83,163.52) -- (215.05,163.62);

\path[draw=drawColor,line width= 0.4pt,dash pattern=on 1pt off 3pt on 4pt off 3pt ,line join=round,line cap=round] (227.05,163.67) -- (238.28,163.67);

\path[draw=drawColor,line width= 0.4pt,dash pattern=on 1pt off 3pt on 4pt off 3pt ,line join=round,line cap=round] (250.28,163.67) -- (261.50,163.67);

\path[draw=drawColor,line width= 0.4pt,dash pattern=on 1pt off 3pt on 4pt off 3pt ,line join=round,line cap=round] (273.50,163.67) -- (284.73,163.67);

\path[draw=drawColor,line width= 0.4pt,line join=round,line cap=round] ( 55.31, 70.26) -- ( 61.67, 70.26);

\path[draw=drawColor,line width= 0.4pt,line join=round,line cap=round] ( 58.49, 67.08) -- ( 58.49, 73.45);

\path[draw=drawColor,line width= 0.4pt,line join=round,line cap=round] ( 78.53, 76.18) -- ( 84.90, 76.18);

\path[draw=drawColor,line width= 0.4pt,line join=round,line cap=round] ( 81.71, 72.99) -- ( 81.71, 79.36);

\path[draw=drawColor,line width= 0.4pt,line join=round,line cap=round] (101.75, 94.90) -- (108.12, 94.90);

\path[draw=drawColor,line width= 0.4pt,line join=round,line cap=round] (104.94, 91.71) -- (104.94, 98.08);

\path[draw=drawColor,line width= 0.4pt,line join=round,line cap=round] (124.98,128.79) -- (131.34,128.79);

\path[draw=drawColor,line width= 0.4pt,line join=round,line cap=round] (128.16,125.61) -- (128.16,131.97);

\path[draw=drawColor,line width= 0.4pt,line join=round,line cap=round] (148.20,152.24) -- (154.57,152.24);

\path[draw=drawColor,line width= 0.4pt,line join=round,line cap=round] (151.38,149.06) -- (151.38,155.42);

\path[draw=drawColor,line width= 0.4pt,line join=round,line cap=round] (171.43,162.09) -- (177.79,162.09);

\path[draw=drawColor,line width= 0.4pt,line join=round,line cap=round] (174.61,158.91) -- (174.61,165.27);

\path[draw=drawColor,line width= 0.4pt,line join=round,line cap=round] (194.65,163.47) -- (201.01,163.47);

\path[draw=drawColor,line width= 0.4pt,line join=round,line cap=round] (197.83,160.29) -- (197.83,166.65);

\path[draw=drawColor,line width= 0.4pt,line join=round,line cap=round] (217.87,163.67) -- (224.24,163.67);

\path[draw=drawColor,line width= 0.4pt,line join=round,line cap=round] (221.05,160.49) -- (221.05,166.85);

\path[draw=drawColor,line width= 0.4pt,line join=round,line cap=round] (241.10,163.67) -- (247.46,163.67);

\path[draw=drawColor,line width= 0.4pt,line join=round,line cap=round] (244.28,160.49) -- (244.28,166.85);

\path[draw=drawColor,line width= 0.4pt,line join=round,line cap=round] (264.32,163.67) -- (270.68,163.67);

\path[draw=drawColor,line width= 0.4pt,line join=round,line cap=round] (267.50,160.49) -- (267.50,166.85);

\path[draw=drawColor,line width= 0.4pt,line join=round,line cap=round] (287.54,163.67) -- (293.91,163.67);

\path[draw=drawColor,line width= 0.4pt,line join=round,line cap=round] (290.73,160.49) -- (290.73,166.85);
\definecolor{drawColor}{RGB}{0,255,0}

\path[draw=drawColor,line width= 0.4pt,dash pattern=on 1pt off 3pt on 4pt off 3pt ,line join=round,line cap=round] ( 61.43, 73.53) -- ( 78.78,104.44);

\path[draw=drawColor,line width= 0.4pt,dash pattern=on 1pt off 3pt on 4pt off 3pt ,line join=round,line cap=round] ( 84.83,114.80) -- (101.82,142.78);

\path[draw=drawColor,line width= 0.4pt,dash pattern=on 1pt off 3pt on 4pt off 3pt ,line join=round,line cap=round] (110.19,150.80) -- (122.91,157.82);

\path[draw=drawColor,line width= 0.4pt,dash pattern=on 1pt off 3pt on 4pt off 3pt ,line join=round,line cap=round] (134.13,161.32) -- (145.41,162.47);

\path[draw=drawColor,line width= 0.4pt,dash pattern=on 1pt off 3pt on 4pt off 3pt ,line join=round,line cap=round] (157.38,163.23) -- (168.61,163.52);

\path[draw=drawColor,line width= 0.4pt,dash pattern=on 1pt off 3pt on 4pt off 3pt ,line join=round,line cap=round] (180.61,163.67) -- (191.83,163.67);

\path[draw=drawColor,line width= 0.4pt,dash pattern=on 1pt off 3pt on 4pt off 3pt ,line join=round,line cap=round] (203.83,163.67) -- (215.05,163.67);

\path[draw=drawColor,line width= 0.4pt,dash pattern=on 1pt off 3pt on 4pt off 3pt ,line join=round,line cap=round] (227.05,163.67) -- (238.28,163.67);

\path[draw=drawColor,line width= 0.4pt,dash pattern=on 1pt off 3pt on 4pt off 3pt ,line join=round,line cap=round] (250.28,163.67) -- (261.50,163.67);

\path[draw=drawColor,line width= 0.4pt,dash pattern=on 1pt off 3pt on 4pt off 3pt ,line join=round,line cap=round] (273.50,163.67) -- (284.73,163.67);

\path[draw=drawColor,line width= 0.4pt,line join=round,line cap=round] ( 56.24, 66.04) -- ( 60.74, 70.54);

\path[draw=drawColor,line width= 0.4pt,line join=round,line cap=round] ( 56.24, 70.54) -- ( 60.74, 66.04);

\path[draw=drawColor,line width= 0.4pt,line join=round,line cap=round] ( 79.46,107.43) -- ( 83.96,111.93);

\path[draw=drawColor,line width= 0.4pt,line join=round,line cap=round] ( 79.46,111.93) -- ( 83.96,107.43);

\path[draw=drawColor,line width= 0.4pt,line join=round,line cap=round] (102.69,145.65) -- (107.19,150.15);

\path[draw=drawColor,line width= 0.4pt,line join=round,line cap=round] (102.69,150.15) -- (107.19,145.65);

\path[draw=drawColor,line width= 0.4pt,line join=round,line cap=round] (125.91,158.46) -- (130.41,162.96);

\path[draw=drawColor,line width= 0.4pt,line join=round,line cap=round] (125.91,162.96) -- (130.41,158.46);

\path[draw=drawColor,line width= 0.4pt,line join=round,line cap=round] (149.13,160.83) -- (153.63,165.33);

\path[draw=drawColor,line width= 0.4pt,line join=round,line cap=round] (149.13,165.33) -- (153.63,160.83);

\path[draw=drawColor,line width= 0.4pt,line join=round,line cap=round] (172.36,161.42) -- (176.86,165.92);

\path[draw=drawColor,line width= 0.4pt,line join=round,line cap=round] (172.36,165.92) -- (176.86,161.42);

\path[draw=drawColor,line width= 0.4pt,line join=round,line cap=round] (195.58,161.42) -- (200.08,165.92);

\path[draw=drawColor,line width= 0.4pt,line join=round,line cap=round] (195.58,165.92) -- (200.08,161.42);

\path[draw=drawColor,line width= 0.4pt,line join=round,line cap=round] (218.80,161.42) -- (223.30,165.92);

\path[draw=drawColor,line width= 0.4pt,line join=round,line cap=round] (218.80,165.92) -- (223.30,161.42);

\path[draw=drawColor,line width= 0.4pt,line join=round,line cap=round] (242.03,161.42) -- (246.53,165.92);

\path[draw=drawColor,line width= 0.4pt,line join=round,line cap=round] (242.03,165.92) -- (246.53,161.42);

\path[draw=drawColor,line width= 0.4pt,line join=round,line cap=round] (265.25,161.42) -- (269.75,165.92);

\path[draw=drawColor,line width= 0.4pt,line join=round,line cap=round] (265.25,165.92) -- (269.75,161.42);

\path[draw=drawColor,line width= 0.4pt,line join=round,line cap=round] (288.48,161.42) -- (292.98,165.92);

\path[draw=drawColor,line width= 0.4pt,line join=round,line cap=round] (288.48,165.92) -- (292.98,161.42);
\definecolor{drawColor}{RGB}{0,0,0}

\path[draw=drawColor,line width= 0.4pt,line join=round,line cap=round] (253.73,121.20) rectangle (300.01, 61.20);

\path[draw=drawColor,line width= 0.4pt,line join=round,line cap=round] (262.73,109.20) circle (  2.25);
\definecolor{drawColor}{RGB}{255,0,0}

\path[draw=drawColor,line width= 0.4pt,line join=round,line cap=round] (262.73,100.70) --
	(265.76, 95.45) --
	(259.70, 95.45) --
	cycle;
\definecolor{drawColor}{RGB}{0,255,0}

\path[draw=drawColor,line width= 0.4pt,line join=round,line cap=round] (259.55, 85.20) -- (265.91, 85.20);

\path[draw=drawColor,line width= 0.4pt,line join=round,line cap=round] (262.73, 82.02) -- (262.73, 88.38);
\definecolor{drawColor}{RGB}{0,0,255}

\path[draw=drawColor,line width= 0.4pt,line join=round,line cap=round] (260.48, 70.95) -- (264.98, 75.45);

\path[draw=drawColor,line width= 0.4pt,line join=round,line cap=round] (260.48, 75.45) -- (264.98, 70.95);
\definecolor{drawColor}{RGB}{0,0,0}

\node[text=drawColor,anchor=base west,inner sep=0pt, outer sep=0pt, scale=  1.00] at (271.73,105.76) {\itshape $\mathcal I_n$};

\node[text=drawColor,anchor=base west,inner sep=0pt, outer sep=0pt, scale=  1.00] at (271.73, 93.76) {\itshape \normalfont $\mathcal I_n^\text{DC}$};

\node[text=drawColor,anchor=base west,inner sep=0pt, outer sep=0pt, scale=  1.00] at (271.73, 81.76) {\itshape \normalfont $\mathcal I_n^\text{BC}$};

\node[text=drawColor,anchor=base west,inner sep=0pt, outer sep=0pt, scale=  1.00] at (271.73, 69.76) {\itshape \normalfont $\mathcal I_n^\text{CvM}$};
\end{scope}
\end{tikzpicture}